\documentclass[final]{siamltex}

\usepackage{xcolor}
\usepackage{amsmath}
\usepackage{amssymb}
\usepackage{mathrsfs}
\usepackage{graphicx}
\usepackage{bm}
\usepackage{ucs}
\usepackage[utf8x]{inputenc}
\usepackage{wrapfig}
\usepackage{color}
\usepackage{float}
\usepackage{epstopdf}
\usepackage[breaklinks,colorlinks=true,linkcolor=blue,citecolor=red,
  backref=page]{hyperref}

\newtheorem{assum}{Assumption}
\newcommand{\eps}{\epsilon}

\newcommand{\bw}{\boldsymbol{w}}

\newcommand{\bI}{\boldsymbol{I}}

\newcommand{\ind}{\,\mbox{d}}
\newcommand{\bk}{{\boldsymbol{k}}}

\newcommand{\bmu}{{\boldsymbol{\mu}}}
\newcommand{\brho}{{\boldsymbol{\rho}}}
\newcommand{\bchi}{{\boldsymbol{\chi}}}
\newcommand{\boldeta}{{\boldsymbol{\eta}}}
\newcommand{\balpha}{{\boldsymbol{\alpha}}}

\newcommand{\bx}{\boldsymbol{x}}
\newcommand{\by}{\boldsymbol{y}}
\newcommand{\br}{\boldsymbol{r}}
\newcommand{\bz}{\boldsymbol{z}}
\newcommand{\bn}{\boldsymbol{n}}
\newcommand{\bj}{\boldsymbol{j}}

\newcommand{\hh}{\hspace*{0.7pt}}
\newcommand{\ww}{w}

\newcommand{\es}{\hspace*{0.7pt}}
\newcommand{\nes}{\hspace*{-0.65pt}}
\newcommand{\wtx}{\tilde{\bx}}
\newcommand{\wtu}{\tilde{u}}

\newcommand{\up}{u^{\mbox{\tiny ($p$)}}}
\newcommand{\Up}{U^{\mbox{\tiny ($p$)}}}

\newtheorem{remark}{Remark}

\begin{document}

\title{A convergent low-wavenumber, high-frequency homogenization of the wave equation in periodic media with a source term}

\author{Shixu Meng\footnotemark[1], Othman Oudghiri-Idrissi \footnotemark[2], Bojan B. Guzina\footnotemark[3]}
\renewcommand{\thefootnote}{\fnsymbol{footnote}}
\footnotetext[1]{Institute of Applied Mathematics, Academy of Mathematics and Systems Science, Chinese Academy of Sciences, Beijing, 100190, China.  {\tt shixumeng@amss.ac.cn}}
\footnotetext[2]{Department of Civil, Environmental, \& Geo- Engineering, University of Minnesota, Twin Cities,
 Minneapolis, MN 55405, USA.  {\tt oudgh001@umn.edu}}
\footnotetext[3]{Department of Civil, Environmental, \& Geo- Engineering, University of Minnesota, Twin Cities,
 Minneapolis, MN 55405, USA.  {\tt guzin001@umn.edu}}

\maketitle

\begin{abstract}
We pursue a low-wavenumber, second-order homogenized solution of the time-harmonic wave equation { at both low and high frequency} in periodic media with a source term whose frequency resides inside a band gap. Considering the wave motion in an unbounded medium $\mathbb{R}^d$ ($d\geqslant1$), we first use the (Floquet-)Bloch transform to formulate an equivalent variational problem in a bounded domain. By investigating the source term's projection onto certain periodic functions, the second-order model can then be derived via asymptotic expansion of the Bloch eigenfunction and the germane dispersion relationship. We establish the convergence of the second-order homogenized solution, and we include numerical examples to illustrate the convergence result.
\end{abstract}

\section{Introduction} \label{Introduction}
\noindent {Let ${\boldsymbol{I}}:=\{\bz: \bz \in \mathbb{Z}^d \}$}, and consider the time-harmonic wave equation
\begin{equation}\label{PDE fast}
-\nabla\!\cdot\!\big(G(\br/\eps)\nabla {  U_{\epsilon} }\big) - { \Omega_{\epsilon}}^2\rho(\br/\eps)\hh {  U_{\epsilon} } \:=\: { f_{\epsilon}(\br)} \qquad\text{in~~}\mathbb{R}^d, ~ d\geqslant 1
\end{equation}
\noindent  {with highly oscillating coefficients $(\eps\!\to\!0)$} at frequency~${ \Omega_{\epsilon}}$,
where~$G$ and~$\rho$ are $\boldsymbol{I}$-periodic i.e.  
\begin{eqnarray} \label{Grho}
G(\br+   \bz) = G(\br), \quad \rho(\br+  \bz) = \rho(\br), \qquad \forall~\br \in \mathbb{R}^d, ~ \bz \in \mathbb{Z}^d.
\end{eqnarray}
Here~$f_{\epsilon}$ denotes the source term, while $G$ and~$\rho$ are assumed to be real-valued, sufficiently smooth functions bounded away from zero. The regularity on $\rho$ and $G$ may be relaxed and we comment on this later. When~$d\!=\!2$, one may interpret~(\ref{PDE fast}) in the context of anti-plane shear (elastic) waves, in which case~${ U_\epsilon},G,\rho$ and~$f_{\epsilon}$ take respectively the roles of transverse displacement, shear modulus, mass density, and body force. 

We seek a low-wavenumber, second-order asymptotic solution~$U_2$ of~\eqref{PDE fast}, formulated in terms of the perturbation parameter $\eps$, such that
\begin{equation}\label{aux1}
\|{  U_{\epsilon} }-U_2\|_{L^2(\mathbb{R}^d)} = O(\eps^3) \quad \text{as } \eps\to 0. 
\end{equation}
More precisely, we are interested in the respective field equations that $U_2$ and its ``mean'' i.e. effective counterpart satisfy. {  For completeness, we consider both low and high frequency cases. Both the driven frequency $\Omega_{\epsilon}$ and source term $f_{\epsilon}$ may depend on $\epsilon$. As will be seen shortly, in the low (resp. high) frequency case ${ \Omega^2_{\epsilon}}$ is $O(1)$  (resp. $O(\eps^{-2})$) and it is a perturbation of the (scaled) Bloch eigenvalue;  ${ f_{\epsilon}}$ is (or more precisely, the multiplication of a certain fast eigenfunction and) the Fourier transform of a compactly-supported function. We specify the  explicit forms of the driven frequency $\Omega_{\epsilon}$ and the source term $f_{\epsilon}$ later by equations \eqref{bg3} and \eqref{source1} respectively. Furthermore we consider the ($\epsilon$ scaled) driven frequency belongs to the band gap (whereby no radiation condition is required in our context), and we refer to~\cite{kuchment2012green,kha2015green} for similar problems near internal edges of the spectra of periodic elliptic operators.}

\subsection{Background and motivation}

Wave motion in periodic and otherwise microstructured media \cite{lions2011asymptotic,JJWM,kuchment2016overview,jikov1994} has keen applications in science, engineering, and technology owing to the emergence of metamaterials facilitating the phenomena such as cloaking, sub-wavelength imaging, and vibration control~\cite{Capo2009,Bava2013}. To simulate the underpinning physical processes effectively, of particular interest is the development of asymptotic models. In this vein, homogenized (i.e. effective) models of waves in periodic media that transcend the quasi-static limit have attracted much attention. In particular, effective models using the concept of Bloch waves were studied in \cite{santosa1991dispersive, conca2002bloch,sjoberg2005floquet,doi:10.1137/040607034,dohnal2014bloch}, while \cite{chen2001dispersive,wautier2015second} investigated the dispersive wave motion via multi-scale homogenization. We also mention related work on homogenization \cite{kristensson2003homogenization,bouchitte2010homogenization,wellander2009two}. A formal link between the homogenization using Bloch waves and two-scale homogenization, {restricted to the periodic index of refraction (no periodicity in the principal part)}, was established in~\cite{allaire2016comparison}, including an account for the \textit{source term}; the study suggests that the source term  must undergo asymptotic correction in order for the equivalence between the two methods to hold. On the other hand, Willis' approach to dynamic homogenization has brought much attention in the engineering community as a means to deal with periodic and random composites, see for instance \cite{milton2007modifications,willis2011effective,norris2012analytical,nassar2015willis}. Recently, \cite{meng2018dynamic} investigated a second-order asymptotic expansion of the Willis' model, which demonstrates that the source term must be homogenized when considering the two-scale, second-order homogenization of the wave equation. We refer the reader to \cite{cakoni2016homogenization,cakoni2019scattering,lin2018leading} for related works on the scattering by bounded periodic structures.

However, the foregoing studies primarily focused on the low-frequency wave motion, either in the frequency- or time-domain. Recent advances in metamaterials have motivated the studies on high- (or finite-) frequency homogenization \cite{craster2010high,harutyunyan2016high,zhang2010multiscale}. Within the framework of {Bloch-wave homogenization}, \cite{guzina2019rational} pursued a (second-order) {finite-frequency, finite-wavelength} effective description of the wave equation with a source term. To our knowledge, however, there are no convergence results on the {high-frequency homogenization of wave motion in periodic media} when the \textit{source term} is present. This motives our work, which  considers the situations when the (high or low) frequency resides inside a band gap. For completeness, we note that this subject is also relevant to the asymptotic behavior of the Green's function near internal edges of the spectra of periodic elliptic operators~\cite{kuchment2012green,kha2015green}. {In our investigation, we apply various mapping properties of the (Floquet)-Bloch transform~\cite{lechleiter2017floquet} to formulate the wave motion in~$\mathbb{R}^d$ as a variational problem in the Wigner-Seitz cell. } {  We emphasize that our work is different from \cite{guzina2019rational}, because it provides a mathematically rigorous proof of the second-order approximation at both low and high frequency, and the underlying approach may play a key role in our ability to extend the current work to asymptotic models of wave motion across periodic surfaces and interfaces of periodic media.}

The paper is organized as follows. In Section \ref{Prelim}, we provide the necessary background on the (Floquet)-Bloch transform and wave dispersion in periodic media. Section \ref{Bloch expansion} formulates the {wave motion in~$\mathbb{R}^d$ with highly oscillating periodic coefficients (due to a source term) as a variational problem in a bounded domain}, and gives both variational and integral representations of the germane wavefield. Here we also discuss the class of source terms under consideration, and we establish the key properties of this class when projected onto periodic functions. We then show in  Section~\ref{Significant contri} that the principal contribution to the second-order homogenization arises {from the nearest ($p$th) branch of the dispersion relationship}. We further give the asymptotics of latter and the corresponding eigenfunction in Section~\ref{Asymptotic eigenfunction dispersion relation}. These results, together with the  source term, are  used to obtain the sought convergence result on the second-order, high-frequency homogenization in Section~\ref{Higher order U}. The analytical results are illustrated in Section \ref{Numerical example} by way of numerical simulations, performed for an example periodic structure. 

\section{(Floquet-)Bloch transform and dispersion relationship} \label{Prelim}

In this section, we provide preliminary background on the (Floquet-)Bloch transform and the dispersion relationship characterizing the periodic medium featured in~\eqref{PDE fast}.

\subsection{(Floquet-)Bloch transform in $\mathbb{R}^d$}

The key elements of (Floquet-)Bloch transform introduced in the sequel are adapted from \cite{lechleiter2017floquet} (see also \cite{kuchment2016overview}). To begin with, we introduce the Wigner-Seitz cell 
\begin{eqnarray} \label{wgcell}
W_{\!\boldsymbol{I}}:=\{   \tilde{\bx}: \tilde{\bx} \in \mathbb{R}^d, \, -1/2 \le \tilde{\bx}_{1,\cdots,d} \le 1/2 \} \subset \mathbb{R}^d.
\end{eqnarray}
{  This} gives the dual or reciprocal lattice ${\boldsymbol{I}}^*\!:=\{ 2\pi \bj: \bj \in \mathbb{Z}^d \}$ and the Brillouin zone
\begin{eqnarray} \label{Brill}
W_{\!{\boldsymbol{I}}^*}:=\{ 2\pi  {\bk}: {\bk} \in \mathbb{R}^d, \, -1/2 \le {\bk}_{1,\cdots,d} \le 1/2 \} \subset \mathbb{R}^d.
\end{eqnarray}
With reference to the periodicity { ${\boldsymbol{I}}=\{\bz: \bz \in \mathbb{Z}^d \}$} and arbitrary wavenumber $\bk\!\in\!\mathbb{R}^d$, any function $\phi(\bx)$ that satisfies  
\begin{eqnarray*}
\phi(\bx+  \bj) \,=\, e^{i \bk \cdot  \bj} \es \phi(\bx), \quad \forall\bx \in \mathbb{R}^d, ~ \bj\in\mathbb{Z}^d,
\end{eqnarray*}
is called $\bk$-quasiperiodic { with respect to $\boldsymbol{I}$}. In passing, we note that any $\boldsymbol{I}$-periodic function becomes $\bk$-quasiperiodic upon multiplication by $e^{i\bk\cdot \bx}$. Note that if a function is $\bk$-quasiperiodic with respect to $\boldsymbol{I}$, then its values (in the variable $\bx$) at $W_{\!{\boldsymbol{I}}}$ fully determines its values at $\mathbb{R}^d$. 

With the above definitions, the (Floquet-)Bloch transform $\mathcal{J}_{\mathbb{R}^d}$ of a function $\psi\in C_0^\infty(\mathbb{R}^d)$ is given by
 \begin{eqnarray} \label{Def Bloch transform}
{ (\mathcal{J}_{\mathbb{R}^d} \psi)} (\bk;\bx) = \frac{1}{(2\pi)^{d/2}} \sum_{\bj \in \mathbb{Z}^d} \psi(\bx+  \bj) e^{-i \bk \cdot  \bj}, \quad \bk {  \,\in W_{\!{\boldsymbol{I}}^*}},\,\bx \in \mathbb{R}^d.
\end{eqnarray}
{ Since $\big(\mathcal{J}_{\mathbb{R}^d} \psi \big)(\bk;\bx)$ is $\bk$-quasiperiodic in the variable $\bx$, the knowledge of $\big(\mathcal{J}_{\mathbb{R}^d} \psi \big)(\bk;\bx)$ at $\bx \in W_{\!\boldsymbol{I}}$ fully determines all its values at $\bx \in \mathbb{R}^d$.} {   When restricting a given function in the domain $W_{\!\boldsymbol{I}}$,  we write variables $\tilde{\bx}$ instead of $\bx$ for the best readability and associate this given function with the tilde symbol~~$\tilde{}$~~as well.}

It is readily seen that $\mathcal{J}_{\mathbb{R}^d}$ commutes with $\boldsymbol{I}$-periodic functions; namely {  if} $q$ is $\boldsymbol{I}$-periodic, then
 \begin{equation} \label{Periodic q commute}
{ \big(\mathcal{J}_{\mathbb{R}^d} (q\psi) \big)}(\bk;\bx) = \frac{1}{(2\pi)^{d/2}} \sum_{\bj \in \mathbb{Z}^d}q(\bx+   \bj)  \psi(\bx+   \bj) e^{-i \bk \cdot   \bj} = q(\bx)(\mathcal{J}_{\mathbb{R}^d} \psi )(\bk;\bx).
\end{equation}

To facilitate the ensuing analysis, we introduce several function spaces. We first recall the Fourier transform 
\begin{eqnarray}
(\mathcal{F} \psi) (\bz) := \frac{1}{(2\pi)^{d/2}} \int_{\mathbb{R}^d} e^{-i \bz \cdot \bx} \psi(\bx) \ind \bx
\end{eqnarray}
for $\psi \in C_0^\infty(\mathbb{R}^d)$ and $\bz \in \mathbb{R}^d$. This transform extends to an isometry on $L^2(\mathbb{R}^d)$ and defines Bessel potential spaces by
\begin{eqnarray*}
\mathcal{H}^s(\mathbb{R}^d): = \Big\{ \psi \in \mathcal{D}'(\mathbb{R}^d): \,\int_{\mathbb{R}^d} (1+|\bz|^2)^s \es |(\mathcal{F} \psi) (\bz) |^2 \ind \bz < \infty \Big\}, \quad s \in \mathbb{R}, 
\end{eqnarray*}
{where $\mathcal{D}'(\mathbb{R}^d)$ denotes the space of distributions in $\mathbb{R}^d$.}

We next introduce (quasi-)periodic Sobolev spaces. Specifically, we denote by $\mathcal{H}^s_{\bk}(W_{\!\boldsymbol{I}})$ the Hilbert space containing all $\bk$-quasiperiodic distributions (which contains the products of all periodic distributions with $e^{i \bk \cdot \bx}$ \cite{saranen2013periodic}) with finite norm
\begin{eqnarray}\label{Hsk Fourier norm}
\| \tilde{\psi} \|_{\mathcal{H}^s_{\bk}(W_{\!\boldsymbol{I}})}:=\bigg( \sum_{\bj \in \mathbb{Z}^d} (1+|\bj|^2)^s |{ c_{\bj,\bk}}|^2 \bigg)^{1/2} < \infty, \quad s \in \mathbb{R},
\end{eqnarray}
where $
{ c_{\bj,\bk}} := \int_{W_{\!\boldsymbol{I}}} \tilde{\psi}(\tilde{\bx}) e^{-i \bk \cdot \tilde{\bx}} \es \overline{e^{i 2\pi \bj \cdot \tilde{\bx}}} ~\ind \tilde{\bx}$. {  We recall that the tilde symbol~~$\tilde{}$~~has been associated with the distribution and its variable (defined in $W_{\!\boldsymbol{I}}$).}

Now we are ready to introduce adapted function spaces in $(\bk;\tilde{\bx})$. We denote by $L^2(W_{\!{\boldsymbol{I}}^*};\mathcal{H}^s_{\bk} (W_{\!\boldsymbol{I}}))$  the Sobolev space containing  distributions in $\mathcal{D}'(\mathbb{R}^d \!\times\! \mathbb{R}^d)$ that are (i) $2\pi\boldsymbol{I}$-periodic in the first variable; (ii) first-variable-quasiperiodic with respect to $\boldsymbol{I}$ in the second variable, and (iii) have a finite  norm
\begin{eqnarray} 
\| \tilde{\psi} \|_{L^2(W_{\!{\boldsymbol{I}}^*};\mathcal{H}^s_{\bk} (W_{\!\boldsymbol{I}}))}&:=& \bigg(\sum_{\bj \in \mathbb{Z}^d} (1+|\bj|^2)^s  \int_{W_{\!{\boldsymbol{I}}^*}} |\hat{\psi}_{\boldsymbol{I}}(\bk,\bj)|^2 \ind {\bk} \bigg)^{1/2} \label{L2Hsk Fourier norm}\\
&=& \bigg( \int_{W_{\!{\boldsymbol{I}}^*}} \|\tilde{\psi}(\bk;\cdot)\|^2_{\mathcal{H}^s_{\bk} (W_{\!\boldsymbol{I}})} \ind {\bk} \bigg)^{1/2}  < \infty, \nonumber 
\end{eqnarray}
where $\hat{\psi}_{\boldsymbol{I}}(\bk,\bj) = \int_{W_{\!{\boldsymbol{I}}}} \tilde{\psi}(\bk;\tilde{\bx}) e^{-i \bk \cdot \tilde{\bx}} \es \overline{e^{i 2\pi \bj \cdot \tilde{\bx}}} \ind \tilde{\bx}$. 

With the above definitions in place, we can state the following lemma \cite[Theorem 4, Lemma 7]{lechleiter2017floquet}.

\begin{lemma} \label{Bloch isomorphism}
\begin{itemize}
\item The (Floquet-)Bloch transform $\mathcal{J}_{\mathbb{R}^d}$ given by \eqref{Def Bloch transform} extends from $C_0^\infty(\mathbb{R}^d)$ to an isometric isomorphism between $L^2(\mathbb{R}^d)$ and $L^2(W_{\!{\boldsymbol{I}}^*};L^2(W_{\!\boldsymbol{I}}))$ with inverse
\begin{equation} \label{Def inverse Bloch transform}
\big(\mathcal{J}^{-1}_{\mathbb{R}^d} \tilde{\psi}   \big) (\bx): =  \frac{1}{(2\pi)^{d/2}} \int_{W_{\!{\boldsymbol{I}}^*}} \tilde{\psi}(\bk; \bx-\bj) e^{i \bk \cdot  \bj} ~\ind \bk,
\end{equation}
where $\bj \in \mathbb{Z}^d$ is the translation such that $\bx-  \bj \in W_{{\boldsymbol{I}}}$. Further for $s \in \mathbb{R}$, the (Floquet-)Bloch transform $\mathcal{J}_{\mathbb{R}^d}$ extends from $C_0^\infty(\mathbb{R}^d)$  to an isomorphism between $\mathcal{H}^s(\mathbb{R}^d)$ and $L^2(W_{\!{\boldsymbol{I}}^*}; \mathcal{H}^s_{\bk}(W_{\!\boldsymbol{I}}))$.

\item For $m \in \mathbb{N}$ and $u \in \mathcal{H}^m(\mathbb{R}^d)$, the (Floquet-)Bloch transform $\mathcal{J}_{\mathbb{R}^d} u(\bk; \tilde{\bx})$ processes weak partial derivatives with respect to $\tilde{\bx} \in W_{\!\boldsymbol{I}}$ in $L^2(W_{\!\boldsymbol{I}})$ up to order $m \in \mathbb{N}$. If $\balpha \in \mathbb{Z}^d$, $\balpha\ge \boldsymbol{0}$ with $|\balpha|\le m$, then 
\begin{eqnarray}  \label{Bloch inside derivative}
\partial^\balpha_{\tilde{\bx}} (\mathcal{J}_{\mathbb{R}^d} u) (\bk; \tilde{\bx}) = \mathcal{J}_{\mathbb{R}^d} [ \partial^\balpha_{\bx} u ] (\bk; \tilde{\bx}).
\end{eqnarray}
\end{itemize}
\end{lemma}
\subsection{Orthonormal basis and dispersion relationship}
For any $\bk \in W_{\!\boldsymbol{I}^*}$, let us introduce the second-order elliptic operator on $W_{\!\boldsymbol{I}}$
\begin{eqnarray*}
\mathcal{A}(\bk)  =-\frac{1}{\rho}\big(\nabla+i \bk ) \!\cdot\!\big[G (\nabla + i\bk) \big].
\end{eqnarray*}
From the theory of compact self-adjoint operators \cite[pp 341]{lions2011asymptotic} and \cite{wilcox1978theory} { (under the assumption that $G$ and $\rho$ are smooth)}, there exists a discrete set of eigenvalues $\{\omega^2_m(\bk)\}_{m=0}^\infty$ and a set of eigenfunctions { (this set of eigenfunctions is not arbitrarily chosen, it is chosen according to \cite{wilcox1978theory})} $\{\phi_m(\bk;\cdot)\}_{m=0}^\infty$ such that
\begin{eqnarray} \label{eigenfunction}
\mathcal{A}(\bk) \phi_m(\bk;\cdot) = \omega^2_m(\bk)   \phi_m(\bk;\cdot),
\end{eqnarray}
where $\phi_m(\bk;\cdot) \in \mathcal{H}^1_{\boldsymbol{0}}(W_{\!\boldsymbol{I}})$;
\begin{eqnarray*}
\langle \rho \phi_m(\bk;\cdot),\phi_n(\bk;\cdot)  \rangle := \int_{ W_{\!\boldsymbol{I}}} \rho(\tilde{\bx}) \phi_m(\bk; \tilde{\bx}) \es \overline{\phi_n(\bk; \tilde{\bx})} \es \ind \tilde{\bx} \:=\: \delta_{mn},
\end{eqnarray*}
 and $\{\phi_m(\bk;\tilde{\bx})\}_{m=0}^\infty$ is a complete orthonormal basis in $L^2(W_{\!\boldsymbol{I}})$. Hereon, we use the notation $\langle \cdot \rangle$ to denote the average of a (possibly tensor-valued) function over $W_{\!\boldsymbol{I}}$. {  Moreover each eigenvalue and each eigenfunction are holomorphic in $\bk$ except a null set where the multiplicity of the eigenvalue change \cite{wilcox1978theory,lions2011asymptotic}.}
 
{  It is then verified for all $\bj \in\mathbb{Z}^d$ that
\begin{eqnarray*}
\mathcal{A}(\bk+2\pi \bj) \phi_m(\bk+2\pi \bj;\cdot) = \omega^2_m(\bk+2\pi \bj)   \phi_m(\bk+2\pi \bj;\cdot),
\end{eqnarray*}
with $\omega_m(\bk+2\pi \bj) := \omega_m(\bk)$, and 
 \begin{eqnarray*}
\phi_m(\bk + 2\pi \bj;\tilde{\bx}) \,:=\, e^{i (-\tilde{\bx}) \cdot  2\pi\bj} \es \phi_m(\bk;\tilde{\bx}), \quad \forall \tilde{\bx} \in W_{\!\boldsymbol{I}}, ~ \bj \in\mathbb{Z}^d.
\end{eqnarray*}
It follows immediately that $e^{i\bk \cdot \bx}\phi_m(\bk;\bx)$ is (i) $2\pi\boldsymbol{I}$-periodic in the $\bk$ variable; (ii) $\bk$-quasiperiodic with respect to $\boldsymbol{I}$ in the $\bx$ variable. 
}

By analogy to~\eqref{Hsk Fourier norm} we find that for all $\bk \!\in\! W_{\!\boldsymbol{I}^*}$, an equivalent norm for $\mathcal{H}^s_{\bk} (W_{\!\boldsymbol{I}})$ can be defined as
\begin{eqnarray}
\|\nes| \tilde{\psi} |\nes\|_{\mathcal{H}^s_{\bk} (W_{\!\boldsymbol{I}})}&:=& \bigg( \sum_{m=0}^\infty (1+\omega_m^2)^s \es |\langle \rho  \tilde{\psi} , e^{i\bk \cdot } \phi_m  \rangle|^2 \bigg)^{1/2}  \\
&=& \Big( \| \tilde{\psi} e^{-i\bk \cdot } \|^2_{L^2(W_{\!\boldsymbol{I}}; \rho)}+  \delta_{s1}\| (\nabla+i\bk) (\tilde{\psi} e^{-i\bk  \cdot}) \|^2_{L^2(W_{\!\boldsymbol{I}}; G)} \Big)^{1/2}  \nonumber \\
&=& \Big( \| \tilde{\psi}  \|^2_{L^2(W_{\!\boldsymbol{I}}; \rho)} +  \delta_{s1}\| \nabla \tilde{\psi}   \|^2_{L^2(W_{\!\boldsymbol{I}}; G)} \Big)^{1/2} < \infty, \quad s=0,1, \nonumber
\end{eqnarray}
where $L^2(W_{\!\boldsymbol{I}};\rho)$ denotes the $\rho$-weighted $L^2$-space in $W_{\!\boldsymbol{I}}$ with norm $\big(\int_{W_{\!\boldsymbol{I}}} \rho |\cdot|^2 \big)^{\frac{1}{2}}$, and the same notation holds for $L^2(W_{\!\boldsymbol{I}};G)$. As a result, $L^2(W_{\!{\boldsymbol{I}^*}};\mathcal{H}^s_{\bk} (W_{\!\boldsymbol{I}}))$ (with $s=0,1$)  has equivalent norm given by
\begin{eqnarray} 
\|\nes| \tilde{\psi} |\nes\|_{L^2(W_{\!{\boldsymbol{I}^*}};\mathcal{H}^s_{\bk} (W_{\!\boldsymbol{I}}))}&:=& \bigg( \sum_{m=0}^\infty (1+\omega_m^2)^s \int_{W_{\!{\boldsymbol{I}^*}}} |\langle \rho  \tilde{\psi}(\bk;\cdot), e^{i\bk  \cdot} \phi_m(\bk;\cdot) \rangle|^2 \ind \bk \bigg)^{1/2} ~~~ \label{Convergence Bloch modes}\\
&=& \bigg( \int_{W_{\!{\boldsymbol{I}^*}}} \|\nes|\tilde{\psi}(\bk;\cdot)|\nes\|^2_{\mathcal{H}^s_{\bk} (W_{\!\boldsymbol{I}})} \ind {\bk} \bigg)^{1/2}  < \infty, \quad s=0,1. \nonumber
\end{eqnarray}
The above $\|\nes|\cdot |\nes\|$-norm is equivalent to that given by~\eqref{L2Hsk Fourier norm} and allows us to write down the series expansion of $\psi \in L^2(W_{\!{\boldsymbol{I}^*}};\mathcal{H}^s_{\bk} (W_{\!\boldsymbol{I}}))$ as  
\begin{eqnarray} \label{Expansion Bloch modes}
\tilde{\psi}(\bk; \tilde{\bx})   = \sum_{m=0}^\infty  \langle \rho  \tilde{\psi}(\bk;\cdot), e^{i\bk  \cdot} \phi_m(\bk;\cdot) \rangle e^{i \bk \cdot \tilde{\bx}}  \phi_m(\bk;\tilde{\bx}),
\end{eqnarray}
where the convergence is in the $\|\nes|\cdot |\nes\|$-norm.

For future reference we introduce the real Bloch variety, i.e. the dispersion relationship \cite{kuchment2016overview}, as the set 
\begin{eqnarray*}
\cup_{m=0}^\infty\{ (\bk, \omega_m(\bk)); \, \bk \in W_{\!{\boldsymbol{I}^*}} \},
\end{eqnarray*}
whose $p$th branch is given by 
\begin{eqnarray*}
\{ (\bk, \omega_p(\bk)); \, \bk \in W_{\!{\boldsymbol{I}^*}} \}, \quad p=0,1,2,\ldots.
\end{eqnarray*}
In this setting, {the union of all band gaps can be conveniently written as}  
\begin{eqnarray} \label{bg1}
{\{\omega^2 \!\in \mathbb{R}\!:\,  \omega^2 \neq \omega_m^2(\bk),~  \bk  \in W_{\!\boldsymbol{I}^*}, ~m=0,1,2,\ldots\}.}
\end{eqnarray}

Finally we state the Bloch expansion theorem \cite{lions2011asymptotic} (see also \cite{wilcox1978theory}), which is a direct consequence of the representation \eqref{Expansion Bloch modes} and Lemma \ref{Bloch isomorphism}.

\begin{lemma} \label{Lemma Lemma Bloch expansion}
Let $\psi \in \mathcal{H}^s(\mathbb{R}^d)$ with $s=0,1$. Then
\begin{eqnarray}
\psi(\bx) &=&  \frac{1}{(2\pi)^{d/2}}  \int_{W_{\!{\boldsymbol{I}^*}}} \sum_{m=0}^\infty 
\langle \rho [\mathcal{J}_{\mathbb{R}^d} \psi](\bk;\cdot), e^{i\bk \cdot  \cdot} \phi_m(\bk;\cdot) \rangle
e^{i\bk\cdot   \bx} \phi_m(\bk;\bx) \ind \bk, \label{Lemma Bloch expansion}
\end{eqnarray}
and the Parseval's identity holds
\begin{eqnarray}
\|\psi\|_{\mathcal{H}^s(\mathbb{R}^d)}^2 = \| \mathcal{J}_{\mathbb{R}^d} \psi \|^2_{L^2(W_{\!{\boldsymbol{I}^*}};\mathcal{H}^s_{\bk} (W_{\!\boldsymbol{I}}))} \approx \|\nes| \mathcal{J}_{\mathbb{R}^d} \psi |\nes\|^2_{L^2(W_{\!{\boldsymbol{I}^*}};\mathcal{H}^s_{\bk} (W_{\!\boldsymbol{I}}))}, \label{Parseval's identity}
\end{eqnarray}
where $``\approx"$ means that there exists positive constants $c_1$ and $c_2$ such that 
\begin{eqnarray*}
c_1  \|\nes| \mathcal{J}_{\mathbb{R}^d} \psi |\nes\|^2_{L^2(W_{\!{\boldsymbol{I}^*}};\mathcal{H}^s_{\bk} (W_{\!\boldsymbol{I}}))} \leqslant \|\psi\|_{\mathcal{H}^s(\mathbb{R}^d)}^2 \leqslant c_2 \|\nes| \mathcal{J}_{\mathbb{R}^d} \psi |\nes\|^2_{L^2(W_{\!{\boldsymbol{I}^*}};\mathcal{H}^s_{\bk} (W_{\!\boldsymbol{I}}))}.
\end{eqnarray*}
\end{lemma}

\section{Formulation of the problem and Bloch-wave solution} \label{Bloch expansion}

{As examined earlier we pursue a low-wavenumber, second-order homogenized solution~$U_2$ of~\eqref{PDE fast} that, when ${ \Omega_{\epsilon}}$ resides inside a band gap, converges to~${  U_{\epsilon} }$ according to~\eqref{aux1}. Since~\eqref{bg1} specifies the union of all band gaps for an $\bI$-periodic medium, from the scaling argument we find that for an $\eps\bI$-periodic domain featured in~\eqref{PDE fast}, we have 
\begin{eqnarray} \label{bg2}
{ \Omega_{\epsilon}}^2 \neq \eps^{-2}\omega_m^2(\bk), \quad \bk\in W_{\!\boldsymbol{I}^*}, ~ m=0,1,2,\ldots
\end{eqnarray} 
as an explicit condition that~${ \Omega_{\epsilon}}$ resides inside a band gap { (whereby no radiation condition is required in our context)}. In the context of the low-wavenumber assumption, we further restrict the analysis by letting  
\begin{eqnarray} \label{bg3}
{ \Omega_{\epsilon}}^2 = \eps^{-2}\omega^2_p(\boldsymbol{0}) + \sigma\hat{\Omega}^2, \quad \sigma=\pm 1
\end{eqnarray} 
where (i) $\omega^2_p(\boldsymbol{0})$ is a simple eigenvalue; (ii) $\hat{\Omega}=O(1)$, and (iii)  $\sigma$ is either $1$ or $-1$ depending on~\eqref{bg2}. In this setting, we can clearly distinguish between the low frequency case ($p=0$, $\omega_0(\boldsymbol{0})=0$, ${ \Omega_{\epsilon}}=O(1)$) and its high-frequency counterpart ($p\geqslant 1$, ${ \Omega^2_{\epsilon}}=O(\eps^{-2})$).} { Once again, we consider the case when the driven frequency   is close to the edge of the band gap, and we refer to~\cite{kuchment2012green,kha2015green} for similar problems and applications near internal edges of the spectra of periodic elliptic operators.}
 
Using the change of variables 
\begin{equation}\label{aux2}
\bx=\eps^{-1}\br, \quad \omega=\eps \es { \Omega_{\epsilon}}, \quad u(\bx)={  U_{\epsilon} }(\br)
\end{equation}
we can conveniently rewrite~\eqref{PDE fast} as 
\begin{equation}\label{PDE}
-\nabla\!\cdot\!\big(G(\bx)\nabla u \big) - \omega^2\rho(\bx)\hh u ~=~ \eps^2  { f_{\epsilon}}(\eps \bx) \qquad\text{in~~}\mathbb{R}^d, 
\end{equation}
where 
\begin{equation} \label{drf1}
\omega^2 = \omega_p^2(\boldsymbol{0}) + \eps^2 \sigma\es \hat\Omega^2, \quad \hat{\Omega}=O(1)
\end{equation} 
due to~\eqref{bg3}. 
In what follows, we first use the variational formulation of~\eqref{PDE} to obtain $u(\bx)$, and then recover ${  U_{\epsilon} }(\br)$ via~\eqref{aux2}. { For the best presentation and without the danger of confusion, we write $u$ and $\omega$ and simply remind their dependence on $\epsilon$ throughout the context.}

\subsection{Variational formulation}
We begin with solutions $u\in C_0^\infty(\mathbb{R}^d)$. This allows us to apply the (Floquet-)Bloch transform to~\eqref{PDE}, resulting in 
\begin{equation*} 
\mathcal{J}_{\mathbb{R}^d} \big[ -\nabla\!\cdot\!\big(G\nabla u\big) - \omega^2\rho\hh u \big]  (\bk; \wtx)~=~ \eps^2  \mathcal{J}_{\mathbb{R}^d} \big[{ f_{\epsilon}}(\eps \cdot) \big]  (\bk; \wtx) \qquad\text{in~~} W_{\!\boldsymbol{I}}.
\end{equation*}
With the aid of~\eqref{Periodic q commute} and~\eqref{Bloch inside derivative}, we obtain 
\begin{equation}  \label{WLambda distributional PDE}
 -\nabla\!\cdot\!\big(G(\wtx)\nabla [\mathcal{J}_{\mathbb{R}^d} u   (\bk; \wtx)]\big) - \omega^2\rho(\wtx)\hh [\mathcal{J}_{\mathbb{R}^d} u   (\bk; \wtx)]~=~ \eps^2  \mathcal{J}_{\mathbb{R}^d} \big[{ f_{\epsilon}}(\eps \cdot) \big] (\bk; \wtx)  \quad \text{in } W_{\!\boldsymbol{I}}.
\end{equation}
Let $\psi(\bk;\wtx) \in L^2(W_{\!{\boldsymbol{I}^*}};\mathcal{H}^1_{\bk} (W_{\!\boldsymbol{I}}))$; then multiplying the above equation by $\overline{\psi}$, integrating with respect to $\wtx$ over $W_{\!\boldsymbol{I}}$, and integrating by parts we obtain
\begin{multline} \label{variational Bloch 1}
\int_{W_{\!\boldsymbol{I}}}G(\wtx)\nabla_{\wtx} [\mathcal{J}_{\mathbb{R}^d} u   (\bk; \wtx)] \overline{\nabla_{\wtx} \psi(\bk;\wtx)} \ind \wtx  - \omega^2 \int_{W_{\!\boldsymbol{I}}}\rho(\wtx)\hh [\mathcal{J}_{\mathbb{R}^d} u   (\bk; \wtx)] \overline{  \psi(\bk;\wtx)} \ind \wtx \\
~=~ \eps^2  \int_{W_{\!\boldsymbol{I}}} \mathcal{J}_{\mathbb{R}^d} \big[{ f_{\epsilon}}(\eps \cdot) \big] (\bk; \wtx) \overline{  \psi(\bk;\wtx)} \ind \wtx.
\end{multline}
On denoting  $\wtu = \mathcal{J}_{\mathbb{R}^d} u$, integration of the above equation w.r.t. $\bk$ over $W_{\!{\boldsymbol{I}^*}}$ yields 
\begin{multline} \label{variational Bloch}
\int_{W_{\!{\boldsymbol{I}^*}}}\int_{W_{\!\boldsymbol{I}}}G(\wtx)\nabla_{\wtx} \wtu   (\bk; \wtx) \overline{\nabla_{\wtx} \psi(\bk;\wtx)} \ind \wtx \ind \bk - \omega^2 \int_{W_{\!{\boldsymbol{I}^*}}}\int_{W_{\!\boldsymbol{I}}}\rho(\wtx)\hh \wtu   (\bk; \wtx) \overline{  \psi(\bk;\wtx)} \ind \wtx \ind \bk   \\
~=~ \eps^2  \int_{W_{\!{\boldsymbol{I}^*}}} \int_{W_{\!\boldsymbol{I}}} \mathcal{J}_{\mathbb{R}^d} \big[{ f_{\epsilon}}(\eps \cdot) \big] (\bk; \wtx) \overline{  \psi(\bk;\wtx)} \ind \wtx \ind \bk.  
\end{multline}
Now we are ready to prove the following theorem.
\begin{theorem}\label{variational Bloch Thm}
If $u \in \mathcal{H}^1(\mathbb{R}^d)$ is a solution to \eqref{PDE}, then $\wtu(\bk; \wtx)=\mathcal{J}_{\mathbb{R}^d} u   (\bk; \wtx) \in L^2(W_{\!{\boldsymbol{I}^*}};\mathcal{H}^1_{\bk} (W_{\!\boldsymbol{I}}))$ solves \eqref{variational Bloch}. Conversely, if $\wtu(\bk;\wtx) \in L^2(W_{\!{\boldsymbol{I}^*}};\mathcal{H}^1_{\bk} (W_{\!\boldsymbol{I}}))$ is the solution to \eqref{variational Bloch}, then $u(\bx)=\big(\mathcal{J}^{-1}_{\mathbb{R}^d} \wtu   \big) (\bx) \in \mathcal{H}^1(\mathbb{R}^d)$ solves \eqref{PDE}.
\end{theorem}
\begin{proof}
From the above arguments we have shown that, for any $u\in C_0^\infty(\mathbb{R}^d)$ {solving~\eqref{PDE}} and $\psi(\bk;\wtx) \in L^2(W_{\!{\boldsymbol{I}^*}};\mathcal{H}^s_{\bk} (W_{\!\boldsymbol{I}}))$, equation \eqref{variational Bloch} holds. Since $C_0^\infty(\mathbb{R}^d)$ is dense in $\mathcal{H}^1(\mathbb{R}^d)$, and the (Floquet-)Bloch transform is an isomorphism between $\mathcal{H}^1(\mathbb{R}^d)$ and $L^2(W_{\!{\boldsymbol{I}^*}};\mathcal{H}^1_{\bk} (W_{\!\boldsymbol{I}}))$, therefore the variational formulation \eqref{variational Bloch} holds.

Conversely, if $\wtu(\bk;\wtx) \in L^2(W_{\!{\boldsymbol{I}^*}};\mathcal{H}^1_{\bk} (W_{\!\boldsymbol{I}}))$ is the solution to \eqref{variational Bloch}, then 
\begin{eqnarray*} 
\int_{W_{\!{\boldsymbol{I}^*}}}\int_{W_{\!\boldsymbol{I}}} \Big( \nabla_{\wtx}\!\cdot\!    \big(G\nabla_{\wtx} \wtu(\bk; \wtx)\big) + \omega^2\rho\hh \wtu    (\bk; \wtx) + \eps^2  \mathcal{J}_{\mathbb{R}^d} \big[{ f_{\epsilon}}(\eps \cdot) \big]  (\bk; \wtx) \Big) \overline{  \psi(\bk;\wtx)} \ind \wtx \ind \bk=0,
\end{eqnarray*}
which yields 
\begin{eqnarray*}
\nabla_{\wtx}\!\cdot\!    \big(G\nabla_{\wtx} \wtu(\bk; \wtx)\big) + \omega^2\rho\hh \wtu    (\bk; \wtx) + \eps^2  \mathcal{J}_{\mathbb{R}^d} \big[{ f_{\epsilon}}(\eps \cdot) \big]  (\bk; \wtx)=0
\end{eqnarray*}
in $L^2(W_{\!{\boldsymbol{I}^*}};\mathcal{H}^{-1}_{\bk} (W_{\!\boldsymbol{I}}))$. Since ${ f_{\epsilon}} \in L^2(\mathbb{R}^d)$, we have that $\mathcal{J}_{\mathbb{R}^d} { f_{\epsilon}} \in L^2(W_{\!{\boldsymbol{I}^*}};\mathcal{H}^{0}_{\bk} (W_{\!\boldsymbol{I}}))$. Since~$G$ is smooth, then by standard regularity, $\wtu \in L^2(W_{\!{\boldsymbol{I}^*}};\mathcal{H}^{2}_{\bk} (W_{\!\boldsymbol{I}}))$. Let  $u(\bx)=\big(\mathcal{J}^{-1}_{\mathbb{R}^d} \wtu   \big) (\bx)$. From Lemma \ref{Bloch isomorphism}, $u \in \mathcal{H}^2(\mathbb{R}^d)$. On applying the inverse (Floquet-)Bloch transform $\mathcal{J}^{-1}_{\mathbb{R}^d}$ to the above equation and making use of~\eqref{Bloch inside derivative}, we find that  
\begin{eqnarray*}
 -  \nabla\!\cdot\!   \big(G\nabla u\big)  (\bx)- \omega^2\rho\hh u  (\bx)   =\eps^2   \big[{ f_{\epsilon}}(\eps \cdot) \big]  (\bx) 
\end{eqnarray*}
in $L^2(\mathbb{R}^d)$, whereby $u(\bx)=\big(\mathcal{J}^{-1}_{\mathbb{R}^d} \wtu   \big) (\bx) \in \mathcal{H}^1(\mathbb{R}^d)$ solves \eqref{PDE}. 
\end{proof}

Since $\omega$  belongs to a band gap, then there exists at most one solution to \eqref{PDE}. Further  for each $\bk$, the variational problem in $W_{\!\boldsymbol{I}}$ \eqref{variational Bloch 1} is uniquely solvable (see for instance \cite{lions2011asymptotic}). As a result, we obtain the following claim. 
\begin{corollary}
Assume that $\omega$ belongs to a band gap. Then for any ${ f_{\epsilon}} \in L^2(\mathbb{R}^d)$, there exists a unique solution $u \in \mathcal{H}^1(\mathbb{R}^d)$ to \eqref{PDE}.
\end{corollary}
\subsection{(Floquet-)Bloch expansion of $u$}
{ 
\begin{theorem} \label{Theorem Bloch expansion solution}
If $u \in \mathcal{H}^1(\mathbb{R}^d)$ is a solution to \eqref{PDE}, then \begin{eqnarray}
\mathcal{J}_{\mathbb{R}^d}u  (\bk; \wtx) &=& \sum_{m=0}^\infty  \langle \rho \mathcal{J}_{\mathbb{R}^d}u(\bk; \cdot), e^{i \bk   \cdot}\phi_m(\bk;\cdot) \rangle e^{i \bk \cdot \wtx}  \phi_m(\bk;\wtx),  \nonumber \\
&=& \sum_{m=0}^\infty \frac{\eps^2    \langle  \mathcal{J}_{\mathbb{R}^d} \big[{ f_{\epsilon}}(\eps \cdot) \big] ({\bk}; \cdot), e^{i  {\bk} \cdot}\phi_m({\bk};\cdot) \rangle  }{\omega^2_m(\bk) -\omega^2} e^{i\bk\cdot   \wtx} \phi_m(\bk; \wtx), \quad  \wtx \in W_{\!\boldsymbol{I}}. \quad \label{Section Bloch expansion 2}
\end{eqnarray}
and $u$ is given by
\begin{eqnarray} \label{u bloch 0}
u(\bx)&=&  \frac{1}{(2\pi)^{d/2}} \int_{W_{\!{\boldsymbol{I}}^*}} \mathcal{J}_{\mathbb{R}^d}u(\bk; \bx-\bj) e^{i \bk \cdot  \bj} ~\ind \bk \qquad (\mbox{$\bj$ is such that $\bx-\bj \in W_{\!{\boldsymbol{I}}}$ } ) \nonumber \\
&=&\frac{1}{(2\pi)^{d/2}} \int_{W_{\!{\boldsymbol{I}^*}}} \sum_{m=0}^\infty  \frac{\eps^2    \langle  \mathcal{J}_{\mathbb{R}^d} \big[{ f_{\epsilon}}(\eps \cdot) \big] ( {\bk}; \cdot), e^{i  {\bk} \cdot}\phi_m( {\bk};\cdot) \rangle  }{\omega^2_m(\bk) -\omega^2} e^{i\bk\cdot   \bx} \phi_m(\bk;\bx) \ind \bk .
\end{eqnarray}
\end{theorem}
}
\begin{proof}
Starting from the variational formulation \eqref{variational Bloch 1} with test function $\psi(\bk;\wtx) = e^{i\bk \cdot \wtx} \phi_m(\bk;\wtx)$, we find 
\begin{multline*} 
\int_{W_{\!\boldsymbol{I}}}G(\wtx)\nabla_{\wtx} [\mathcal{J}_{\mathbb{R}^d} u(\bk; \wtx)] \overline{\nabla_{\wtx} e^{i\bk \cdot \wtx} \phi_m(\bk;\wtx)} \ind\wtx \\
- \omega^2 \int_{W_{\!\boldsymbol{I}}}\rho(\wtx)\hh [\mathcal{J}_{\mathbb{R}^d} u(\bk; \wtx)] \overline{e^{i\bk \cdot \wtx} \phi_m(\bk;\wtx)} \ind \wtx 
= \eps^2   \int_{W_{\!\boldsymbol{I}}} \mathcal{J}_{\mathbb{R}^d} \big[{ f_{\epsilon}}(\eps \cdot) \big] (\bk; \wtx) \overline{e^{i\bk \cdot \wtx} \phi_m(\bk;\wtx)} \ind \wtx.  
\end{multline*}
Integrating by parts the first term in the above equation and using~\eqref{eigenfunction}, one obtains 
\begin{multline*} 
\omega_m^2 \!\int_{W_{\!\boldsymbol{I}}} \rho(\wtx)[\mathcal{J}_{\mathbb{R}^d} u   (\bk; \wtx)] \overline{ e^{i\bk \cdot \wtx} \phi_m(\bk;\wtx)} \ind \wtx  \,-\, \omega^2 \!\int_{W_{\!\boldsymbol{I}}}\rho(\wtx)\hh [\mathcal{J}_{\mathbb{R}^d} u   (\bk; \wtx)] \overline{  e^{i\bk \cdot \wtx} \phi_m(\bk;\wtx)} \ind \wtx   \nonumber \\
= \eps^2   \int_{W_{\!\boldsymbol{I}}} \mathcal{J}_{\mathbb{R}^d} \big[{ f_{\epsilon}}(\eps \cdot) \big] (\bk; \wtx) \overline{  e^{i\bk \cdot \wtx} \phi_m(\bk;\wtx)} \ind \wtx, \label{Section Bloch expansion 1}
\end{multline*}
i.e.   
\begin{eqnarray*} 
&& (\omega_m^2(\bk)- \omega^2)  \langle \rho \mathcal{J}_{\mathbb{R}^d}u  (\bk; \cdot), e^{i \bk  \cdot}\phi_m(\bk;\cdot) \rangle    = \eps^2    \langle \mathcal{J}_{\mathbb{R}^d} \big[{ f_{\epsilon}}(\eps \cdot) \big] (\bk; \cdot), e^{i \bk \cdot}\phi_m(\bk;\cdot) \rangle.
\end{eqnarray*}
Since $\omega$ is fixed and $\omega^2_m(\bk) -\omega^2\not=0$, we thus obtain \eqref{Section Bloch expansion 2}. Then equation \eqref{u bloch 0} follows by applying the inverse (Floquet-)Bloch transform in Lemma \ref{Bloch isomorphism}.
\end{proof}

\subsection{Source term}

The source term ${ f_{\epsilon}}$ plays an important role in the second-order homogenization of wave motion in periodic media. To facilitate the analysis, we assume that ${ f_{\epsilon}}$ is the Fourier transform of a compactly-supported function (multiplied by the $p$th eigenfunction in the high frequency case). To bring specificity to the discussion, we assume the following.
\begin{assum} \label{assumption f}
The source term ${ f_{\epsilon}}$ in~\eqref{PDE fast} is assumed to has the following form
\begin{eqnarray}\label{source1}
{ f_{\epsilon}}(\bx) = \frac{1}{(2\pi)^{d/2}} \Big(\int_{\mathbb{R}^d} F(\bk) e^{i\bk\cdot \bx} \ind \bk\Big) \rho(\bx/\eps)\phi_p(\boldsymbol{0}; \bx/\eps),
\end{eqnarray}
where $\mathcal{F}[F] \in L^2(\mathbb{R}^d) \cap  L^1(\mathbb{R}^d)$, $F \in L^2(\mathbb{R}^d) \cap  L^1(\mathbb{R}^d)$, and $F$ is compactly supported in some open bounded region $Y \subset \mathbb{R}^d$ with $\boldsymbol{0} \in Y$. 
\end{assum}

\begin{remark}
When $p=0$, $\phi_p(\boldsymbol{0}; \bx/\eps)$ is a constant. This implies that ${ f_{\epsilon}}(\bx)/\rho({\bx/\eps})$ is proportional to the inverse Fourier transform of $F$. We also remark that the compact-support requirement on $F$ could be relaxed, for instance by allowing for sufficiently fast decaying functions. We illustrate this claim by a numerical example in Section~7.  
\end{remark}
{ 
\begin{lemma} \label{almost orthogonality general}
Let $\phi$ and $\psi$ be bounded $\boldsymbol{I}$-periodic functions. Assume that the Fourier series  of $\rho(\bx)\overline{\psi(\bx)}\phi(\bx)$ converges pointwise almost everywhere, i.e. 
\begin{eqnarray} \label{Fourier of phiphi general}
\rho(\bx)\overline{\psi(\bx)}\phi(\bx) \sim \sum_{\bn \in \mathbb{Z}^d} a_{\bn} e^{i 2\pi \bn \cdot \bx}.
\end{eqnarray}
Then  
\begin{equation} \label{lemma 1 limit general}
 \int_{\mathbb{R}^{d}}  \mathcal{F}[F](-\eps \bx)  \rho(\bx) \psi(\bx)\overline{e^{i\bk \cdot \bx} \phi(\bx)}  \ind\bx =\eps^{-d}(2\pi)^{d/2}   \sum_{\bn \in \mathbb{Z}^d}  \overline{a}_{\bn} F\big(\big(\bk+2\pi \bn\big)/\eps\big).
\end{equation}
Further for sufficiently small $\eps$, the following claims hold. 

\begin{enumerate}
\item[(a)] If $\bk \in W_{\!{\boldsymbol{I}^*}} \backslash \overline{\eps Y}$, then
\begin{eqnarray*}
 \int_{\mathbb{R}^{d}}  \mathcal{F}[F](-\eps \bx)  \rho(\bx) \psi(\bx)\overline{e^{i\bk \cdot \bx} \phi(\bx)}  \ind\bx =0.
\end{eqnarray*}
\item[(b)] If $\bk \in \eps Y$, then 
\begin{eqnarray*}
 \int_{\mathbb{R}^{d}}  \mathcal{F}[F](-\eps \bx)  \rho(\bx) \psi(\bx)\overline{e^{i\bk \cdot \bx} \phi(\bx)}  \ind\bx  =\eps^{-d} (2\pi)^{d/2}  \overline{\langle \rho(\cdot)\overline{\psi(\cdot)}\phi(\cdot) \rangle}  F(\bk/\eps),
\end{eqnarray*}
where $\langle \cdot \rangle$ denotes the average in $W_{\boldsymbol{I}}$.
\end{enumerate}
\end{lemma}
\begin{proof}
From the pointwise convergence of the Fourier series expansion~\eqref{Fourier of phiphi general} of $ \rho(\bx)\overline{\psi}(\bx)\phi(\bx)$, $\mathcal{F}[F] \in L^1(\mathbb{R}^d)$, and the dominated convergence theorem, we conclude that
\begin{eqnarray*}
&&\int_{\mathbb{R}^{d}}  \mathcal{F}[F](-\eps \bx)  \rho(\bx) \psi(\bx)\overline{e^{i\bk \cdot \bx} \phi(\bx)}  \ind\bx \\
&=& \sum_{n \in \mathbb{Z}^d}   \overline{a}_{\bn} \int_{\mathbb{R}^d}\mathcal{F}[F](-\eps \bx) e^{-i ( \bk + 2 \pi \bn)\cdot \bx}  \ind\bx =  \eps^{-d} \sum_{n \in \mathbb{Z}^d}  \overline{a}_{\bn} \int_{\mathbb{R}^d}\mathcal{F}[F](\by) e^{i ( \bk + 2 \pi \bn)\cdot \by/\eps}  \ind \by \\
&=& \eps^{-d}(2\pi)^{d/2} \sum_{\bn \in \mathbb{Z}^d}  \overline{a}_{\bn} F(\frac{\bk+2\pi \bn}{\eps}),
\end{eqnarray*}
where we have applied a change of variable $\by=-\eps \bx$ in the last two steps. This establishes claim~\eqref{lemma 1 limit general}.

Let us now prove part (a). Since $\bk \in W_{\!{\boldsymbol{I}^*}}$, we have $-\pi \leqslant {\bk}_{1,\cdots,d} \leqslant \pi $. We assume that $\bk \not \in \eps Y$ with $\boldsymbol{0} \in Y$; then for sufficiently small $\eps$, ${\bk}\not=\boldsymbol{0}$ and $\bk+2\pi \bn \not=\boldsymbol{0}$ for any $\bn \in \mathbb{Z}^d$. Therefore $\eps^{-1}(\bk+2 \pi \bn) \not \in Y$ for sufficiently small $\eps$. By Assumption \ref{assumption f}, $F$ is compactly supported in $Y \subset \mathbb{R}^d$, whereby $F(\eps^{-1}(\bk+2 \pi \bn))=0$ for any $\bn \in \mathbb{Z}^d$. By virtue of~\eqref{lemma 1 limit general}, we thus obtain   
\begin{eqnarray*}
\int_{\mathbb{R}^{d}}  \mathcal{F}[F](-\eps \bx)  \rho(\bx) \psi(\bx)\overline{e^{i\bk \cdot \bx} \phi(\bx)}  \ind\bx=0.
\end{eqnarray*}
Lastly we establish part (b). Since $\bk \in \eps Y$,  then for sufficiently small $\eps$, assumption $\eps^{-1}(\bk+ 2 \pi \bn) \in Y$ requires that  $\bn=\boldsymbol{0}$. This yields   
\begin{eqnarray*}
\int_{\mathbb{R}^{d}}  \mathcal{F}[F](-\eps \bx)  \rho(\bx) \psi(\bx)\overline{e^{i\bk \cdot \bx} \phi(\bx)}  \ind\bx =\eps^{-d}(2\pi)^{d/2} \overline{a}_{\boldsymbol{0}} F(\bk/\eps).
\end{eqnarray*}
This equation together with $\overline{a}_{\boldsymbol{0}}=\overline{\langle \rho(\bx)\overline{\psi(\bx)}\phi(\bx) \rangle}$ completes the proof.
\end{proof}
}
The above Lemma simply yields the following proposition.
\begin{proposition} \label{almost orthogonality}
Let $\phi$ be a  bounded $\boldsymbol{I}$-periodic function. Assume that the Fourier series  of $\rho(\bx)\overline{\phi_p(\boldsymbol{0}; \bx)} \phi(\bx)$ converges pointwise almost everywhere, then for sufficiently small $\eps$, the following claims hold. 

\begin{enumerate}
\item[(a)] If $\bk \in W_{\!{\boldsymbol{I}^*}} \backslash \overline{\eps Y}$, then
\begin{eqnarray*}
\int_{\mathbb{R}^d}   { f_{\epsilon}}(\eps  \bx) \overline{e^{i\bk \cdot \bx} \phi(\bx)} \ind\bx =0.
\end{eqnarray*}
\item[(b)] If $\bk \in \eps Y$, then 
\begin{eqnarray*}
\int_{\mathbb{R}^d}   { f_{\epsilon}}(\eps  \bx) \overline{e^{i\bk \cdot \bx} \phi(\bx)} \ind\bx =\eps^{-d} (2\pi)^{d/2}  \overline{\langle \rho(\cdot)\overline{\phi_p(\boldsymbol{0}; \cdot)}\phi(\cdot) \rangle} F(\bk/\eps).
\end{eqnarray*}
\end{enumerate}
\end{proposition}
\begin{proof}
By~\eqref{source1}, we obtain that
\begin{eqnarray*}
&&\hspace{-0.5cm}\int_{\mathbb{R}^d}   { f_{\epsilon}}(\eps  \bx) \overline{e^{i\bk \cdot \bx} \phi(\bx)} \ind\bx 
=\int_{\mathbb{R}^d}   \frac{1}{(2\pi)^{d/2}} \Big(\int_{\mathbb{R}^d} F(\boldeta) e^{i\eps \boldeta\cdot \bx} d\boldeta\Big) \rho(\bx)\phi_p(\boldsymbol{0}; \bx) \overline{e^{i\bk \cdot \bx} \phi(\bx)} \ind\bx \\
&=& \int_{\mathbb{R}^{d}}  \mathcal{F}[F](-\eps \bx)  \rho(\bx) \phi_p(\boldsymbol{0}; \bx)\overline{e^{i\bk \cdot \bx} \phi(\bx)}  \ind\bx.
\end{eqnarray*}
Then the proposition follows from Lemma \ref{almost orthogonality general}.
\end{proof}
\begin{remark}
In Proposition \ref{almost orthogonality}, we assumed that the Fourier series  of $ \rho(\bx)\overline{\phi_p(\boldsymbol{0};\bx)} \phi(\bx)$ converges pointwise almost everywhere, namely 
\begin{eqnarray*}
 \rho(\bx)\overline{\phi_p(\boldsymbol{0};\bx)} \phi(\bx) = \sum_{\bn \in \mathbb{Z}^d} a_{\bn} e^{i \bn \cdot \bx}.
\end{eqnarray*}
This premise holds for instance if: (i) $\rho(\cdot)\overline{\phi_p(\boldsymbol{0};\cdot)}\phi(\cdot) \!\in H^2_{\boldsymbol{0}}(W_{\!\boldsymbol{I}})$, {or} (ii) $\rho(\cdot)\overline{\phi_p(\boldsymbol{0};\cdot)}\phi(\cdot)$ is piecewise smooth in one dimension \cite[pp 35]{folland2009fourier}. In the former case, a standard regularity result demonstrates that the Fourier coefficients  satisfy
\begin{eqnarray*}
\sum_{\bn \in \mathbb{Z}^d} |a_{\bn}|^2 |\bn|^4 < \infty
\end{eqnarray*}
which, together with Holder inequality, yields 
\begin{eqnarray*}
\sum_{\bn \in \mathbb{Z}^d} |a_{\bn}| \leqslant  \big(\sum_{\bn \in \mathbb{Z}^d} |a_{\bn}|^2|\bn|^4\big)^{1/2}  \big(\sum_{\bn \in \mathbb{Z}^d} \frac{1}{|\bn|^4}\big)^{1/2} < \infty.
\end{eqnarray*}
This demonstrates that $\sum_{\bn \in \mathbb{Z}^d} a_{\bn} e^{i \bn \cdot \bx}$
{converges uniformly and pointwise}. {When} $\rho(\bx)\overline{\phi_p(\boldsymbol{0};\bx)}\phi(\bx) \in H^2_{\boldsymbol{0}}(W_{\!\boldsymbol{I}})$, then by the Sobolev embedding theorem \cite[pp 85--86]{adams2003sobolev}, $ \rho(\bx)\overline{\phi_p(\boldsymbol{0};\bx)}\phi(\bx)$ is continuous in $W_{\!\boldsymbol{I}}$. {In this case} we conclude that $\sum_{n \in \mathbb{Z}^d} a_{\bn} e^{i \bn \cdot \bx}$ converges uniformly and pointwise to the continuous function $\rho(\bx)\overline{\phi_p(\boldsymbol{0};\bx)}\phi(\bx)$.
\end{remark}
\subsection{Reduced (Floquet-)Bloch expansion of~$u$}
We seek to simplify the solution $u$ given by Theorem \ref{Theorem Bloch expansion solution}. To begin with, we prove the following proposition.

\begin{proposition} \label{Prop f Y to f Rd}
\begin{eqnarray}  \label{f Y to f Rd}
&&  \langle  \mathcal{J}_{\mathbb{R}^d} \big[{ f_{\epsilon}}(\eps \cdot) \big] ( {\bk}; \cdot), e^{i  {\bk}  \cdot}\phi_m( ;\cdot) \rangle \nonumber \\
&=& \Big\{ 
\begin{array}{cc}
\frac{1}{(2\pi)^{d/2}}  \int_{\mathbb{R}^d}     { f_{\epsilon}}(\eps \by) e^{-i  {\bk} \cdot \by} \overline{ \phi_m( {\bk};\by)}\ind \by,   &  \bk \in \epsilon Y\\
0,  &   \bk \in W_{\!{\boldsymbol{I}^*}}\backslash \overline{\eps Y}
\end{array}.
\end{eqnarray}  
\end{proposition}
\begin{proof}
We first directly compute
\begin{eqnarray*} 
&&  \langle  \mathcal{J}_{\mathbb{R}^d} \big[{ f_{\epsilon}}(\eps \cdot) \big] ( {\bk}; \cdot), e^{i  {\bk}  \cdot}\phi_m( ;\cdot) \rangle \nonumber \\
&=& \frac{1}{(2\pi)^{d/2}}  \int_{W_{\!\boldsymbol{I}}}   \sum_{\bj \in \mathbb{Z}^d} { f_{\epsilon}}(\eps \wtx+\eps  \bj) e^{-i  {\bk} \cdot   \bj} \overline{e^{i  {\bk} \cdot  \wtx}\phi_m( {\bk};\wtx)}\ind \wtx    \nonumber \\
&=& \frac{1}{(2\pi)^{d/2}}  \int_{\mathbb{R}^d}     { f_{\epsilon}}(\eps \by) e^{-i  {\bk} \cdot \by} \overline{ \phi_m( {\bk};\by)}\ind \by.  
\end{eqnarray*}  
Since~$G$  and~$\rho$ are smooth by premise, a standard regularity result yields $\rho\phi_m(\bk;\cdot) \in H^2_{\boldsymbol{0}}(W_{\!\boldsymbol{I}})$, see for instance \cite[Chap. 7]{saranen2013periodic} or \cite{lions2011asymptotic}. Then by Proposition \ref{almost orthogonality}, $\int_{\mathbb{R}^d}     { f_{\epsilon}}(\eps \by) e^{-i  {\bk} \cdot \by} \overline{ \phi_m( {\bk};\by)}\ind \by$ vanishes for $\bk \in W_{\!{\boldsymbol{I}^*}}\backslash \overline{\eps Y}$. This completes the proof.  
\end{proof}

With the above equality, we can rewrite the solution $u \in \mathcal{H}^1(\mathbb{R}^d)$ in Theorem \ref{Theorem Bloch expansion solution} as
\begin{eqnarray}  
&&u(\bx) \nonumber\\
&=& \frac{1}{(2\pi)^{d/2}} \int_{W_{\!{\boldsymbol{I}^*}}} \sum_{m=0}^\infty  \frac{\eps^2    \langle  \mathcal{J}_{\mathbb{R}^d} \big[{ f_{\epsilon}}(\eps \cdot) \big] ( {\bk}; \cdot), e^{i  {\bk} \cdot}\phi_m( {\bk};\cdot) \rangle  }{\omega^2_m(\bk) -\omega^2} e^{i\bk\cdot   \bx} \phi_m(\bk;\bx) \ind \bk \nonumber \\
&=&\frac{1}{(2\pi)^{d/2}} \int_{\eps Y} \sum_{m=0}^\infty  \frac{\eps^2   \langle  \mathcal{J}_{\mathbb{R}^d} \big[{ f_{\epsilon}}(\eps \cdot) \big] ( {\bk}; \cdot), e^{i  {\bk}  \cdot}\phi_m( ;\cdot) \rangle }{\omega^2_m(\bk) -\omega^2} e^{i\bk\cdot   \bx} \phi_m(\bk;\bx) \ind \bk.  \label{Section solution u 1}
%
\end{eqnarray}
\section{Contributions to the second-order approximation} \label{Significant contri} 

{Recalling we are interested in approximating, up to the second order, the wave motion at low wavenumbers. In this section, we demonstrate that such second-order
model originates completely from the $m=p$ term, while the contribution from all other branches is of higher order, more specifically~$O(\eps^3)$.} 

We first establish the following claim. 
\begin{theorem} \label{Thm uothers no contri}
Let 
\begin{equation} \label{u(p)}
\up(\bx)=
\frac{1}{(2\pi)^{d/2}} \int_{\eps Y}   \frac{\eps^2   \langle  \mathcal{J}_{\mathbb{R}^d} \big[{ f_{\epsilon}}(\eps \cdot) \big] ( {\bk}; \cdot), e^{i  {\bk}  \cdot}\phi_p( ;\cdot) \rangle }{\omega^2_p(\bk) -\omega^2} e^{i\bk\cdot   \bx} \phi_p(\bk;\bx) \ind \bk,
\end{equation}
and $\Up(\cdot) := \up(\cdot/\eps)$. Then for sufficiently small $\eps$, one has 
\begin{equation} \label{Thm u - uothers eqn 1}
\|u - \up\|_{L^2(\mathbb{R}^d)} \leqslant {C}\, \eps^{3-d/2} \|F\|_{L^2(\mathbb{R}^d)} 
\end{equation}
where {$C$ is a constant independent of~$\eps$}, and 
\begin{eqnarray}
\|{  U_{\epsilon} } - \Up\|_{L^2(\mathbb{R}^d)} \leqslant {C}\, \eps^{3} \|F\|_{L^2(\mathbb{R}^d)} \label{Thm u - uothers eqn 2},
\end{eqnarray}
\end{theorem}
\begin{proof}
From equation  \eqref{f Y to f Rd}, we have that
\begin{equation*}  
\up(\bx)=
\frac{1}{(2\pi)^{d}} \int_{\eps Y}   \frac{\eps^2  \int_{\mathbb{R}^d}     { f_{\epsilon}}(\eps \by) e^{-i  {\bk} \cdot \by} \overline{ \phi_p( {\bk};\by)}\ind \by }{\omega^2_p(\bk) -\omega^2} e^{i\bk\cdot   \bx} \phi_p(\bk;\bx) \ind \bk.
\end{equation*}
Let
\begin{equation} \label{Section contri u others}
u^{\mbox{\tiny ($m\!\neq\!p$)}}(\bx) = \frac{1}{(2\pi)^{d}} \int_{\eps Y}  \sum_{m\not=p} \frac{\eps^2  \int_{\mathbb{R}^d}     { f_{\epsilon}}(\eps \by) e^{-i  {\bk} \cdot \by} \overline{ \phi_m( {\bk};\by)}\ind \by }{\omega^2_m(\bk) -\omega^2} e^{i\bk\cdot   \bx} \phi_m(\bk;\bx) \ind \bk.
\end{equation}
In then follows that $u - \up = u^{\mbox{\tiny ($m\!\neq\!p$)}}$.

First, we find from assumption~\eqref{source1} that
\begin{eqnarray}
&& \int_{\mathbb{R}^d}     { f_{\epsilon}}(\eps \by) e^{-i  {\bk} \cdot \by} \overline{ \phi_m( {\bk};\by)}\ind \by \nonumber \\
&=&\int_{\mathbb{R}^d}      \frac{1}{(2\pi)^{d/2}} \Big(\int_{\mathbb{R}^d} F(\bk) e^{i\bk\cdot \eps \by} \ind \bk\Big) \rho(\by)\phi_p(\boldsymbol{0}; \by) e^{-i  {\bk} \cdot \by} \overline{ \phi_m( {\bk};\by)}\ind \by. \label{Section contri others 0-1}
\end{eqnarray}
We now seek to estimate this quantity. Our idea is to replace $\phi_p(\boldsymbol{0} ;\cdot)$ by $\phi_p({\bk};\cdot)$, and then estimate the difference. Since $\phi_p({\bk};\cdot)$ has a convergent Taylor series for $\bk$ in a small neighborhood of $\boldsymbol{0}$ (since $\phi_p({\bk};\cdot)$ and $\omega_p^2({\bk})$ are holomorphic in $\bk$ except a null set where the multiplicity of $\omega_p^2({\bk})$ changes \cite{wilcox1978theory}\cite[pp 341]{lions2011asymptotic}, and $\omega_p(\boldsymbol{0})$ is a simple eigenvalue), then for sufficiently small $\eps$ and $\bk \in  \eps Y$ we have 
 \begin{eqnarray}
\phi_p({\bk};\cdot) \:=\: \phi_p(\boldsymbol{0};\cdot) + {\bk} \cdot \boldsymbol{\psi}_p({\bk};\cdot)  
\end{eqnarray}
for some $\boldsymbol{\psi}_p({\bk};\cdot)$ that satisfies $\|\boldsymbol{\psi}_p({\bk};\cdot)\|_{(L^2(W_{\!\boldsymbol{I}}))^d} \leqslant c$, where $c$ is a constant independent of ${\bk}$. We then obtain
\begin{eqnarray}
 &&  \int_{\mathbb{R}^d} \frac{1}{(2\pi)^{d/2}} \Big(\int_{\mathbb{R}^d} F(\bk) e^{i\bk\cdot \eps \by} \ind \bk\Big) \rho(\by)\phi_p(\boldsymbol{0}; \by) e^{-i  {\bk} \cdot \by} \overline{ \phi_m( {\bk};\by)}\ind \by \nonumber \\
 &=& \int_{\mathbb{R}^d} \frac{1}{(2\pi)^{d/2}} \Big(\int_{\mathbb{R}^d} F(\bk) e^{i\bk\cdot \eps \by} \ind \bk\Big) \rho(\by)\Big(\phi_p({\bk};\by) -  {\bk} \cdot \boldsymbol{\psi}_p({\bk};\by) \Big) \cdot e^{-i  {\bk} \cdot \by} \overline{ \phi_m( {\bk};\by)}\ind \by. \nonumber \\ \label{Section contri others 1}
\end{eqnarray}
{ We next apply Lemma \ref{almost orthogonality general} with $\psi = \phi_p(\bk; \cdot), \phi = \phi_m(\bk;\cdot)$ and with $\psi = \boldsymbol{\psi}_p({\bk};\cdot), \phi = \phi_m(\bk;\cdot)$   respectively}. {Recalling the notation $\bk = \eps\hat{\bk}$}, one finds this applicable since both $\phi_p(\eps \hat{\bk};\by)$ and $\bk \cdot \boldsymbol{\psi}_p(\eps \hat{\bk};\by) = \phi_p(\eps \hat{\bk};\by) - \phi_p(\boldsymbol{0}; \by)$ belong to $\mathcal{H}^2_{\boldsymbol{0}}(W_{\!\boldsymbol{I}})$. This gives  
\begin{multline}
\int_{\mathbb{R}^d} \frac{1}{(2\pi)^{d/2}} \Big(\int_{\mathbb{R}^d} F(\bk) e^{i\bk\cdot \eps \by} \ind \bk\Big) \rho(\by) \phi_p({\bk};\by)  e^{-i  {\bk} \cdot \by} \overline{ \phi_m( {\bk};\by)}\ind \by \\
= \eps^{-d} (2\pi)^{d/2}  \langle  {\rho \phi_p(  {\bk}; \cdot)} \overline{\phi_m( {\bk};{\cdot})} \rangle F(\bk/\eps) =0 \mbox{ (orthogonality)}, \label{Section contri others 2}
\end{multline}
and
\begin{multline}
\int_{\mathbb{R}^d} \frac{1}{(2\pi)^{d/2}} \Big(\int_{\mathbb{R}^d} F(\bk) e^{i\bk\cdot \eps \by} \ind \bk\Big) \rho(\by) \big(\bk \cdot \boldsymbol{\psi}_p({\bk};\by) \big)  e^{-i  {\bk} \cdot \by} \overline{ \phi_m( {\bk};\by)}\ind \by \\
= \eps^{-d} (2\pi)^{d/2}  \langle  {\rho \bk \cdot \boldsymbol{\psi}_p(  {\bk}; \cdot)} \overline{\phi_m( {\bk};{\cdot})} \rangle F(\bk/\eps).\label{Section contri others 3}
\end{multline}
In this setting, \eqref{Section contri others 0-1} can be estimated via~\eqref{Section contri others 1}, \eqref{Section contri others 2} and \eqref{Section contri others 3}; in particular, we find 
\begin{eqnarray}
&&\hspace{-1cm} \int_{\mathbb{R}^d} { f_{\epsilon}}(\eps \by) e^{-i  {\bk} \cdot \by} \overline{ \phi_m( {\bk};\by)}\ind \by \:=\, -\bk \eps^{-d} (2\pi)^{d/2}  \langle  {\rho \bk\cdot \boldsymbol{\psi}_p(  {\bk}; \cdot)} \overline{\phi_m( {\bk};{\cdot})} \rangle F(\bk/\eps). \label{Section contri others 4}
\end{eqnarray}

To estimate $\|u^{\mbox{\tiny ($m\!\neq\!p$)}}\|_{L^2(\mathbb{R}^d)}$ according to~\eqref{Section contri u others}, we write
\begin{eqnarray*}
\|u^{\mbox{\tiny ($m\!\neq\!p$)}}\|^2_{L^2(\mathbb{R}^d)} &=& \sum_{m\not=p} \int_{\eps Y} \Big| \frac{1}{(2\pi)^d}\frac{\eps^2  \int_{\mathbb{R}^d}     { f_{\epsilon}}(\eps \by) e^{-i  {\bk} \cdot \by} \overline{ \phi_m( {\bk};\by)}\ind \by }{\omega^2_m(\bk) -\omega^2} \Big|^2 \ind \bk,
\end{eqnarray*}
and make use of~\eqref{Section contri others 4} to obtain 
\begin{eqnarray*}
\|u^{\mbox{\tiny ($m\!\neq\!p$)}}\|^2_{L^2(\mathbb{R}^d)} &\leq& \frac{1}{(2\pi)^d}\sum_{m\not=p} \int_{\eps Y} \|\bk\|^2 \Big|\frac{\eps^{2-d}  \langle  {\rho \boldsymbol{\psi}_p(  {\bk}; \cdot)} \overline{\phi_m( {\bk};{\cdot})} \rangle F(\bk/\eps) }{\omega^2_m(\bk) -\omega^2} \Big|^2 \ind \bk \\
&=&  \frac{\eps^{4-2d} }{(2\pi)^d} \int_{\eps Y} \|\bk\|^2 \Big|  F(\bk/\eps)  \Big|^2  \sum_{m\not=p} \Big| \frac{\langle  {\rho \boldsymbol{\psi}_p(  {\bk}; \cdot)} \overline{\phi_m( {\bk};{\cdot})}\rangle}{\omega^2_m(\bk) -\omega^2} \Big|^2\ind \bk. \\
\end{eqnarray*}
Since $\omega^2_m(\bk) -\omega^2=O(1)$ for any $m\not=p$, from the last equality  we find 
\begin{eqnarray*}
\|u^{\mbox{\tiny ($m\!\neq\!p$)}}\|^2_{L^2(\mathbb{R}^d)} &\leqslant&  c{''} \es \frac{\eps^{4-2d} }{(2\pi)^d} \int_{\eps Y} \|\bk\|^2 \Big|  F(\bk/\eps)  \Big|^2  \sum_{m\not=p} \Big| \langle  {\rho \boldsymbol{\psi}_p(  {\bk}; \cdot)} \overline{\phi_m( {\bk};{\cdot})}\rangle \Big|^2\ind \bk \\
&\leqslant& c{'} \es \frac{\eps^{4-2d} }{(2\pi)^d} \int_{\eps Y} \|\bk\|^2 \Big|  F(\bk/\eps)  \Big|^2 \|\boldsymbol{\psi}_p(  {\bk}; \cdot) \|^2_{(L^2(W_{\!\boldsymbol{I}}))^d} \ind \bk, \\
\end{eqnarray*}
where $c'$ and $c^{''}$ are constants independent of $\eps$. On recalling that $\|\boldsymbol{\psi}_p({\bk};\cdot)\|_{(L^2(W_{\!\boldsymbol{I}}))^d} \leqslant c$ for some $c$ independent of ${\bk}$, from the above estimate we obtain 
\begin{eqnarray*}
\|u^{\mbox{\tiny ($m\!\neq\!p$)}}\|{^2}_{L^2(\mathbb{R}^d)} 
&\leqslant& c\es c' \, \frac{\eps^{4-2d} }{(2\pi)^d} \int_{\eps Y} \|\bk\|^2 \Big|  F(\bk/\eps)  \Big|^2  \ind \bk \:=\: c\es c'  \frac{\eps^{6-d} }{(2\pi)^d} \int_{Y} |\hat{\bk}|^2 \Big|  F(\hat{\bk})  \Big|^2  \ind \hat{\bk} \\
&\leqslant& {C^2} \, \eps^{6-d} \|F\|^2_{L^2(\mathbb{R}^d)},
\end{eqnarray*}
where $C$ is a constant independent of $\eps$. This establishes claim~\eqref{Thm u - uothers eqn 1}. Equation \eqref{Thm u - uothers eqn 2} then follows immediately from via the change of variable. 
\end{proof}

%
To motivate our asymptotic study in the next section, we make the following observation. First thanks to equation  \eqref{f Y to f Rd} and Proposition \ref{almost orthogonality}, one finds that 
\begin{eqnarray*}
 \langle  \mathcal{J}_{\mathbb{R}^d} \big[{ f_{\epsilon}}(\eps \cdot) \big] ( {\bk}; \cdot), e^{i  {\bk}  \cdot}\phi_m( ;\cdot) \rangle &=& \frac{1}{(2\pi)^{d/2}} \int_{\mathbb{R}^d}    { f_{\epsilon}}(\eps \by ) e^{-i  {\bk} \cdot \by} \overline{ \phi_p( {\bk};\by)}\ind \by \\
 &=&\eps^{-d}  \langle  {\rho \phi_p(\boldsymbol{0}; \cdot)} \overline{\phi_p( {\bk};{\cdot})} \rangle F(\bk/\eps).
\end{eqnarray*}
As a result, we obtain from \eqref{u(p)} the following representation
\begin{eqnarray}\label{Section u drive eqn 1}
\up(\bx) &=& \frac{1}{(2\pi)^{d/2}} \int_{\eps Y}    \frac{\eps^{2-d}    \langle \rho {\phi_p(\boldsymbol{0}; \cdot)} \overline{\phi_p( {\bk};{\cdot})} \rangle F(\bk/\eps) }{\omega^2_p(\bk) -\omega^2} e^{i\bk\cdot   \bx} \phi_p(\bk;\bx) \ind \bk,
\end{eqnarray}
and the expression of $U^{(p)}=\up(\cdot/\epsilon)$ with the help of a change of variable
\begin{eqnarray}
U^{(p)}(\br) &\:=\:& \frac{1}{(2\pi)^{d/2}} \int_{Y}    \frac{   \langle \rho {\phi_p(\boldsymbol{0}; \cdot)} \overline{\phi_p( \eps \hat{\bk};{\cdot})} \rangle F(\hat{\bk}) e^{i\eps \hat{\bk} \cdot   \br/\epsilon} \phi_p(\eps \hat{\bk};\br/\epsilon)}{\omega^2_p(\eps \hat{\bk}) -\omega^2}  \ind \hat{\bk}. \label{Section U drive eqn 1}
\end{eqnarray}
To obtain the second-order asymptotic model of $\Up$, one accordingly needs the respective second-order approximations of $\phi_p(\eps \hat{\bk};\cdot)$ and $\omega^2_p(\eps \hat{\bk}) $. This motivates our study in the next section.
\section{Asymptotic expansion of $\phi_p({\bk};\cdot)$ and $\omega^2_p({\bk})$} \label{Asymptotic eigenfunction dispersion relation}
{  According to \cite{wilcox1978theory}, $\phi_p({\bk};\cdot)$ and $\omega_p^2({\bk})$ are holomorphic in $\bk$ except a null set where the multiplicity of $\omega_p^2({\bk})$ changes (see also \cite[pp 341]{lions2011asymptotic}).}
{  Since $\omega_p^2({\boldsymbol{0}})$ under our consideration is a simple eigenvalue (such that its multiplicity does not change in some small neighborhood of $\omega_p^2({\boldsymbol{0}})$, and accordingly $\phi_p({\bk};\cdot)$ and $\omega_p^2({\bk})$ are holomorphic in $\bk$ in such a small neighborhood) have convergent Taylor series, and the convergence for the Taylor series of $\phi_p({\bk};\cdot)$ is in $C(\overline{W_{\boldsymbol{I}}})$ \cite{wilcox1978theory} and hence in $L^2(W_{\boldsymbol{I}})$. As before, we let $\bk=\eps \hat{\bk}$ where $\eps$ is sufficiently small and $\hat{\bk}=O(1)$. Since $\omega_p({\boldsymbol{0}})$ is simple by premise, one easily verifies that $\omega_p(\eps \hat{\bk}) = \omega_p(-\eps \hat{\bk})$ for sufficiently small $\eps$; then $\frac{\partial^q}{\partial {\eps^q}}\omega_0^2(\eps \hat{\bk})=0$ for any odd $q$. As a result, we obtain the series expansion as in equations \eqref{wexp 1} and~\eqref{wexp 2}.} We further remark that such series expansion can be easily casted according to the perturbation theory \cite{kato2013perturbation} for a fixed direction $\hat{\bk}$ for less regular $G$ and $\rho$. However, here we apply the result of \cite{wilcox1978theory} such that the series expansion holds for all directions. 

We first state the following lemma which summarize the second-order approximations and prove the lemma later on.
\begin{lemma} \label{asymptotic main phi_0} 
For sufficiently small $\epsilon$, $\phi_p(\eps \hat{\bk};\bx)$ and $\omega_p^2(\eps \hat{\bk})$ have the following series expansion
\begin{eqnarray}
\phi_p(\eps \hat{\bk};\bx) ~&=&~  \tilde{w}_0(\bx) \;+\; \eps \hh \tilde{ w}_1(\bx) \;+\; \eps^2\hh \tilde{w}_2(\bx) \;+\; \eps^3\hh \tilde{ w}_3(\bx)  \;+\; \cdots, \label{wexp 1}\\
\omega_p^2(\eps \hat{\bk}) ~&=&~  \hat{\omega}_0^2 + \eps^2 \hat{\omega}_2^2 + \eps^4 \hat{\omega}_4^2 +   \cdots, \label{wexp 2}
\end{eqnarray}
where the convergence of \eqref{wexp 1} is in $L^2(W_{\!\boldsymbol{I}})$. Moreover, the unknown terms in \eqref{wexp 1} and~\eqref{wexp 2} are given respectively by 
\begin{eqnarray}
\tilde{w}_0(\bx) &=&  \ww_0 \phi_p(\boldsymbol{0};\bx), \label{wexp eigenfunction w0}\\
\tilde{ w}_1(\bx) &=& \bchi^{\mbox{\tiny{(1)}}} (\bx) \cdot i \hat{\bk}\hh \ww_0 + \bw_1 \cdot i\hat{\bk} \phi_p(\boldsymbol{0};\bx), \label{wexp eigenfunction w1}\\
\tilde{ w}_2 (\bx) \,&=&\, \bchi^{\mbox{\tiny{(2)}}} (\bx) : (i \hat{\bk})^{2}\hh  \ww_0 \,+\, \bchi^{\mbox{\tiny{(1)}}} (\bx) \otimes \bw_1 : (i \hat{\bk})^2   \, +\, \bw_2 :  (i \hat{\bk})^2 \phi_p(\boldsymbol{0};\bx), \label{wexp eigenfunction w2}\\
\tilde{ w}_3 (\bx) &=& \bchi^{\mbox{\tiny{(3)}}} (\bx):  (i \hat{\bk})^3 w_0  + (\bchi^{\mbox{\tiny{(2)}}} (\bx) : (i \hat{\bk})^2) \bw_1\cdot i \hat{\bk} \nonumber \\
&& +\bchi^{\mbox{\tiny{(1)}}}(\bx)   \cdot i\hat{\bk}\hh (\bw_2 : (i \hat{\bk})^2) + \bw_3: (i \hat{\bk})^3\phi_p(\boldsymbol{0};\bx),\label{wexp eigenfunction w3}
\end{eqnarray}
and 
\begin{eqnarray}\label{wexp dispersion}
\hat{\omega}_0^2=\omega_p^2(\boldsymbol{0}),\hspace{0.5cm}\hat{\omega}_2^2 =  -\frac{{\bmu}^{\mbox{\tiny{(0)}}}}{\rho^{\mbox{\tiny{(0)}}}}\!: (i\hat{\bk})^2,\hspace{0.5cm} \hat{\omega}_4^2 = -\frac{{\bmu}^{\mbox{\tiny{(2)}}}}{\rho^{\mbox{\tiny{(0)}}}}\!: (i\hat{\bk})^4, \qquad 
\end{eqnarray}
where constants $\ww_0$, $\bw_1$ satisfy the necessary conditions 
\begin{eqnarray}
 \ww_0 &=& 1, \label{necessary w0}\\
i \hat{\bk} \cdot ( \overline{\bw}_1-  \bw_1)&=&0, \label{necessary w1}\\
(i\hat{\bk})^2:\Big[ (\bw_2   + \overline{\bw}_2   ) - (\bw_1 \otimes \overline{\bw}_1) -      \langle \rho \bchi^{(1)} \otimes \overline{\bchi^{(1)}} \rangle  \Big]&=&0; \quad \label{necessary w2}
\end{eqnarray}
$\bchi^{\mbox{\tiny{(1)}}} $ is given by \eqref{ComparisonEqnchi1}; $\rho^{\mbox{\tiny{(0)}}}$ and ${\bmu}^{\mbox{\tiny{(0)}}}$ are given by \eqref{SecondOrderHomogenizationZerothOrderCoefficients}; $ \bchi^{\mbox{\tiny{(2)}}} $ is given by \eqref{Comparisonchi2}; $\brho_2$ and ${\bmu}^{\mbox{\tiny{(2)}}}$ are given by \eqref{SecondOrderHomogenizationSecondOrderCoefficients}, and $\bchi^{\mbox{\tiny{(3)}}}$ is given by \eqref{Comparisonchi3}. Here we have used the following short-hand notations.
\end{lemma}
\begin{itemize}
\item $:$ denotes the contraction between two tensors.
\item $\{\boldsymbol{\cdot}\}$ denotes tensor averaging over all index permutations; in particular for an $n$th-order tensor~$\boldsymbol{\tau}$, one has 
\begin{equation} \label{symtot}
\{\boldsymbol{\tau}\}_{j_1,j_2,\ldots j_n} ~=~ \frac{1}{n!}\sum_{(l_1,l_2,\ldots l_n)\in P} \boldsymbol{\tau}_{l_1,l_2,\ldots l_n}, \qquad j_1,j_2,\ldots j_n \in\overline{1,d}
\end{equation}
where~$P$ denotes the set of all permutations of~$(j_1,j_2,\ldots j_n)$. 
\item $\{\boldsymbol{\cdot}\}'$ denotes partial tensor symmetrization according to 
\begin{equation} \label{sympart}
\{\boldsymbol{\tau}\}'_{j_1,j_2,\ldots j_n} ~=~ \frac{1}{(n\!-\!1)!}\sum_{(l_2,\ldots l_n)\in Q} \boldsymbol{\tau}_{j_1,l_2,\ldots l_n}, \qquad j_1,j_2,\ldots j_n \in\overline{1,d}
\end{equation}
where~$Q$ denotes the set of all permutations of~$(j_2,j_3,\ldots j_n)$.
\end{itemize}

{ Here we remark that the equations for $\bchi^{\mbox{\tiny{(j)}}} $ (with $j=1,2,3$) are in the same spirit to those in \cite{wautier2015second,meng2018dynamic,guzina2019rational}.} Moreover
 that the featured smooth restriction on  $G$ and $\rho$ could in principle be relaxed, see for instance \cite{conca2002bloch} the the of case non-smooth coefficients  when $p=0$ (corresponding to the first branch). We also illustrate by numerical examples that our convergence results appear to apply for piecewise-constant coefficients.
\subsection{Leading-order approximation} \label{Asymptotic eigenfunction dispersion relation subsection 1}
{  
To begin with, note that the eigenfunction $\phi_p(\eps \hat{\bk};\bx)$ satisfies the eigenvalue problem \eqref{eigenfunction}, therefore multiplying equation \eqref{eq:A2} by $\phi \in H^2_{\boldsymbol{0}}(W_{\!\boldsymbol{I}})$, integrating over $W_{\!\boldsymbol{I}}$, and performing integration by parts twice yield that 
$$
-\int_{W_{\!\boldsymbol{I}}}   \phi_p(\eps \hat{\bk};\bx)   \overline{ (\nabla + i \eps \hat{\bk}) \cdot \big[G \big(\nabla+ i \eps \hat{\bk}  ) \phi \big]} \ind \bx=  \omega_p^2(\eps \hat{\bk}) \int_{W_{\!\boldsymbol{I}}}  \rho  \phi_p(\eps \hat{\bk};\bx) \overline{\phi} \ind \bx,
$$
for any $\phi \in H^2_{\boldsymbol{0}}(W_{\!\boldsymbol{I}})$ and recalling that ${\mathcal{H}^2_{\boldsymbol{0}}}(W_{\!\boldsymbol{I}})$ is given by \eqref{Hsk Fourier norm}, noting that the subscript $\boldsymbol{0}$ indicates periodicity.

From~\eqref{wexp 1} and~\eqref{wexp 2}, then we directly have 
\begin{eqnarray} \label{eq:A1}
&&-\int_{W_{\!\boldsymbol{I}}}   \big( \tilde{w}_0 + \eps \hh \tilde{ w}_1 + \eps^2\hh \tilde{w}_2 + \eps^3\hh \tilde{ w}_3  +\cdots \big)   \overline{(\nabla + i \eps \hat{\bk}) \cdot \big[G \big(\nabla+ i \eps \hat{\bk}  ) \phi \big]}  \ind \bx~~~~~~  \\
&&=  \big(\hat{\omega}_0^2 + \eps^2 \hat{\omega}_2^2 + \eps^4 \hat{\omega}_4^2 +\cdots \big) \int_{W_{\!\boldsymbol{I}}}  \rho \big( \tilde{w}_0 + \eps \hh \tilde{ w}_1 + \eps^2\hh \tilde{w}_2 + \eps^3\hh \tilde{ w}_3  +\cdots \big) \overline{\phi} \ind \bx, \nonumber 
\end{eqnarray}
for any $\phi \in H^2_{\boldsymbol{0}}(W_{\!\boldsymbol{I}})$.

From~\eqref{wexp 2}, it is clear that $\hat{\omega}_0=\omega_p(\boldsymbol{0})$. The $O(1)$ contribution stemming from \eqref{eq:A1} reads
\begin{eqnarray} \label{eq:A2}
&&-\int_{W_{\!\boldsymbol{I}}}       \tilde{w}_0 \overline{\nabla \cdot \big[G\nabla\phi\big]} \ind \bx \:=\: \hat{\omega}_0^2  \int_{W_{\!\boldsymbol{I}}}  \rho  \tilde{w}_0 \overline{\phi} \ind \bx,  \quad \forall \phi \in {\mathcal{H}^2_{\boldsymbol{0}}}(W_{\!\boldsymbol{I}}), \quad
\end{eqnarray}
Note that $\hat{\omega}_0=\omega_p(\boldsymbol{0})$, 
then $\tilde{w}_0(\bx) \in L^2(W_{\!\boldsymbol{I}})$ can be solved by
\begin{equation} \label{derivation wexp eigenfunction w0}
\tilde{w}_0(\bx) = \ww_0 \phi_p(\boldsymbol{0}; \bx).
\end{equation}  
Since $\phi_p({\bk}; \bx)=\tilde{w}_0(\bx) + O(\|\bk\|)$ for sufficiently small $\|\bk\|$,  it is clear that $\ww_0=1$. Moreover as $\phi_p(\boldsymbol{0}; \bx) \in {\mathcal{H}^1_{\boldsymbol{0}}}(W_{\!\boldsymbol{I}})$, we have $\tilde{w}_0(\bx) \in {\mathcal{H}^1_{\boldsymbol{0}}}(W_{\!\boldsymbol{I}})$ as well.}

Next, the $O(\eps)$ statement reads 
{ 
\begin{multline*}
-\int_{W_{\!\boldsymbol{I}}}       \tilde{w}_1    \overline{\nabla \cdot \big[G\nabla\phi\big]} \ind \bx + \int_{W_{\!\boldsymbol{I}}}  \big[G  i  \hat{\bk}  \tilde{w}_0  \big] \cdot   \overline{\nabla\phi} \ind \bx + \int_{W_{\!\boldsymbol{I}}} \tilde{w}_0 {\nabla \cdot \big[G i \hat{\bk} \overline{\phi}\big] } \\
=   \hat{\omega}_0^2  \int_{W_{\!\boldsymbol{I}}}  \rho  \tilde{w}_1 \overline{\phi} \ind \bx, \quad \forall \phi \in H^2_{\boldsymbol{0}}(W_{\!\boldsymbol{I}}).
\end{multline*}
}
This equation can be solved by
\begin{equation} \label{derivation wexp eigenfunction w1}
\tilde{ w}_1(\bx) = \bchi^{\mbox{\tiny{(1)}}} (\bx) \cdot i \hat{\bk}\hh w_0 + \bw_1 \cdot i\hat{\bk}\phi_p(\boldsymbol{0}; \bx),
\end{equation}
where $\bchi^{\mbox{\tiny{(1)}}}\in (H^1_{\boldsymbol{0}}(W_{\!\boldsymbol{I}}))^d$ is a \emph{zero-mean} ($\int_{W_{\!\boldsymbol{I}}} \rho\bchi^{\mbox{\tiny{(1)}}}  \overline{\phi_p(\boldsymbol{0}; \cdot)} \ind\bx=\boldsymbol{0}$), vector-valued function solving 
\begin{multline} \label{ComparisonEqnchi1}
\int_{W_{\!\boldsymbol{I}}}  \big( G (\nabla \bchi^{\mbox{\tiny{(1)}}} \!+\boldsymbol{I}\phi_p(\boldsymbol{0}; \cdot)\big) : \overline{\nabla\phi} \ind \bx - \int_{W_{\!\boldsymbol{I}}} \big(  G\boldsymbol{I} \nabla \phi_p(\boldsymbol{0}; \cdot)  \big) \overline{\phi} \ind \bx  \\
 = \hat{\omega}_0^2 \int_{W_{\!\boldsymbol{I}}}  \rho\bchi^{\mbox{\tiny{(1)}}}    \overline{ \phi} \ind \bx,  \quad \forall \phi \in H^1_{\boldsymbol{0}}(W_{\!\boldsymbol{I}}).
\end{multline}
{ We remark that the above equation for $\bchi^{\mbox{\tiny{(1)}}} $ is uniquely solvable since one can check that, the terms (which do not involve $\bchi^{\mbox{\tiny{(1)}}}$)  
$$\int_{W_{\!\boldsymbol{I}}}  \big( G \boldsymbol{I}\phi_p(\boldsymbol{0}; \cdot)\big) : \overline{\nabla\phi} \ind \bx - \int_{W_{\!\boldsymbol{I}}} \big(  G\boldsymbol{I} \nabla \phi_p(\boldsymbol{0}; \cdot)  \big) \overline{\phi} \ind \bx
$$ vanish when $\phi = \phi_p(\boldsymbol{0}; \cdot)$ (since $\phi_p(\boldsymbol{0}; \cdot)$ has a constant phase thanks to the assumption that $\omega_p^2(\boldsymbol{0})$ is a simple eigenvalue), i.e. the Fredholm alternative holds.
}

The $O(\eps^{2})$ contribution stemming from \eqref{eq:A1} reads
{ 
\begin{multline}\label{w2}
-\int_{W_{\!\boldsymbol{I}}}         \tilde{w}_2  \nabla \cdot \big(G \overline{\nabla\phi} \big) \ind \bx+
\int_{W_{\!\boldsymbol{I}}}  \big(G  i  \hat{\bk}  \tilde{w}_1  \big) \cdot   \overline{\nabla\phi} \ind \bx -\int_{W_{\!\boldsymbol{I}}}  \big[G \big(\nabla \tilde{w}_1 + i\hat{\bk}  \tilde{w}_0\big)  \cdot i  \hat{\bk}\big]  \overline{   \phi} \ind \bx \\
=   \int_{W_{\!\boldsymbol{I}}}  \rho   \big[ \hat{\omega}_0^2 \tilde{w}_2+  \hat{\omega}_2^2  \tilde{w}_0 \big] \overline{\phi} \ind \bx, \quad \forall \phi \in H^2_{\boldsymbol{0}}(W_{\!\boldsymbol{I}}).
\end{multline}
}
Taking $\phi(\bx)=\phi_p(\boldsymbol{0}; \bx)$, we find 
\begin{multline*}
\int_{W_{\!\boldsymbol{I}}}  \big[G   i  \hat{\bk}  \tilde{w}_1  \big] \cdot   \overline{\nabla\phi_p(\boldsymbol{0}; \cdot)} \ind \bx -\int_{W_{\!\boldsymbol{I}}}  \big[G \big(\nabla \tilde{w}_1  + i\hat{\bk}  \tilde{w}_0 \big)  \cdot i  \hat{\bk}\big]  \overline{  \phi_p(\boldsymbol{0}; \cdot)} \ind \bx ~~~~ \nonumber\\
=\; \hat{\omega}_2^2  \int_{W_{\!\boldsymbol{I}}}  \rho  \tilde{w}_0(\cdot) \overline{\phi_p(\boldsymbol{0}; \cdot)} \ind \bx.
\end{multline*}
On substituting~\eqref{derivation wexp eigenfunction w0}--\eqref{derivation wexp eigenfunction w1} into the above equation, all the terms involving $\bw_1$ vanish due to the fact that $\phi_p(\boldsymbol{0}; \cdot)$ has a constant phase. Accordingly, we obtain 
\begin{eqnarray} \label{ComparisonSecondOrderHomogenizationZerothOrderPDE}
\rho^{\mbox{\tiny{(0)}}} \hat{\omega}^2_2 =  -{\bmu}^{\mbox{\tiny{(0)}}} \!: (i\hat{\bk})^2,
\end{eqnarray}
where 
\begin{eqnarray} 
\rho^{\mbox{\tiny{(0)}}} &=&  {\alpha_p}\es \langle \rho  \phi_p(\boldsymbol{0}; \cdot) \overline{\phi_p(\boldsymbol{0}; \cdot) } \rangle, \label{SecondOrderHomogenizationZerothOrderCoefficients-0}\\
  {\bmu}^{\mbox{\tiny{(0)}}} &=& {\alpha_p}\es \big\langle\{G( \nabla  \bchi^{\mbox{\tiny{(1)}}}   + \boldsymbol{I}  \phi_p(\boldsymbol{0}; \cdot)  ) \overline{ \phi_p(\boldsymbol{0}; \cdot) }\} -
  \{G \bchi^{\mbox{\tiny{(1)}}} \otimes \overline{\nabla \phi_p(\boldsymbol{0}; \cdot)}\} \big\rangle, \label{SecondOrderHomogenizationZerothOrderCoefficients}
\end{eqnarray}
{and
\[
\alpha_p = \langle\phi_p(\boldsymbol{0}; \cdot) \overline{\phi_p(\boldsymbol{0};\cdot)}\rangle^{-1}. 
\]}

\subsection{First-order corrector}\label{Asymptotic eigenfunction dispersion relation subsection 2}

Let $\bchi^{\mbox{\tiny{(2)}}}\in (H^1_{\boldsymbol{0}}(W_{\!\boldsymbol{I}}))^{d\times d}$ denote a \emph{zero-mean} (i.e. $\int_{W_{\!\boldsymbol{I}}} \rho \bchi^{\mbox{\tiny{(2)}}} \overline{\phi_p(\boldsymbol{0}; \cdot)} \ind\bx=\boldsymbol{0}$), tensor-valued function  satisfying 

\begin{multline} \label{Comparisonchi2}
\int_{W_{\!\boldsymbol{I}}}  \big( G (\nabla \bchi^{\mbox{\tiny{(2)}}} \!+ \{\boldsymbol{I} \otimes \bchi^{\mbox{\tiny{(1)}}}\}' \big) : \overline{\nabla\phi} \ind \bx - \int_{W_{\!\boldsymbol{I}}} G(\{\nabla  \bchi^{\mbox{\tiny{(1)}}}\} + \boldsymbol{I} \phi_p(\boldsymbol{0};\cdot)) \overline{\phi} \ind \bx \\
 + \int_{W_{\!\boldsymbol{I}}} \frac{\rho}{\rho^{\mbox{\tiny{(0)}}}} \bmu^{\mbox{\tiny{(0)}}}  \phi_p(\boldsymbol{0};\cdot) \overline{\phi} \ind \bx \:=\: \hat{\omega}_0^2 \int_{W_{\!\boldsymbol{I}}}  \rho\bchi^{\mbox{\tiny{(2)}}}    \overline{ \phi} \ind \bx,  \quad \forall \phi \in H^1_{\boldsymbol{0}}(W_{\!\boldsymbol{I}}).
\end{multline}
{ It is directly verified that the above equation is uniquely solvable since one can again check that, the terms which do not involve $\bchi^{\mbox{\tiny{(2)}}}$ vanish when $\phi = \phi_p(\boldsymbol{0}; \cdot)$ (thanks to \eqref{SecondOrderHomogenizationZerothOrderCoefficients}), i.e. the Fredholm alternative holds.} With such definitions, a direct calculation yields that equation \eqref{w2} can be solved by 
\begin{equation}\label{derivation wexp eigenfunction w2}
\tilde{w}_2 (\bx) \,=\, \bchi^{\mbox{\tiny{(2)}}} (\bx) : (i \hat{\bk})^{2}\hh  w_0 \,+\, \big(\bchi^{\mbox{\tiny{(1)}}} (\bx) \cdot i\hat{\bk}\big)   \big(\bw_1 \cdot i \hat{\bk} \big)   \, +\, \bw_2 :  (i \hat{\bk})^2 \phi_p(\boldsymbol{0};\bx).
\end{equation}
Proceeding with the asymptotic analysis, the $O(\eps^3)$ contribution stemming from \eqref{eq:A1} is found as 
{ 
\begin{multline}\label{w3}
-\int_{W_{\!\boldsymbol{I}}}      \tilde{w}_3 \nabla \cdot \big(G  \overline{\nabla\phi} \big) \ind \bx
+
\int_{W_{\!\boldsymbol{I}}}  \big( G   i  \hat{\bk}  \tilde{w}_2  \big) \cdot   \overline{\nabla\phi} \ind \bx -\int_{W_{\!\boldsymbol{I}}}  \big[G \big(\nabla \tilde{w}_2 + i\hat{\bk}  \tilde{w}_1\big)  \cdot i  \hat{\bk}\big]  \overline{   \phi} \ind \bx \\
=   \int_{W_{\!\boldsymbol{I}}}  \rho   \big[ \hat{\omega}_0^2 \tilde{w}_3+  \hat{\omega}_2^2  \tilde{w}_1 \big] \overline{\phi} \ind \bx, \quad \forall \phi \in H^2_{\boldsymbol{0}}(W_{\!\boldsymbol{I}}).
\end{multline}
}
In a manner similar to earlier treatment, taking $\phi(\bx)=\phi_p(\boldsymbol{0}; \bx)$ we find  
\begin{multline*}
\int_{W_{\!\boldsymbol{I}}}  \big[G   i  \hat{\bk}  \tilde{w}_2  \big] \cdot   \overline{\nabla\phi_p(\boldsymbol{0}; \cdot)} \ind \bx -\int_{W_{\!\boldsymbol{I}}}  \big[G \big(\nabla \tilde{w}_2  + i\hat{\bk}  \tilde{w}_1 \big)  \cdot i  \hat{\bk}\big]  \overline{  \phi_p(\boldsymbol{0}; \cdot)} \ind \bx ~~~~ \nonumber\\
=\; \hat{\omega}_2^2  \int_{W_{\!\boldsymbol{I}}}  \rho  \tilde{w}_1(\cdot) \overline{\phi_p(\boldsymbol{0}; \cdot)} \ind \bx.
\end{multline*}
On substituting \eqref{wexp eigenfunction w1}--\eqref{wexp eigenfunction w2} into the above equation, all the terms involving $\bw_2$ vanish due to the fact that $\phi_p(\boldsymbol{0}; \cdot)$ has a constant phase. In this way, we obtain 
\begin{eqnarray} \label{ComparisonSecondOrderHomogenizationFirstOrderPDE-1}
\big( \bmu^{\mbox{\tiny{(1)}}}\!: (i \hat{\bk})^{3}  + {\brho}^{\mbox{\tiny{(1)}}} \!\cdot\! i\hat{\bk} ~\hat{\omega}^2_2 \big)w_0 \,+\, \big( {\bmu}^{\mbox{\tiny{(0)}}}\!: (i \hat{\bk})^{2}+ \rho^{\mbox{\tiny{(0)}}} \hat{\omega}^2_2 \big) \bw_1\cdot i\hat{\bk} ~=\: 0,
\end{eqnarray}
where 
\begin{eqnarray}  
\brho^{\mbox{\tiny{(1)}}} &=&  {\alpha_p}\es \langle \rho   \bchi^{\mbox{\tiny{(1)}}} \overline{\phi_p(\boldsymbol{0}; \cdot) } \rangle =\boldsymbol{0}, \label{SecondOrderHomogenizationFirstOrderCoefficients}\\
  {\bmu}^{\mbox{\tiny{(1)}}} &=& {\alpha_p}\es \big \langle \{G( \nabla  \bchi^{\mbox{\tiny{(2)}}}   + \boldsymbol{I}  \otimes \bchi^{\mbox{\tiny{(1)}}}   ) \overline{ \phi_p(\boldsymbol{0}; \cdot) } \} -
  \{G \bchi^{\mbox{\tiny{(2)}}}\otimes \overline{\nabla \phi_p(\boldsymbol{0}; \cdot) }\} \big \rangle. \label{SecondOrderHomogenizationFirstOrderCoefficients1}
\end{eqnarray}
\begin{remark} \label{rem3}
It can be shown that $  {\bmu}^{\mbox{\tiny{(1)}}} = \boldsymbol{0}$. Indeed, \eqref{ComparisonEqnchi1} with $\phi$ taken as the $k\ell$th component of $\overline{\bchi^{\mbox{\tiny{(2)}}}}$ yields
\begin{multline}  
\int_{W_{\!\boldsymbol{I}}}  \big( G (\nabla \bchi^{\mbox{\tiny{(1)}}} \!+\boldsymbol{I}\phi_p(\boldsymbol{0}; \cdot)\big) :  {\nabla\bchi^{\mbox{\tiny{(2)}}}_{k\ell} } \ind \bx - \int_{W_{\!\boldsymbol{I}}} \big(  G\boldsymbol{I} \nabla \phi_p(\boldsymbol{0}; \cdot)  \big) \bchi^{\mbox{\tiny{(2)}}}_{k\ell} \ind \bx  \\
 =\; \hat{\omega}_0^2 \int_{W_{\!\boldsymbol{I}}}  \rho\bchi^{\mbox{\tiny{(1)}}}    \bchi^{\mbox{\tiny{(2)}}}_{k\ell}  \ind \bx,  \label{mu1 vanish remark eqn 1}
\end{multline}
while \eqref{Comparisonchi2} with $\phi$ being the $j$th component of $\overline{\bchi^{\mbox{\tiny{(1)}}}_j}$ gives
\begin{eqnarray} 
&&\int_{W_{\!\boldsymbol{I}}}  \big( G (\nabla \bchi^{\mbox{\tiny{(2)}}} \!+ \{\boldsymbol{I} \otimes \bchi^{\mbox{\tiny{(1)}}}\}' \big) :  {\nabla \bchi^{\mbox{\tiny{(1)}}}_j} \ind \bx - \int_{W_{\!\boldsymbol{I}}} G(\{\nabla  \bchi^{\mbox{\tiny{(1)}}}\} + \boldsymbol{I} \phi_p(\boldsymbol{0};\cdot))  {\bchi^{\mbox{\tiny{(1)}}}_j} \ind \bx \nonumber \\
&& + \int_{W_{\!\boldsymbol{I}}} \frac{\rho}{\rho^{\mbox{\tiny{(0)}}}} \bmu^{\mbox{\tiny{(0)}}}  \phi_p(\boldsymbol{0};\cdot)  {\bchi^{\mbox{\tiny{(1)}}}_j} \ind \bx= \hat{\omega}_0^2 \int_{W_{\!\boldsymbol{I}}}  \rho\bchi^{\mbox{\tiny{(2)}}}     { \bchi^{\mbox{\tiny{(1)}}}_j}\ind \bx. \label{mu1 vanish remark eqn 2}
\end{eqnarray}
Taking the difference of~\eqref{mu1 vanish remark eqn 1} and~\eqref{mu1 vanish remark eqn 2}, we obtain 
\begin{multline*} 
 \int_{W_{\!\boldsymbol{I}}}  \big[ \{G( \nabla  \bchi^{\mbox{\tiny{(2)}}}   + \boldsymbol{I}  \otimes \bchi^{\mbox{\tiny{(1)}}}   )  { \phi_p(\boldsymbol{0}; \cdot) } \}- 
  \{G \bchi^{\mbox{\tiny{(2)}}}\otimes  {\nabla \phi_p(\boldsymbol{0}; \cdot) }\} \big] \ind \bx  \nonumber \\
  =\: \frac{1}{\rho^{\mbox{\tiny{(0)}}}} \Big\{ \big(\int_{W_{\!\boldsymbol{I}}} \rho   \bchi^{\mbox{\tiny{(1)}}}  {\phi_p(\boldsymbol{0}; \cdot) } \ind \bx \big)\otimes  {\bmu^{\mbox{\tiny{(0)}}}} \Big\}. 
\end{multline*}
Since $ {\phi_p(\boldsymbol{0}; \cdot) }$ has a constant phase, the right-hand side of the above equation {is proportional to} $\brho^{\mbox{\tiny{(1)}}}=\boldsymbol{0}$ and thus vanishes. By the same (constant-phase) argument, the left-hand side is a constant multiplication of ${\bmu}^{\mbox{\tiny{(1)}}} $ and hence
\begin{equation}\label{w1zero}
  {\bmu}^{\mbox{\tiny{(1)}}} =\boldsymbol{0}.
\end{equation}
\end{remark}

From Remark~\ref{rem3} and~\eqref{ComparisonSecondOrderHomogenizationFirstOrderPDE-1}, we obtain the balance statement 
\begin{eqnarray} \label{ComparisonSecondOrderHomogenizationFirstOrderPDE}
\big( {\bmu}^{\mbox{\tiny{(0)}}}\!: (i \hat{\bk})^{2}+ \rho^{\mbox{\tiny{(0)}}} \hat{\omega}^2_2 \big) \bw_1\cdot i\hat{\bk} ~=\: 0,
\end{eqnarray}
that is satisfied by constant vector $\bw_1$ thanks to~\eqref{ComparisonSecondOrderHomogenizationZerothOrderPDE}.

\subsection{Second-order corrector}  \label{Asymptotic eigenfunction dispersion relation subsection 3}

Let $\bchi^{\mbox{\tiny{(3)}}}\in (H^1_{\boldsymbol{0}}(W_{\!\boldsymbol{I}}))^{d\times d \times d}$ be the \emph{zero-mean} (i.e. $\int_{W_{\!\boldsymbol{I}}} \rho \bchi^{\mbox{\tiny{(3)}}} \overline{\phi_p(\boldsymbol{0}; \cdot)} \ind\bx=\boldsymbol{0}$)  tensor-valued function  satisfying 
\begin{multline}   \label{Comparisonchi3}
\int_{W_{\!\boldsymbol{I}}}  \big( G (\nabla \bchi^{\mbox{\tiny{(3)}}} \!+ \{\boldsymbol{I} \otimes \bchi^{\mbox{\tiny{(2)}}}\}' \big) : \overline{\nabla\phi} \ind \bx - \int_{W_{\!\boldsymbol{I}}} G\big(\{\nabla  \bchi^{\mbox{\tiny{(2)}}}\} +\{\boldsymbol{I} \otimes \bchi^{\mbox{\tiny{(1)}}}\} \big)\overline{\phi} \ind \bx \\
 + \int_{W_{\!\boldsymbol{I}}} \frac{1}{\rho^{\mbox{\tiny{(0)}}}}\{\rho \bchi^{\mbox{\tiny{(1)}}} \otimes \bmu^{\mbox{\tiny{(0)}}}\}\overline{\phi} \ind \bx \:=\: \hat{\omega}_0^2 \int_{W_{\!\boldsymbol{I}}}  \rho\bchi^{\mbox{\tiny{(3)}}}    \overline{ \phi} \ind \bx,  \quad \forall \phi \in H^1_{\boldsymbol{0}}(W_{\!\boldsymbol{I}}).
\end{multline}
{ It is directly verified that the above equation is uniquely solvable since one can again check that, the terms which do not involve $\bchi^{\mbox{\tiny{(3)}}}$ vanish when $\phi = \phi_p(\boldsymbol{0}; \cdot)$ (thanks to \eqref{SecondOrderHomogenizationFirstOrderCoefficients},\eqref{SecondOrderHomogenizationFirstOrderCoefficients1}, and \eqref{w1zero}), i.e. the Fredholm alternative holds.} With such definitions, a direct calculation yields that equation \eqref{w3} can be solved by 
\begin{eqnarray} \label{derivation wexp eigenfunction w3}
\tilde{w}_3 (\bx) &=& \bchi^{\mbox{\tiny{(3)}}} (\bx) : (i \hat{\bk})^{3}\hh  w_0 \,+\,  \big(\bchi^{\mbox{\tiny{(2)}}} (\bx) : (i \hat{\bk})^{2} \big)\hh  \bw_1\cdot i \hat{\bk}  \nonumber\\ 
&&\,+\, \big(\bchi^{\mbox{\tiny{(1)}}} (\bx) \cdot i\hat{\bk}\big)   \big(\bw_2 \cdot i \hat{\bk}^2 \big)   
\,+\, \bw_3 :  (i \hat{\bk})^3 \phi_p(\boldsymbol{0};\cdot).
\end{eqnarray}

To complete the second-order analysis, one must consider the $O(\eps^4)$ contribution to~\eqref{eq:A1} which reads 
{ 
\begin{multline*} 
-\int_{W_{\!\boldsymbol{I}}}     \tilde{w}_4  \nabla \cdot \big(G  \overline{\nabla\phi}\big) \ind \bx 
+
\int_{W_{\!\boldsymbol{I}}}  \big(G   i  \hat{\bk}  \tilde{w}_3 \big) \cdot   \overline{\nabla\phi} \ind \bx 
-\int_{W_{\!\boldsymbol{I}}}  \big[G \big(\nabla \tilde{w}_3 + i\hat{\bk}  \tilde{w}_2\big)  \cdot i  \hat{\bk}\big]  \overline{   \phi} \ind \bx ~~~~ \nonumber\\
=\, \int_{W_{\!\boldsymbol{I}}}  \rho   \big[ \hat{\omega}_0^2 \tilde{w}_4+ \hat{\omega}_2^2 \tilde{w}_2+  \hat{\omega}_4^2  \tilde{w}_0 \big] \overline{\phi} \ind \bx, \quad \forall \phi \in H^2_{\boldsymbol{0}}(W_{\!\boldsymbol{I}}).
\end{multline*}
}
Taking $\phi(\bx)=\phi_p(\boldsymbol{0}; \bx)$ in the above equation yields 
\begin{multline*}
\int_{W_{\!\boldsymbol{I}}}  \big[G   i  \hat{\bk}  \tilde{w}_3  \big] \cdot   \overline{\nabla\phi_p(\boldsymbol{0}; \cdot)} \ind \bx -\int_{W_{\!\boldsymbol{I}}}  \big[G \big(\nabla \tilde{w}_3  + i\hat{\bk}  \tilde{w}_2 \big)  \cdot i  \hat{\bk}\big]  \overline{  \phi_p(\boldsymbol{0}; \cdot)} \ind \bx ~~~~ \nonumber\\
=\, \int_{W_{\!\boldsymbol{I}}}   \rho   \big[  \hat{\omega}_2^2 \tilde{w}_2+  \hat{\omega}_4^2  \tilde{w}_0 \big]\overline{\phi_p(\boldsymbol{0}; \cdot)} \ind \bx.
\end{multline*}
{ On inserting~\eqref{derivation wexp eigenfunction w0},  \eqref{derivation wexp eigenfunction w2}, and \eqref{derivation wexp eigenfunction w3} }into the above equation, we find that all the terms involving $\bw_1$ must vanish because  $\phi_p(\boldsymbol{0}; \cdot)$ has a constant phase, while all the terms containing $\bw_3$ are necessarily trivial due to  \eqref{SecondOrderHomogenizationFirstOrderCoefficients} and \eqref{w1zero}. As a result, we obtain 
\begin{equation} \label{ComparisonSecondOrderHomogenizationSecondOrderPDE}
\big( {\bmu}^{\mbox{\tiny{(2)}}} \!: (i\hat{\bk})^4 + {\brho}^{\mbox{\tiny{(2)}}} \!: (i\hat{\bk})^2  ~  \hat{\omega}^2_2  + \rho^{\mbox{\tiny{(0)}}} \hat{\omega}_4^2 \big) w_0 +
\big( {\bmu}^{\mbox{\tiny{(0)}}} \!: (i\hat{\bk})^2 + \rho^{\mbox{\tiny{(0)}}} \hat{\omega}_2^2 \big) \bw_2 : (i\hat{\bk})^2  =0,~~~~
\end{equation} 
where 
\begin{eqnarray} 
\brho^{\mbox{\tiny{(2)}}} &=&  {\alpha_p}\es   \langle \rho   \bchi^{\mbox{\tiny{(2)}}} \overline{\phi_p(\boldsymbol{0}; \cdot) } \rangle=\boldsymbol{0}, \label{SecondOrderHomogenizationSecondOrderCoefficients-0}\\
  {\bmu}^{\mbox{\tiny{(2)}}} &=& {\alpha_p}\es  \big\langle \{G( \nabla  \bchi^{\mbox{\tiny{(3)}}}   + \boldsymbol{I}  \otimes \bchi^{\mbox{\tiny{(2)}}}   ) \overline{ \phi_p(\boldsymbol{0}; \cdot) } \} -
  \{G \bchi^{\mbox{\tiny{(3)}}}\otimes \overline{\nabla \phi_p(\boldsymbol{0}; \cdot) }\} \big\rangle.  \label{SecondOrderHomogenizationSecondOrderCoefficients}
\end{eqnarray}
From~\eqref{ComparisonSecondOrderHomogenizationSecondOrderPDE} and~\eqref{ComparisonSecondOrderHomogenizationZerothOrderPDE}, we also find that 
\begin{eqnarray*}
{\bmu}^{\mbox{\tiny{(2)}}} \!: (i\hat{\bk})^4 + \rho^{\mbox{\tiny{(0)}}} \hat{\omega}_4^2 =0. 
\end{eqnarray*}

\subsection{Proof of Lemma \ref{asymptotic main phi_0}}
Now we are ready to prove Lemma \ref{asymptotic main phi_0}.

\begin{proof} \label{Asymptotic eigenfunction dispersion relation subsection 4}
{    According to \cite{wilcox1978theory}, $\phi_p({\bk};\cdot)$ and $\omega_p^2({\bk})$ are holomorphic in $\bk$ except a null set where the multiplicity of $\omega_p^2({\bk})$ changes (see also \cite[pp 341]{lions2011asymptotic}).}
{  Since $\omega_p^2({\boldsymbol{0}})$ under our consideration is a simple eigenvalue (such that its multiplicity does not change in some small neighborhood of $\omega_p^2({\boldsymbol{0}})$, and accordingly $\phi_p({\bk};\cdot)$ and $\omega_p^2({\bk})$ are holomorphic in $\bk$ in such a small neighborhood) have convergent Taylor series, and the convergence for the Taylor series of $\phi_p({\bk};\cdot)$ is in $C(\overline{W_{\boldsymbol{I}}})$ \cite{wilcox1978theory} and hence in $L^2(W_{\boldsymbol{I}})$. As before, we let $\bk=\eps \hat{\bk}$ where $\eps$ is sufficiently small and $\hat{\bk}=O(1)$. Since $\omega_p({\boldsymbol{0}})$ is simple by premise, one easily verifies that $\omega_p(\eps \hat{\bk}) = \omega_p(-\eps \hat{\bk})$ for sufficiently small $\eps$; then $\frac{\partial^q}{\partial {\eps^q}}\omega_0^2(\eps \hat{\bk})=0$ for any odd $q$. As a result, we obtain the series expansion as in equations \eqref{wexp 1} and~\eqref{wexp 2}.} 

Now we proceed with the results in Sections~\ref{Asymptotic eigenfunction dispersion relation subsection 1}--\ref{Asymptotic eigenfunction dispersion relation subsection 3}. It directly follows { that~\eqref{wexp eigenfunction w0}--\eqref{wexp eigenfunction w3} apply with~$w_0=1$, and}  
\begin{eqnarray} 
\hat{\omega}_0^2=\omega_p^2(\boldsymbol{0}),\hspace{0.5cm}\hat{\omega}_2^2 &=&  -\frac{{\bmu}^{\mbox{\tiny{(0)}}}}{\rho^{\mbox{\tiny{(0)}}}}\!: (i\hat{\bk})^2, \hspace{0.5cm}\hat{\omega}_4^2 =  -\frac{{\bmu}^{\mbox{\tiny{(2)}}}}{\rho^{\mbox{\tiny{(0)}}}}\!: (i\hat{\bk})^4. 
\end{eqnarray}

Now we {demonstrate~\eqref{necessary w1} and~\eqref{necessary w2}} by the normalization of eigenfunctions. Specifically, the normalization of $\phi_p(\eps\hat{\bk};\bx)$ yields
\begin{eqnarray*}
\int_{\mathbb{T}^d} \rho(\bx) \phi_p(\eps\hat{\bk};\bx) \overline{\phi_p(\eps\hat{\bk};\bx)} \ind \bx \:=\: 1.
\end{eqnarray*}
From the asymptotic expansion of $\phi_p$ according to~\eqref{wexp 1}, we find that the $O(\eps)$ term gives
\begin{eqnarray*}
i \hat{\bk} \cdot (w_0 \overline{\bw}_1-\overline{w}_0 \bw_1) \rho^{\mbox{\tiny{(0)}}}=i \hat{\bk}  |w_0|^2   \cdot \big(  {\brho}^{\mbox{\tiny{(1)}}}  - \overline{ {\brho}^{\mbox{\tiny{(1)}}} } \big) = 0,
\end{eqnarray*}
and hence
\begin{eqnarray*}
i \hat{\bk} \cdot ( \overline{\bw}_1- \bw_1)=0.
\end{eqnarray*}
Similarly, the $O(\eps^2)$ term gives
\begin{eqnarray*}
 (i\hat{\bk})^2:\big[ (\bw_2  + \overline{\bw}_2  ) - (\bw_1 \otimes \overline{\bw}_1) -     \langle \rho \bchi^{(1)} \otimes \overline{\bchi^{(1)}} \rangle  \big]=0, 
\end{eqnarray*}
which completes the proof.
\end{proof}
\subsection{Property}
The following Proposition is needed in the derivation of the main result in the next section.
\begin{proposition} \label{Prop Section Higher order eqn 1}
The following holds 
\begin{eqnarray}
&& \Big\|\langle \rho {\phi_p(\boldsymbol{0}; \cdot)} \overline{\phi_p( \eps \hat{\bk};{\cdot})} \rangle \phi_p(\eps \hat{\bk};\cdot) \nonumber\\
&&\:  -\Big[\phi_p(\boldsymbol{0}; \cdot)  + \eps   i \hat{\bk}  \cdot \bchi^{\mbox{\tiny{(1)}}}  + \eps^2 (i \hat{\bk})^2 : \big(   \langle \rho \bchi^{(1)}(\cdot) \otimes \overline{\bchi^{(1)}} \rangle  \phi_p(\boldsymbol{0}; \cdot)+  \bchi^{\mbox{\tiny{(2)}}}(\cdot) \big) \Big] \Big\|_{L^2(W_{\boldsymbol{I}})} \nonumber \\
&=& O(\epsilon^3) \label{Section Higher order eqn 1}
\end{eqnarray}
\end{proposition}
\begin{proof}
From the asymptotic expansion of $\phi_p(\eps \hat{\bk};\cdot)$ in { Lemma~\ref{asymptotic main phi_0}}, we can directly obtain 
\begin{eqnarray}
&&\Big\|  \phi_p(\eps \hat{\bk};\cdot) - \Big[ w_0\phi_p(\boldsymbol{0}; \cdot) 
+  \eps \bchi^{\mbox{\tiny{(1)}}}   \cdot i \hat{\bk}\hh w_0 + \eps \bw_1 \cdot i\hat{\bk}\phi_p(\boldsymbol{0}; \cdot) + \eps^2  \bchi^{\mbox{\tiny{(2)}}}   \!:\! (i \hat{\bk})^{2}\hh  w_0 \, \nonumber\\
&& +\, \eps^2\bchi^{\mbox{\tiny{(1)}}}  \otimes \bw_1 \!:\! (i \hat{\bk})^2   \, +\, \eps^2\bw_2 \!:\!  (i \hat{\bk})^2 \phi_p(\boldsymbol{0}; \cdot) \Big] \Big\|_{L^2(W_{\boldsymbol{I}})} = O(\epsilon^3), \label{Section Higher order eqn 01}\\
&&\Big| \langle \rho {\phi_p(\boldsymbol{0}; \cdot)} \overline{\phi_p( \eps \hat{\bk};{\cdot})} \rangle-\langle \rho {\phi_p(\boldsymbol{0}; \cdot)} \overline{\phi_p( \boldsymbol{0};{\cdot})} \rangle   \big( \overline{w}_0 + \eps \overline{\bw_1 \cdot i \hat{\bk}} + \eps^2 \overline{\bw_2 \cdot (i \hat{\bk})^2 } \big)  \Big|\nonumber\\
&& = O(\epsilon^3).\label{Section Higher order eqn 02}
\end{eqnarray}

Therefore we can derive that
\begin{eqnarray*}
&&\bigg\| \langle \rho {\phi_p(\boldsymbol{0}; \cdot)} \overline{\phi_p( \eps \hat{\bk};{\cdot})} \rangle \phi_p(\eps \hat{\bk};\cdot) -\Big[\langle \rho {\phi_p(\boldsymbol{0}; \cdot)} \overline{\phi_p( \boldsymbol{0};{\cdot})} \rangle   \big( \overline{w}_0 + \eps \overline{\bw_1 \cdot i \hat{\bk}} + \eps^2 \overline{\bw_2 \cdot (i \hat{\bk})^2 } \big) \Big] \\
&&\,\,\cdot \Big[ w_0\phi_p(\boldsymbol{0}; \cdot) +  \eps \bchi^{\mbox{\tiny{(1)}}}   \cdot i \hat{\bk}\hh w_0 + \eps \bw_1 \cdot i\hat{\bk}\phi_p(\boldsymbol{0}; \cdot) \\
&&\hspace{1cm}+ \eps^2  \bchi^{\mbox{\tiny{(2)}}}   \!:\! (i \hat{\bk})^{2}\hh  w_0 \,+\, \eps^2\bchi^{\mbox{\tiny{(1)}}}  \otimes \bw_1 \!:\! (i \hat{\bk})^2   \, +\, \eps^2\bw_2 \!:\!  (i \hat{\bk})^2 \phi_p(\boldsymbol{0}; \cdot) \Big] \bigg\|_{L^2(W_{\boldsymbol{I}})} = O(\epsilon^3).
\end{eqnarray*}
The complicated term (second one inside $\|\cdot\|$) in the above equation can be simplified by a direct calculation as follows
\begin{multline*}
\Big[\langle \rho {\phi_p(\boldsymbol{0}; \cdot)} \overline{\phi_p( \boldsymbol{0};{\cdot})} \rangle   \big( \overline{w}_0 + \eps \overline{\bw_1 \cdot i \hat{\bk}} + \eps^2 \overline{\bw_2 \cdot (i \hat{\bk})^2 } \big) \Big] \cdot \Big[ w_0\phi_p(\boldsymbol{0}; \cdot) \\
+  \eps \bchi^{\mbox{\tiny{(1)}}}   \cdot i \hat{\bk}\hh w_0 + \eps \bw_1 \cdot i\hat{\bk}\phi_p(\boldsymbol{0}; \cdot) 
+ \eps^2  \bchi^{\mbox{\tiny{(2)}}}   \!:\! (i \hat{\bk})^{2}\hh  w_0 \,+\, \eps^2\bchi^{\mbox{\tiny{(1)}}}  \otimes \bw_1 \!:\! (i \hat{\bk})^2   \, +\, \eps^2\bw_2 \!:\!  (i \hat{\bk})^2 \phi_p(\boldsymbol{0}; \cdot) \Big] \\
=  |w_0|^2 \phi_p(\boldsymbol{0}; \cdot)  + \eps    \Big( |w_0|^2\bchi^{\mbox{\tiny{(1)}}}  \cdot  i \hat{\bk} +  (-w_0 \overline{\bw}_1 +\overline{w}_0 \bw_1)\cdot i \hat{\bk}  \phi_p(\boldsymbol{0}; \cdot) \Big) \\
+ \eps^2 (i \hat{\bk})^2 \!:\! \Big(   (\bw_2 \overline{w}_0 + \overline{\bw}_2 w_0  )\phi_p(\boldsymbol{0}; \cdot)-\bw_1 \otimes \overline{\bw}_1 \phi_p(\boldsymbol{0}; \cdot)+ (-w_0 \overline{\bw}_1 +\overline{w}_0 \bw_1)\otimes  \bchi^{\mbox{\tiny{(1)}}}  + |w_0|^2  \bchi^{\mbox{\tiny{(2)}}} \Big).
\end{multline*}
Using the necessary conditions \eqref{necessary w0} -- \eqref{necessary w2}, we can simplify the above equation to prove equation \eqref{Section Higher order eqn 1}.
\end{proof}

\section{Main result of the second-order homogenization} \label{Higher order U}

Now we are ready to derive the asymptotic model for~$U_\epsilon$. 
\begin{theorem} \label{finalthm}
Let $U_{\epsilon} $ be the solution to \eqref{PDE fast}, where the source term $f_\epsilon$ is given by Assumption \ref{assumption f}, and let the driving frequency ${ \Omega_{\epsilon}}$ reside within a band gap according to~\eqref{bg2}--\eqref{bg3}. Then we have
\begin{eqnarray*}
\|{  U_{\epsilon} }-U_0\|_{L^2(\mathbb{R}^d)} = O(\eps),~~\|{  U_{\epsilon} }-U_1\|_{L^2(\mathbb{R}^d)} = O(\eps^2),~~\|{  U_{\epsilon} }-U_2\|_{L^2(\mathbb{R}^d)} = O(\eps^3),
\end{eqnarray*}
where the leading-order $U_0$, first-order $U_1$, and second-order $U_2$ asymptotic approximations  are given respectively by \eqref{Main leading-order U0}, \eqref{Main first-order U1}, and \eqref{Main second-order U2} with
\begin{eqnarray} 
U_0(\br) 
&:=&  \phi_p(\boldsymbol{0}; \br/\eps) \,  W_0({\br}), \label{Main leading-order U0} \\
U_1(\br) &:=& \phi_p(\boldsymbol{0}; \br/\eps) W_0( \br) +  \eps \bchi^{\mbox{\tiny{(1)}}}(\br/\eps)  \cdot \nabla_{\br}W_0(\br).  \label{Main first-order U1}\\
U_2(\br) &:=&   \phi_p(\boldsymbol{0}; \br/\eps) W_2(\br) +   \eps \bchi^{\mbox{\tiny{(1)}}}(\br/\eps)   \cdot \nabla_{\br}W_2(\br) \nonumber \\
&& +  \eps^2 \Big(   \langle \rho \bchi^{(1)} \otimes \overline{\bchi^{(1)}} \rangle  \phi_p(\boldsymbol{0}; \br/\eps)+  \bchi^{\mbox{\tiny{(2)}}}(\br/\eps) \Big) : \nabla^2_{\br}W_2(\br), \label{Main second-order U2}
\end{eqnarray}
and $W_0$ and $W_2$ are given by \eqref{Main W0} and \eqref{Main W2} respectively,
\begin{eqnarray}
&&W_0(\br) := \frac{1}{(2\pi)^{d/2}} \int_{\mathbb{R}^d}  \frac{     F(\hat{\bk}) e^{i \hat{\bk} \cdot   \br}  }{-\frac{{\bmu}^{\mbox{\tiny{(0)}}}}{\rho^{\mbox{\tiny{(0)}}}}\!: (i\hat{\bk})^2  -{ \sigma}{\hat{\Omega}}^2}  \ind \hat{\bk}, \label{Main W0}\\
&&W_2(\br) := \frac{1}{(2\pi)^{d/2}}  
\int_{\mathbb{R}^d}    \frac{    F(\hat{\bk}) e^{i  \hat{\bk} \cdot   \br}  }{-\frac{{\bmu}^{\mbox{\tiny{(0)}}}}{\rho^{\mbox{\tiny{(0)}}}}\!: (i\hat{\bk})^2 -\eps^2 \frac{{\bmu}^{\mbox{\tiny{(2)}}}}{\rho^{\mbox{\tiny{(0)}}}}\!: (i\hat{\bk})^4 -{ \sigma}{\hat{\Omega}}^2}  \ind \hat{\bk}. ~~~~~~~~~~~\label{Main W2}
\end{eqnarray}
Here $(\bk,\omega_p(\bk))$ is the $p$th branch of the dispersion relationship, and $\phi_p(\boldsymbol{\bk}; \cdot)$ is the corresponding eigenfunction given by \eqref{eigenfunction}; the effective coefficients $\rho^{\mbox{\tiny{(0)}}}$ and ${\bmu}^{\mbox{\tiny{(0)}}}$ are given by \eqref{SecondOrderHomogenizationZerothOrderCoefficients-0} and \eqref{SecondOrderHomogenizationZerothOrderCoefficients};  ${\bmu}^{\mbox{\tiny{(2)}}}$ are given by \eqref{SecondOrderHomogenizationSecondOrderCoefficients}, and the tensor-valued functions $ \bchi^{\mbox{\tiny{(1)}}} $,$ \bchi^{\mbox{\tiny{(2)}}} $, and $ \bchi^{\mbox{\tiny{(3)}}} $ are given respectively by \eqref{ComparisonEqnchi1}, \eqref{Comparisonchi2}, and \eqref{Comparisonchi3}.
\end{theorem}
\begin{proof}
Recall from Theorem \ref{Thm uothers no contri} (in particular equation \eqref{Thm u - uothers eqn 2}) that
\begin{eqnarray*}
\|{  U_{\epsilon} } - \Up\|_{L^2(\mathbb{R}^d)} \leqslant {C}\, \eps^{3} \|F\|_{L^2(\mathbb{R}^d)},
\end{eqnarray*}
where $U^{(p)}=\up(\cdot/\epsilon)$ with $\up$ given by~\eqref{Section u drive eqn 1}. Therefore to prove the theorem, it is sufficient to show that
\begin{eqnarray*}
\|U^{(p)}-U_0\|_{L^2(\mathbb{R}^d)} = O(\eps),~~\|U^{(p)}-U_1\|_{L^2(\mathbb{R}^d)} = O(\eps^2),~~\|U^{(p)}-U_2\|_{L^2(\mathbb{R}^d)} = O(\eps^3).
\end{eqnarray*}

We first prove $\|U^{(p)}-U_0\|_{L^2(\mathbb{R}^d)} = O(\eps)$, and then prove the higher order estimates. (We have adjusted the font size of the equations in the following proof for the best presentation since they are too long.)

\textbf{1.}  Proof of $\|U^{(p)}-U_0\|_{L^2(\mathbb{R}^d)} = O(\eps)$: This is equivalent to $\|u^{(p)}-u_0\|_{L^2(\mathbb{R}^d)} = \epsilon^{-d/2}O(\eps)$ where $u_0(\bx):= U_0(\epsilon\bx)$. According to Lemma \ref{Lemma Lemma Bloch expansion}, it is sufficient to show that $\|\mathcal{J}_{\mathbb{R}^d} u^{(p)}-\mathcal{J}_{\mathbb{R}^d} u_0\|_{L^2(W_{\!{\boldsymbol{I}^*}};\mathcal{H}^0_{\bk} (W_{\!\boldsymbol{I}}))} = \epsilon^{-d/2}O(\eps)$.

We first obtain the representation for $\mathcal{J}_{\mathbb{R}^d} u_0$ and $\mathcal{J}_{\mathbb{R}^d} u^{(p)}$. Indeed, from the expression \eqref{Main leading-order U0} of $U_0$ with $W_0$ defined by \eqref{Main W0}, we obtain (by noting that $F$ is compactly supported in $Y$) that
{\footnotesize
\begin{eqnarray}
u_0(\bx) &=& U_0(\epsilon\bx) = \frac{1}{(2\pi)^{d/2}} \int_{Y}    \frac{   \phi_p(\boldsymbol{0}; \bx)  F(\hat{\bk}) e^{i\eps \hat{\bk} \cdot   \bx}  }{-\frac{{\bmu}^{\mbox{\tiny{(0)}}}}{\rho^{\mbox{\tiny{(0)}}}}\!: (i\hat{\bk})^2  -\sigma{\hat{\Omega}}^2}  \ind \hat{\bk} \nonumber \\
&=& \frac{\epsilon^{2-d}}{(2\pi)^{d/2}} \int_{\epsilon Y}    \frac{   \phi_p(\boldsymbol{0}; \bx)  F({\bk}/\epsilon) e^{i{\bk} \cdot   \bx}  }{-\frac{{\bmu}^{\mbox{\tiny{(0)}}}}{\rho^{\mbox{\tiny{(0)}}}}\!: (i{\bk})^2  -\epsilon^2 \sigma{\hat{\Omega}}^2}  \ind {\bk} \quad(\mbox{using a change of variable } \hat{\bk} \to \bk)\nonumber \\
&=& \phi_p(\boldsymbol{0}; \tilde{\bx})  \frac{\epsilon^{2-d}}{(2\pi)^{d/2}} \int_{ W_{\boldsymbol{I}^*}}    \frac{    F({\bk}/\epsilon) e^{i{\bk} \cdot   \tilde{\bx}}  }{-\frac{{\bmu}^{\mbox{\tiny{(0)}}}}{\rho^{\mbox{\tiny{(0)}}}}\!: (i{\bk})^2  -\epsilon^2 \sigma{\hat{\Omega}}^2}  e^{i{\bk} \cdot   \bj} ~\ind \bk \quad {(\tiny\mbox{$\bj$ is such that $\tilde{\bx}=\bx-\bj \in W_{\!{\boldsymbol{I}}}$ } )} \nonumber \\
&=&\epsilon^{2-d} \phi_p(\boldsymbol{0}; \tilde{\bx})   \mathcal{J}_{\mathbb{R}^d}^{-1} [ \frac{    F({\bk}/\epsilon) e^{i{\bk} \cdot \tilde{\bx}   }  }{-\frac{{\bmu}^{\mbox{\tiny{(0)}}}}{\rho^{\mbox{\tiny{(0)}}}}\!: (i{\bk})^2  -\epsilon^2 \sigma{\hat{\Omega}}^2}]. \nonumber 
\end{eqnarray}
}
Therefore we obtain that
{\footnotesize
\begin{eqnarray} 
(\mathcal{J}_{\mathbb{R}^d} u_0)(\bk;\tilde{\bx}) = \epsilon^{2-d} \phi_p(\boldsymbol{0}; \tilde{\bx}) \frac{    F({\bk}/\epsilon) e^{i{\bk} \cdot \tilde{\bx}   }  }{-\frac{{\bmu}^{\mbox{\tiny{(0)}}}}{\rho^{\mbox{\tiny{(0)}}}}\!: (i{\bk})^2  -\epsilon^2 \sigma{\hat{\Omega}}^2}, \quad \bk \in W_{\boldsymbol{I}^*}, \tilde{\bx} \in W_{\boldsymbol{I}}, \label{Main theorem proof Ju0}
\end{eqnarray}
}and $(\mathcal{J}_{\mathbb{R}^d} u_0)(\bk;\tilde{\bx})$ is compactly supported in $\eps Y$ (in the $\bk$ variable) since $F$ is compactly supported in $Y$.
From the expression \eqref{u(p)} of $\up$, it is seen that 
{\footnotesize
\begin{eqnarray}
(\mathcal{J}_{\mathbb{R}^d} \up)(\bk;\tilde{\bx}) &=&  \frac{\eps^2   \langle  \mathcal{J}_{\mathbb{R}^d} \big[{ f_{\epsilon}}(\eps \cdot) \big] ( {\bk}; \cdot), e^{i  {\bk}  \cdot}\phi_p( ;\cdot) \rangle }{\omega^2_p(\bk) -\omega^2} e^{i\bk\cdot   \tilde{\bx}} \phi_p(\bk;\tilde{\bx}),  \label{Main theorem proof Jup}
\end{eqnarray}
}for $ \bk \in W_{\boldsymbol{I}^*}, \tilde{\bx} \in W_{\boldsymbol{I}}$, and that $(\mathcal{J}_{\mathbb{R}^d} \up)(\bk;\tilde{\bx})$ has compact support in $\epsilon Y$ in the $\bk$ variable (since $ \langle  \mathcal{J}_{\mathbb{R}^d} \big[{ f_{\epsilon}}(\eps \cdot) \big] ( {\bk}; \cdot), e^{i  {\bk}  \cdot}\phi_p( ;\cdot) \rangle$ has compact support in $\epsilon Y$ according to Proposition \ref{Prop f Y to f Rd}).

Now from equations \eqref{Main theorem proof Ju0}--\eqref{Main theorem proof Jup}, we obtain (noting both terms are compactly supported in $\eps Y$) that
{\footnotesize
\begin{eqnarray}
&&\|\mathcal{J}_{\mathbb{R}^d} u^{(p)}-\mathcal{J}_{\mathbb{R}^d} u_0\|^2_{L^2(W_{\!{\boldsymbol{I}^*}};\mathcal{H}^0_{\bk} (W_{\!\boldsymbol{I}}))} = \int_{\eps Y} \int_{W_{\!\boldsymbol{I}}} \Big| \frac{\eps^2   \langle  \mathcal{J}_{\mathbb{R}^d} \big[{ f_{\epsilon}}(\eps \cdot) \big] ( {\bk}; \cdot), e^{i  {\bk}  \cdot}\phi_p( ;\cdot) \rangle }{\omega^2_p(\bk) -\omega^2} 
\nonumber \\
&&\hspace{2cm}\times e^{i\bk\cdot   \tilde{\bx}} \phi_p(\bk;\tilde{\bx}) - \epsilon^{2-d} \phi_p(\boldsymbol{0}; \tilde{\bx}) \frac{    F({\bk}/\epsilon) e^{i{\bk} \cdot \tilde{\bx}   }  }{-\frac{{\bmu}^{\mbox{\tiny{(0)}}}}{\rho^{\mbox{\tiny{(0)}}}}\!: (i{\bk})^2  -\epsilon^2 \sigma{\hat{\Omega}}^2}\Big|^2 \ind \tilde{\bx} \ind \bk.  \label{Main theorem proof Jup-Ju0 norm}
\end{eqnarray}
}
We further simplify the first term in the above integrand as
{\footnotesize
\begin{eqnarray}
&&\frac{\eps^2   \langle  \mathcal{J}_{\mathbb{R}^d} \big[{ f_{\epsilon}}(\eps \cdot) \big] ( {\bk}; \cdot), e^{i  {\bk}  \cdot}\phi_p( ;\cdot) \rangle }{\omega^2_p(\bk) -\omega^2} e^{i\bk\cdot   \tilde{\bx}} \phi_p(\bk;\tilde{\bx}) \nonumber \\
&=& \frac{\eps^2   \frac{1}{(2\pi)^{d/2}}  \int_{\mathbb{R}^d}     { f_{\epsilon}}(\eps \by) e^{-i  {\bk} \cdot \by} \overline{ \phi_p( {\bk};\by)}\ind \by }{\omega^2_p(\bk) -\omega^2} e^{i\bk\cdot   \tilde{\bx}} \phi_p(\bk;\tilde{\bx}) \qquad (\mbox{using } \eqref{f Y to f Rd})\nonumber \\
&=& \eps^{2-d}     \frac{\langle \rho {\phi_p(\boldsymbol{0}; \cdot)} \overline{\phi_p( {\bk};{\cdot})} \rangle}{\omega^2_p(\bk) -\omega^2}  F(\bk/\eps) e^{i\bk\cdot   \tilde{\bx}} \phi_p(\bk;\tilde{\bx})  \quad (\mbox{using Proposition }\eqref{almost orthogonality}).\label{Main theorem proof Jup-Ju0 first term}
\end{eqnarray}
}
From equation \eqref{Main theorem proof Jup-Ju0 norm} and \eqref{Main theorem proof Jup-Ju0 first term} we obtain that
{\footnotesize
\begin{eqnarray} \label{Main theorem proof Jup-Ju0 norm 1}
&&\|\mathcal{J}_{\mathbb{R}^d} u^{(p)}-\mathcal{J}_{\mathbb{R}^d} u_0\|^2_{L^2(W_{\!{\boldsymbol{I}^*}};\mathcal{H}^0_{\bk} (W_{\!\boldsymbol{I}}))} \nonumber \\
&=& \epsilon^{4-2d}  \int_{\eps Y} \int_{W_{\!\boldsymbol{I}}} \bigg|     F(\bk/\eps) e^{i\bk\cdot   \tilde{\bx}} \Big[ \frac{   \langle \rho {\phi_p(\boldsymbol{0}; \cdot)} \overline{\phi_p( {\bk};{\cdot})} \rangle }{\omega^2_p(\bk) -\omega^2} \phi_p(\bk;\tilde{\bx}) - \frac{  1 }{-\frac{{\bmu}^{\mbox{\tiny{(0)}}}}{\rho^{\mbox{\tiny{(0)}}}}\!: (i{\bk})^2  -\epsilon^2 \sigma{\hat{\Omega}}^2} \phi_p(\boldsymbol{0}; \tilde{\bx}) 
\Big]\bigg|^2 \ind \tilde{\bx} \ind \bk \nonumber \\
&=& \epsilon^{-d}  \int_{Y} \int_{W_{\!\boldsymbol{I}}} \bigg|     F(\hat{\bk}) e^{i \epsilon \hat{\bk}\cdot   \tilde{\bx}} \Big[ \frac{   \langle \rho {\phi_p(\boldsymbol{0}; \cdot)} \overline{\phi_p( {\bk};{\cdot})} \rangle }{\big(\omega^2_p(\eps \hat{\bk}) -\omega^2\big)/\epsilon^2} \phi_p(\eps \hat{\bk};\tilde{\bx}) - \frac{  1 }{-\frac{{\bmu}^{\mbox{\tiny{(0)}}}}{\rho^{\mbox{\tiny{(0)}}}}\!: (i{\hat{\bk}})^2  -  \sigma{\hat{\Omega}}^2} \phi_p(\boldsymbol{0}; \tilde{\bx}) 
\Big]\bigg|^2 \ind \tilde{\bx} \ind \hat{\bk} \nonumber \\
&\le& \epsilon^{-d} W_{\!\boldsymbol{I}}  \int_{Y} \Big|F(\hat{\bk}) \Big|^2 \Big \|\frac{   \langle \rho {\phi_p(\boldsymbol{0}; \cdot)} \overline{\phi_p( {\bk};{\cdot})} \rangle \phi_p(\eps \hat{\bk};\tilde{\bx}) }{\big(\omega^2_p(\eps \hat{\bk}) -\omega^2\big)/\epsilon^2} - \frac{  \phi_p(\boldsymbol{0}; \tilde{\bx}) }{-\frac{{\bmu}^{\mbox{\tiny{(0)}}}}{\rho^{\mbox{\tiny{(0)}}}}\!: (i{\hat{\bk}})^2  -  \sigma{\hat{\Omega}}^2}   \Big\|^2_{L^2(W_{\!{\boldsymbol{I}}})}  \ind \hat{\bk}.
\end{eqnarray}
}From Proposition \ref{Prop Section Higher order eqn 1}, the expression of $\omega$ in \eqref{drf1}, and the asymptotic of $\omega^2_p(\eps \hat{\bk})$ in Lemma \ref{asymptotic main phi_0}, we obtain that
{\footnotesize
\begin{eqnarray}
&& \Big \|\frac{   \langle \rho {\phi_p(\boldsymbol{0}; \cdot)} \overline{\phi_p( {\bk};{\cdot})} \rangle \phi_p(\eps \hat{\bk};\tilde{\bx}) }{\big(\omega^2_p(\eps \hat{\bk}) -\omega^2\big)/\epsilon^2} - \frac{  \phi_p(\boldsymbol{0}; \tilde{\bx}) }{-\frac{{\bmu}^{\mbox{\tiny{(0)}}}}{\rho^{\mbox{\tiny{(0)}}}}\!: (i{\hat{\bk}})^2  -  \sigma{\hat{\Omega}}^2}   \Big\|_{L^2(W_{\!{\boldsymbol{I}}})} \nonumber \\
&=&  \Big \|\frac{   \langle \rho {\phi_p(\boldsymbol{0}; \cdot)} \overline{\phi_p( {\bk};{\cdot})} \rangle \phi_p(\eps \hat{\bk};\tilde{\bx})  }{-\frac{{\bmu}^{\mbox{\tiny{(0)}}}}{\rho^{\mbox{\tiny{(0)}}}}\!: (i{\hat{\bk}})^2  -  \sigma{\hat{\Omega}}^2 +O(\epsilon^2)}  - \frac{  \phi_p(\boldsymbol{0}; \tilde{\bx})  }{-\frac{{\bmu}^{\mbox{\tiny{(0)}}}}{\rho^{\mbox{\tiny{(0)}}}}\!: (i{\hat{\bk}})^2  -  \sigma{\hat{\Omega}}^2}  \Big\|_{L^2(W_{\!{\boldsymbol{I}}})} (\mbox{using } \eqref{drf1} \mbox{ and Lemma } \ref{asymptotic main phi_0}) \nonumber \\
&=& \Big|  \frac{  1}{-\frac{{\bmu}^{\mbox{\tiny{(0)}}}}{\rho^{\mbox{\tiny{(0)}}}}\!: (i{\hat{\bk}})^2  -  \sigma{\hat{\Omega}}^2} \Big| \Big \|    \langle \rho {\phi_p(\boldsymbol{0}; \cdot)} \overline{\phi_p( {\bk};{\cdot})} \rangle \phi_p(\eps \hat{\bk};\tilde{\bx})- \phi_p(\boldsymbol{0}; \tilde{\bx})  \Big\|_{L^2(W_{\!{\boldsymbol{I}}})} +O(\epsilon^2) \nonumber \\
&=& O(\epsilon) \qquad ( \mbox{using Proposition }\ref{Prop Section Higher order eqn 1}). \label{Main theorem proof Jup-Ju0 norm 2}
\end{eqnarray}
}
Now from equations \eqref{Main theorem proof Jup-Ju0 norm 1} and \eqref{Main theorem proof Jup-Ju0 norm 2}, we obtain that
{\footnotesize
\begin{eqnarray} \label{Main theorem proof Jup-Ju0 norm 3}
\|\mathcal{J}_{\mathbb{R}^d} u^{(p)}-\mathcal{J}_{\mathbb{R}^d} u_0\|^2_{L^2(W_{\!{\boldsymbol{I}^*}};\mathcal{H}^0_{\bk} (W_{\!\boldsymbol{I}}))}  \le  c \epsilon^{-d} (O(\epsilon) )^2,
\end{eqnarray}
}for some constant $c$ independent of $\epsilon$. Thus this proves that $\|\mathcal{J}_{\mathbb{R}^d} u^{(p)}-\mathcal{J}_{\mathbb{R}^d} u_0\|_{L^2(W_{\!{\boldsymbol{I}^*}};\mathcal{H}^0_{\bk} (W_{\!\boldsymbol{I}}))} = \epsilon^{-d/2}O(\eps)$ and hence $\|u^{(p)}-u_0\|_{L^2(\mathbb{R}^d)} = \epsilon^{-d/2}O(\eps)$, i.e., $\|U^{(p)}-U_0\|_{L^2(\mathbb{R}^d)} = O(\eps)$. This completes the proof in the leading-order case.

\textbf{2.}  Proof of $\|U^{(p)}-U_1\|_{L^2(\mathbb{R}^d)} = O(\eps^2)$: The idea is similar to the proof in the leading-order case. We show below the detailed proof. $\|U^{(p)}-U_1\|_{L^2(\mathbb{R}^d)} = O(\eps^2)$ is equivalent to $\|u^{(p)}-u_1\|_{L^2(\mathbb{R}^d)} = \epsilon^{-d/2}O(\eps^2)$ where $u_1(\bx):= U_1(\epsilon\bx)$. According to Lemma \ref{Lemma Lemma Bloch expansion}, it is sufficient to show that $\|\mathcal{J}_{\mathbb{R}^d} u^{(p)}-\mathcal{J}_{\mathbb{R}^d} u_1\|_{L^2(W_{\!{\boldsymbol{I}^*}};\mathcal{H}^0_{\bk} (W_{\!\boldsymbol{I}}))} = \epsilon^{-d/2}O(\eps^2)$.

We first obtain the representation for $\mathcal{J}_{\mathbb{R}^d} u_1$. Indeed, from the expression \eqref{Main first-order U1} of $U_1$ with $W_0$ defined by \eqref{Main W0}, we obtain (by noting that $F$ is compactly supported in $Y$) that
{\footnotesize
\begin{eqnarray}
&& u_1(\bx) = U_1(\epsilon\bx)  \nonumber \\
&=&\frac{1}{(2\pi)^{d/2}} \int_{Y}    \frac{   \big[\phi_p(\boldsymbol{0}; \bx) + \epsilon \bchi^{\mbox{\tiny{(1)}}}(\bx) \cdot i \hat{\bk} \big] F(\hat{\bk}) e^{i\eps \hat{\bk} \cdot   \bx}  }{-\frac{{\bmu}^{\mbox{\tiny{(0)}}}}{\rho^{\mbox{\tiny{(0)}}}}\!: (i\hat{\bk})^2  -\sigma{\hat{\Omega}}^2}  \ind \hat{\bk} \nonumber \\
&=& \frac{\epsilon^{2-d}}{(2\pi)^{d/2}} \int_{\epsilon Y}    \frac{  \big[\phi_p(\boldsymbol{0}; \bx) +  \bchi^{\mbox{\tiny{(1)}}}(\bx) \cdot i  {\bk} \big] F({\bk}/\epsilon) e^{i{\bk} \cdot   \bx}  }{-\frac{{\bmu}^{\mbox{\tiny{(0)}}}}{\rho^{\mbox{\tiny{(0)}}}}\!: (i{\bk})^2  -\epsilon^2 \sigma{\hat{\Omega}}^2}  \ind {\bk} \nonumber \\
&=& \big[\phi_p(\boldsymbol{0}; \tilde{\bx}) +  \bchi^{\mbox{\tiny{(1)}}}(\tilde{\bx}) \cdot i  {\bk} \big]  \frac{\epsilon^{2-d}}{(2\pi)^{d/2}} \int_{ W_{\boldsymbol{I}^*}}    \frac{    F({\bk}/\epsilon) e^{i{\bk} \cdot   \tilde{\bx}}  }{-\frac{{\bmu}^{\mbox{\tiny{(0)}}}}{\rho^{\mbox{\tiny{(0)}}}}\!: (i{\bk})^2  -\epsilon^2 \sigma{\hat{\Omega}}^2}  e^{i{\bk} \cdot   \bj} ~\ind \bk \quad \nonumber \\
&=&\epsilon^{2-d} \big[\phi_p(\boldsymbol{0}; \tilde{\bx}) + \bchi^{\mbox{\tiny{(1)}}}(\tilde{\bx}) \cdot i  {\bk} \big]  \mathcal{J}_{\mathbb{R}^d}^{-1} [ \frac{    F({\bk}/\epsilon) e^{i{\bk} \cdot \tilde{\bx}   }  }{-\frac{{\bmu}^{\mbox{\tiny{(0)}}}}{\rho^{\mbox{\tiny{(0)}}}}\!: (i{\bk})^2  -\epsilon^2 \sigma{\hat{\Omega}}^2}], \nonumber 
\end{eqnarray}
}
where $\mbox{$\bj$ is such that $\tilde{\bx}=\bx-\bj \in W_{\!{\boldsymbol{I}}}$ }$.
Therefore we obtain that
{\footnotesize
\begin{eqnarray} 
(\mathcal{J}_{\mathbb{R}^d} u_0)(\bk;\tilde{\bx}) = \epsilon^{2-d} \big[\phi_p(\boldsymbol{0}; \tilde{\bx}) + \bchi^{\mbox{\tiny{(1)}}}(\tilde{\bx}) \cdot i  {\bk} \big]   \frac{    F({\bk}/\epsilon) e^{i{\bk} \cdot \tilde{\bx}   }  }{-\frac{{\bmu}^{\mbox{\tiny{(0)}}}}{\rho^{\mbox{\tiny{(0)}}}}\!: (i{\bk})^2  -\epsilon^2 \sigma{\hat{\Omega}}^2}, \quad \bk \in W_{\boldsymbol{I}^*}, \tilde{\bx} \in W_{\boldsymbol{I}}. \label{Main theorem proof Ju0 first-order}
\end{eqnarray}
}Now from equations \eqref{Main theorem proof Ju0 first-order}  and \eqref{Main theorem proof Jup}, we obtain (noting both terms are compactly supported in $\eps Y$) that
{\footnotesize
\begin{eqnarray}
&&\|\mathcal{J}_{\mathbb{R}^d} u^{(p)}-\mathcal{J}_{\mathbb{R}^d} u_1\|^2_{L^2(W_{\!{\boldsymbol{I}^*}};\mathcal{H}^0_{\bk} (W_{\!\boldsymbol{I}}))} \nonumber \\
&=& \int_{\eps Y} \int_{W_{\!\boldsymbol{I}}} \Big| \frac{\eps^2   \langle  \mathcal{J}_{\mathbb{R}^d} \big[{ f_{\epsilon}}(\eps \cdot) \big] ( {\bk}; \cdot), e^{i  {\bk}  \cdot}\phi_p( ;\cdot) \rangle }{\omega^2_p(\bk) -\omega^2} e^{i\bk\cdot   \tilde{\bx}} \phi_p(\bk;\tilde{\bx})  \nonumber \\
&&\hspace{1cm}- \epsilon^{2-d} \big[\phi_p(\boldsymbol{0}; \tilde{\bx}) + \bchi^{\mbox{\tiny{(1)}}}(\tilde{\bx}) \cdot i  {\bk} \big] \frac{    F({\bk}/\epsilon) e^{i{\bk} \cdot \tilde{\bx}   }  }{-\frac{{\bmu}^{\mbox{\tiny{(0)}}}}{\rho^{\mbox{\tiny{(0)}}}}\!: (i{\bk})^2  -\epsilon^2 \sigma{\hat{\Omega}}^2}\Big|^2 \ind \tilde{\bx} \ind \bk.  \label{Main theorem proof Jup-Ju0 norm first-order}
\end{eqnarray}}
From equation \eqref{Main theorem proof Jup-Ju0 norm first-order} and \eqref{Main theorem proof Jup-Ju0 first term} we obtain that
{\footnotesize
\begin{eqnarray} \label{Main theorem proof Jup-Ju0 norm 1 first-order}
&&\|\mathcal{J}_{\mathbb{R}^d} u^{(p)}-\mathcal{J}_{\mathbb{R}^d} u_1\|^2_{L^2(W_{\!{\boldsymbol{I}^*}};\mathcal{H}^0_{\bk} (W_{\!\boldsymbol{I}}))} \nonumber \\
&=& \epsilon^{4-2d}  \int_{\eps Y} \int_{W_{\!\boldsymbol{I}}} \bigg|     F(\bk/\eps) e^{i\bk\cdot   \tilde{\bx}} \Big[ \frac{   \langle \rho {\phi_p(\boldsymbol{0}; \cdot)} \overline{\phi_p( {\bk};{\cdot})} \rangle }{\omega^2_p(\bk) -\omega^2} \phi_p(\bk;\tilde{\bx}) - \frac{  \phi_p(\boldsymbol{0}; \tilde{\bx}) + \bchi^{\mbox{\tiny{(1)}}}(\tilde{\bx}) \cdot i  {\bk}   }{-\frac{{\bmu}^{\mbox{\tiny{(0)}}}}{\rho^{\mbox{\tiny{(0)}}}}\!: (i{\bk})^2  -\epsilon^2 \sigma{\hat{\Omega}}^2}  
\Big]\bigg|^2 \ind \tilde{\bx} \ind \bk \nonumber \\
&=& \epsilon^{-d}  \int_{Y} \int_{W_{\!\boldsymbol{I}}} \bigg|     F(\hat{\bk}) e^{i \epsilon \hat{\bk}\cdot   \tilde{\bx}} \Big[ \frac{   \langle \rho {\phi_p(\boldsymbol{0}; \cdot)} \overline{\phi_p( {\bk};{\cdot})} \rangle }{\big(\omega^2_p(\eps \hat{\bk}) -\omega^2\big)/\epsilon^2} \phi_p(\eps \hat{\bk};\tilde{\bx}) - \frac{ \phi_p(\boldsymbol{0}; \tilde{\bx}) + \eps \bchi^{\mbox{\tiny{(1)}}}(\tilde{\bx}) \cdot i  \hat{\bk}  }{-\frac{{\bmu}^{\mbox{\tiny{(0)}}}}{\rho^{\mbox{\tiny{(0)}}}}\!: (i{\hat{\bk}})^2  -  \sigma{\hat{\Omega}}^2}
\Big]\bigg|^2 \ind \tilde{\bx} \ind \hat{\bk} \nonumber \\
&\le& \epsilon^{-d} W_{\!\boldsymbol{I}}  \int_{Y} \Big|F(\hat{\bk}) \Big|^2 \Big \|\frac{   \langle \rho {\phi_p(\boldsymbol{0}; \cdot)} \overline{\phi_p( {\bk};{\cdot})} \rangle \phi_p(\eps \hat{\bk};\bx) }{\big(\omega^2_p(\eps \hat{\bk}) -\omega^2\big)/\epsilon^2} - \frac{ \phi_p(\boldsymbol{0}; \tilde{\bx}) + \eps \bchi^{\mbox{\tiny{(1)}}}(\tilde{\bx}) \cdot i  \hat{\bk} }{-\frac{{\bmu}^{\mbox{\tiny{(0)}}}}{\rho^{\mbox{\tiny{(0)}}}}\!: (i{\hat{\bk}})^2  -  \sigma{\hat{\Omega}}^2}   \Big\|^2_{L^2(W_{\!{\boldsymbol{I}}})}  \ind \hat{\bk}.
\end{eqnarray}
}From Proposition \ref{Prop Section Higher order eqn 1}, the expression of $\omega$ in \eqref{drf1}, and the asymptotic of $\omega^2_p(\eps \hat{\bk})$ in Lemma \ref{asymptotic main phi_0}, we obtain that
{\footnotesize
\begin{eqnarray}
&& \Big \|\frac{   \langle \rho {\phi_p(\boldsymbol{0}; \cdot)} \overline{\phi_p( {\bk};{\cdot})} \rangle \phi_p(\eps \hat{\bk};\tilde{\bx}) }{\big(\omega^2_p(\eps \hat{\bk}) -\omega^2\big)/\epsilon^2} - \frac{\phi_p(\boldsymbol{0}; \tilde{\bx}) + \eps \bchi^{\mbox{\tiny{(1)}}}(\tilde{\bx}) \cdot i  \hat{\bk} }{-\frac{{\bmu}^{\mbox{\tiny{(0)}}}}{\rho^{\mbox{\tiny{(0)}}}}\!: (i{\hat{\bk}})^2  -  \sigma{\hat{\Omega}}^2}   \Big\|_{L^2(W_{\!{\boldsymbol{I}}})} \nonumber \\
&=&  \Big \|\frac{   \langle \rho {\phi_p(\boldsymbol{0}; \cdot)} \overline{\phi_p( {\bk};{\cdot})} \rangle \phi_p(\eps \hat{\bk};\tilde{\bx})  }{-\frac{{\bmu}^{\mbox{\tiny{(0)}}}}{\rho^{\mbox{\tiny{(0)}}}}\!: (i{\hat{\bk}})^2  -  \sigma{\hat{\Omega}}^2 +O(\epsilon^2)}  - \frac{  \phi_p(\boldsymbol{0}; \tilde{\bx}) + \eps \bchi^{\mbox{\tiny{(1)}}}(\tilde{\bx}) \cdot i  \hat{\bk}   }{-\frac{{\bmu}^{\mbox{\tiny{(0)}}}}{\rho^{\mbox{\tiny{(0)}}}}\!: (i{\hat{\bk}})^2  -  \sigma{\hat{\Omega}}^2}  \Big\|_{L^2(W_{\!{\boldsymbol{I}}})} (\mbox{using } \eqref{drf1} \mbox{ and Lemma } \ref{asymptotic main phi_0}) \nonumber \\
&=& \Big|  \frac{  1}{-\frac{{\bmu}^{\mbox{\tiny{(0)}}}}{\rho^{\mbox{\tiny{(0)}}}}\!: (i{\hat{\bk}})^2  -  \sigma{\hat{\Omega}}^2} \Big| \Big \|    \langle \rho {\phi_p(\boldsymbol{0}; \cdot)} \overline{\phi_p( {\bk};{\cdot})} \rangle \phi_p(\eps \hat{\bk};\tilde{\bx})- \phi_p(\boldsymbol{0}; \tilde{\bx}) -\eps \bchi^{\mbox{\tiny{(1)}}}(\tilde{\bx}) \cdot i  \hat{\bk} \Big\|_{L^2(W_{\!{\boldsymbol{I}}})} +O(\epsilon^2) \nonumber \\
&=& O(\epsilon^2) \qquad ( \mbox{using Proposition }\ref{Prop Section Higher order eqn 1}). \label{Main theorem proof Jup-Ju0 norm 2 first-order}
\end{eqnarray}
}Now from equations \eqref{Main theorem proof Jup-Ju0 norm 1 first-order} and \eqref{Main theorem proof Jup-Ju0 norm 2 first-order}, we obtain that
{\footnotesize
\begin{eqnarray} \label{Main theorem proof Jup-Ju0 norm 3 first-order}
\|\mathcal{J}_{\mathbb{R}^d} u^{(p)}-\mathcal{J}_{\mathbb{R}^d} u_1\|^2_{L^2(W_{\!{\boldsymbol{I}^*}};\mathcal{H}^0_{\bk} (W_{\!\boldsymbol{I}}))}  \le  c \epsilon^{-d} (O(\epsilon^2) )^2,
\end{eqnarray}
}for some constant $c$ independent of $\epsilon$. Thus this proves that $\|\mathcal{J}_{\mathbb{R}^d} u^{(p)}-\mathcal{J}_{\mathbb{R}^d} u_1\|_{L^2(W_{\!{\boldsymbol{I}^*}};\mathcal{H}^0_{\bk} (W_{\!\boldsymbol{I}}))} = \epsilon^{-d/2}O(\eps^2)$ and hence $\|u^{(p)}-u_1\|_{L^2(\mathbb{R}^d)} = \epsilon^{-d/2}O(\eps^2)$, i.e., $\|U^{(p)}-U_1\|_{L^2(\mathbb{R}^d)} = O(\eps^2)$. This completes the proof in the first-order case.

\textbf{3.}  Proof of $\|U^{(p)}-U_2\|_{L^2(\mathbb{R}^d)} = O(\eps^3)$: The idea again is similar to the proof in the leading-order case. We show below the detailed proof. $\|U^{(p)}-U_2\|_{L^2(\mathbb{R}^d)} = O(\eps^3)$ is equivalent to $\|u^{(p)}-u_2\|_{L^2(\mathbb{R}^d)} = \epsilon^{-d/2}O(\eps^3)$ where $u_2(\bx):= U_2(\epsilon\bx)$. According to Lemma \ref{Lemma Lemma Bloch expansion}, it is sufficient to show that $\|\mathcal{J}_{\mathbb{R}^d} u^{(p)}-\mathcal{J}_{\mathbb{R}^d} u_2\|_{L^2(W_{\!{\boldsymbol{I}^*}};\mathcal{H}^0_{\bk} (W_{\!\boldsymbol{I}}))} = \epsilon^{-d/2}O(\eps^3)$.

We first obtain the representation for $\mathcal{J}_{\mathbb{R}^d} u_2$. Indeed, from the expression \eqref{Main second-order U2} of $U_2$ with $W_2$ defined by \eqref{Main W2}, we obtain (by noting that $F$ is compactly supported in $Y$) that
{\footnotesize
\begin{eqnarray}
&& u_2(\bx)= U_2(\epsilon\bx) \nonumber \\
&=& \frac{1}{(2\pi)^{d/2}} \int_{Y}    \frac{    \big[\phi_p(\boldsymbol{0}; \bx) + \epsilon \bchi^{\mbox{\tiny{(1)}}}(\bx) \cdot i \hat{\bk} + \eps^2  \big(\langle \rho \bchi^{(1)} \otimes \overline{\bchi^{(1)}} \rangle  \phi_p(\boldsymbol{0}; \bx)+  \bchi^{\mbox{\tiny{(2)}}}(\bx)  \big): (i\hat{\bk})^2\big]  F(\hat{\bk}) e^{i\eps \hat{\bk} \cdot   \bx}  }{-\frac{{\bmu}^{\mbox{\tiny{(0)}}}}{\rho^{\mbox{\tiny{(0)}}}}\!: (i\hat{\bk})^2  -\eps^2 \frac{{\bmu}^{\mbox{\tiny{(2)}}}}{\rho^{\mbox{\tiny{(0)}}}}\!: (i\hat{\bk})^4 -{ \sigma}{\hat{\Omega}}^2}  \ind \hat{\bk} \nonumber \\
&=& \frac{\epsilon^{2-d}}{(2\pi)^{d/2}} \int_{\epsilon Y}    \frac{  \big[\phi_p(\boldsymbol{0}; \bx) +  \bchi^{\mbox{\tiny{(1)}}}(\bx) \cdot i  {\bk} + \big( \langle \rho \bchi^{(1)} \otimes \overline{\bchi^{(1)}} \rangle  \phi_p(\boldsymbol{0}; \bx)+   \bchi^{\mbox{\tiny{(2)}}}(\bx) \big) : (i {\bk})^2 \big] F({\bk}/\epsilon) e^{i{\bk} \cdot   \bx}  }{-\frac{{\bmu}^{\mbox{\tiny{(0)}}}}{\rho^{\mbox{\tiny{(0)}}}}\!: (i{\bk})^2  - \frac{{\bmu}^{\mbox{\tiny{(2)}}}}{\rho^{\mbox{\tiny{(0)}}}}\!: (i {\bk})^4 -\eps^2{ \sigma}{\hat{\Omega}}^2}  \ind {\bk} \nonumber \\
&=& \big[\phi_p(\boldsymbol{0}; \tilde{\bx}) +  \bchi^{\mbox{\tiny{(1)}}}(\tilde{\bx}) \cdot i  {\bk} + \big( \langle \rho \bchi^{(1)} \otimes \overline{\bchi^{(1)}} \rangle  \phi_p(\boldsymbol{0}; \tilde{\bx})+   \bchi^{\mbox{\tiny{(2)}}}(\tilde{\bx}) \big) : (i {\bk})^2 \big] \nonumber \\
&&\times \frac{\epsilon^{2-d}}{(2\pi)^{d/2}} \int_{ W_{\boldsymbol{I}^*}}    \frac{    F({\bk}/\epsilon) e^{i{\bk} \cdot   \tilde{\bx}}  }{-\frac{{\bmu}^{\mbox{\tiny{(0)}}}}{\rho^{\mbox{\tiny{(0)}}}}\!: (i{\bk})^2  - \frac{{\bmu}^{\mbox{\tiny{(2)}}}}{\rho^{\mbox{\tiny{(0)}}}}\!: (i {\bk})^4 -\eps^2{ \sigma}{\hat{\Omega}}^2}  e^{i{\bk} \cdot   \bj} ~\ind \bk \nonumber \\
&=&\epsilon^{2-d} \big[\phi_p(\boldsymbol{0}; \tilde{\bx}) +  \bchi^{\mbox{\tiny{(1)}}}(\tilde{\bx}) \cdot i  {\bk} + \big( \langle \rho \bchi^{(1)} \otimes \overline{\bchi^{(1)}} \rangle  \phi_p(\boldsymbol{0}; \tilde{\bx})+   \bchi^{\mbox{\tiny{(2)}}}(\tilde{\bx}) \big) : (i {\bk})^2 \big]  \nonumber \\
&& \times \mathcal{J}_{\mathbb{R}^d}^{-1} [ \frac{    F({\bk}/\epsilon) e^{i{\bk} \cdot \tilde{\bx}   }  }{-\frac{{\bmu}^{\mbox{\tiny{(0)}}}}{\rho^{\mbox{\tiny{(0)}}}}\!: (i{\bk})^2  - \frac{{\bmu}^{\mbox{\tiny{(2)}}}}{\rho^{\mbox{\tiny{(0)}}}}\!: (i {\bk})^4-\eps^2{ \sigma}{\hat{\Omega}}^2}], \nonumber 
\end{eqnarray}
}
where $\mbox{$\bj$ is such that $\tilde{\bx}=\bx-\bj \in W_{\!{\boldsymbol{I}}}$ }$. Therefore we obtain that
{\footnotesize
\begin{eqnarray} 
&&(\mathcal{J}_{\mathbb{R}^d} u_0)(\bk;\tilde{\bx}) = \epsilon^{2-d} \big[\phi_p(\boldsymbol{0}; \tilde{\bx}) +  \bchi^{\mbox{\tiny{(1)}}}(\tilde{\bx}) \cdot i  {\bk} + \big( \langle \rho \bchi^{(1)} \otimes \overline{\bchi^{(1)}} \rangle  \phi_p(\boldsymbol{0}; \tilde{\bx})+   \bchi^{\mbox{\tiny{(2)}}}(\tilde{\bx}) \big) : (i {\bk})^2 \big]  \nonumber \\
&&\times   \frac{    F({\bk}/\epsilon) e^{i{\bk} \cdot \tilde{\bx}   }  }{-\frac{{\bmu}^{\mbox{\tiny{(0)}}}}{\rho^{\mbox{\tiny{(0)}}}}\!: (i{\bk})^2 - \frac{{\bmu}^{\mbox{\tiny{(2)}}}}{\rho^{\mbox{\tiny{(0)}}}}\!: (i {\bk})^4-\eps^2{ \sigma}{\hat{\Omega}}^2}, \quad \bk \in W_{\boldsymbol{I}^*}, \tilde{\bx} \in W_{\boldsymbol{I}}. \label{Main theorem proof Ju0 second-order}
\end{eqnarray}
}Now from equations \eqref{Main theorem proof Ju0 second-order}  and \eqref{Main theorem proof Jup}, we obtain (noting both terms are compactly supported in $\eps Y$) that
{\footnotesize
\begin{eqnarray}
&&\|\mathcal{J}_{\mathbb{R}^d} u^{(p)}-\mathcal{J}_{\mathbb{R}^d} u_1\|^2_{L^2(W_{\!{\boldsymbol{I}^*}};\mathcal{H}^0_{\bk} (W_{\!\boldsymbol{I}}))} \nonumber \\
&=& \int_{\eps Y} \int_{W_{\!\boldsymbol{I}}} \Big| \frac{\eps^2   \langle  \mathcal{J}_{\mathbb{R}^d} \big[{ f_{\epsilon}}(\eps \cdot) \big] ( {\bk}; \cdot), e^{i  {\bk}  \cdot}\phi_p( ;\cdot) \rangle }{\omega^2_p(\bk) -\omega^2} e^{i\bk\cdot   \tilde{\bx}} \phi_p(\bk;\tilde{\bx})  \nonumber \\
&& - \epsilon^{2-d} \big[\phi_p(\boldsymbol{0}; \tilde{\bx}) +  \bchi^{\mbox{\tiny{(1)}}}(\tilde{\bx}) \cdot i  {\bk} + \big( \langle \rho \bchi^{(1)} \otimes \overline{\bchi^{(1)}} \rangle  \phi_p(\boldsymbol{0}; \tilde{\bx})+   \bchi^{\mbox{\tiny{(2)}}}(\tilde{\bx}) \big) : (i {\bk})^2 \big]  \nonumber \\
&& \times \frac{    F({\bk}/\epsilon) e^{i{\bk} \cdot \tilde{\bx}   }  }{-\frac{{\bmu}^{\mbox{\tiny{(0)}}}}{\rho^{\mbox{\tiny{(0)}}}}\!: (i{\bk})^2 - \frac{{\bmu}^{\mbox{\tiny{(2)}}}}{\rho^{\mbox{\tiny{(0)}}}}\!: (i {\bk})^4 -\eps^2{ \sigma}{\hat{\Omega}}^2}\Big|^2 \ind \tilde{\bx} \ind \bk.  \label{Main theorem proof Jup-Ju0 norm second-order}
\end{eqnarray}}
From equation \eqref{Main theorem proof Jup-Ju0 norm second-order} and \eqref{Main theorem proof Jup-Ju0 first term} we obtain that
{\footnotesize
\begin{eqnarray} \label{Main theorem proof Jup-Ju0 norm 1 second-order}
&&\|\mathcal{J}_{\mathbb{R}^d} u^{(p)}-\mathcal{J}_{\mathbb{R}^d} u_1\|^2_{L^2(W_{\!{\boldsymbol{I}^*}};\mathcal{H}^0_{\bk} (W_{\!\boldsymbol{I}}))} \nonumber \\
&=& \epsilon^{4-2d}  \int_{\eps Y} \int_{W_{\!\boldsymbol{I}}} \bigg|     F(\bk/\eps) e^{i\bk\cdot   \tilde{\bx}} \Big[ \frac{   \langle \rho {\phi_p(\boldsymbol{0}; \cdot)} \overline{\phi_p( {\bk};{\cdot})} \rangle }{\omega^2_p(\bk) -\omega^2} \phi_p(\bk;\tilde{\bx}) \nonumber \\
&&- \frac{ \phi_p(\boldsymbol{0}; \tilde{\bx}) +  \bchi^{\mbox{\tiny{(1)}}}(\tilde{\bx}) \cdot i  {\bk} + \big( \langle \rho \bchi^{(1)} \otimes \overline{\bchi^{(1)}} \rangle  \phi_p(\boldsymbol{0}; \tilde{\bx})+   \bchi^{\mbox{\tiny{(2)}}}(\tilde{\bx}) \big) : (i {\bk})^2  }{-\frac{{\bmu}^{\mbox{\tiny{(0)}}}}{\rho^{\mbox{\tiny{(0)}}}}\!: (i{\bk})^2 - \frac{{\bmu}^{\mbox{\tiny{(2)}}}}{\rho^{\mbox{\tiny{(0)}}}}\!: (i {\bk})^4 -\eps^2{ \sigma}{\hat{\Omega}}^2}  
\Big]\bigg|^2 \ind \tilde{\bx} \ind \bk \nonumber \\
&=& \epsilon^{-d}  \int_{Y} \int_{W_{\!\boldsymbol{I}}} \bigg|     F(\hat{\bk}) e^{i \epsilon \hat{\bk}\cdot   \tilde{\bx}} \Big[ \frac{   \langle \rho {\phi_p(\boldsymbol{0}; \cdot)} \overline{\phi_p( {\bk};{\cdot})} \rangle }{\big(\omega^2_p(\eps \hat{\bk}) -\omega^2\big)/\epsilon^2} \phi_p(\eps \hat{\bk};\tilde{\bx}) \nonumber \\
&& -\frac{ \phi_p(\boldsymbol{0}; \tilde{\bx}) + \eps \bchi^{\mbox{\tiny{(1)}}}(\tilde{\bx}) \cdot i  \hat{\bk} +\eps^2  \big( \langle \rho \bchi^{(1)} \otimes \overline{\bchi^{(1)}} \rangle  \phi_p(\boldsymbol{0}; \tilde{\bx})+   \bchi^{\mbox{\tiny{(2)}}}(\tilde{\bx}) \big) : (i \hat{\bk})^2  }{-\frac{{\bmu}^{\mbox{\tiny{(0)}}}}{\rho^{\mbox{\tiny{(0)}}}}\!: (i\hat{\bk})^2 -   \eps^2  \frac{{\bmu}^{\mbox{\tiny{(2)}}}}{\rho^{\mbox{\tiny{(0)}}}}\!: (i \hat{\bk})^4 -{ \sigma}{\hat{\Omega}}^2}
\Big]\bigg|^2 \ind \tilde{\bx} \ind \hat{\bk} \nonumber \\
&\le& \epsilon^{-d} W_{\!\boldsymbol{I}}  \int_{Y} \Big|F(\hat{\bk}) \Big|^2 \Big \|\frac{   \langle \rho {\phi_p(\boldsymbol{0}; \cdot)} \overline{\phi_p( {\bk};{\cdot})} \rangle \phi_p(\eps \hat{\bk};\bx) }{\big(\omega^2_p(\eps \hat{\bk}) -\omega^2\big)/\epsilon^2} \nonumber \\
&&-\frac{ \phi_p(\boldsymbol{0}; \tilde{\bx}) + \eps \bchi^{\mbox{\tiny{(1)}}}(\tilde{\bx}) \cdot i  \hat{\bk} +\eps^2  \big( \langle \rho \bchi^{(1)} \otimes \overline{\bchi^{(1)}} \rangle  \phi_p(\boldsymbol{0}; \tilde{\bx})+   \bchi^{\mbox{\tiny{(2)}}}(\tilde{\bx}) \big) : (i \hat{\bk})^2  }{-\frac{{\bmu}^{\mbox{\tiny{(0)}}}}{\rho^{\mbox{\tiny{(0)}}}}\!: (i\hat{\bk})^2 -   \eps^2  \frac{{\bmu}^{\mbox{\tiny{(2)}}}}{\rho^{\mbox{\tiny{(0)}}}}\!: (i \hat{\bk})^4-{ \sigma}{\hat{\Omega}}^2} \Big\|^2_{L^2(W_{\!{\boldsymbol{I}}})}  \ind \hat{\bk}.
\end{eqnarray}
}From  Proposition \ref{Prop Section Higher order eqn 1}, the expression of $\omega$ in \eqref{drf1}, and the asymptotic of $\omega^2_p(\eps \hat{\bk})$ in Lemma \ref{asymptotic main phi_0}, we obtain that
{\footnotesize
\begin{eqnarray}
&& \Big \|\frac{   \langle \rho {\phi_p(\boldsymbol{0}; \cdot)} \overline{\phi_p( {\bk};{\cdot})} \rangle \phi_p(\eps \hat{\bk};\bx) }{\big(\omega^2_p(\eps \hat{\bk}) -\omega^2\big)/\epsilon^2} \nonumber \\
&&-\frac{ \phi_p(\boldsymbol{0}; \tilde{\bx}) + \eps \bchi^{\mbox{\tiny{(1)}}}(\tilde{\bx}) \cdot i  \hat{\bk} +\eps^2  \big( \langle \rho \bchi^{(1)} \otimes \overline{\bchi^{(1)}} \rangle  \phi_p(\boldsymbol{0}; \tilde{\bx})+   \bchi^{\mbox{\tiny{(2)}}}(\tilde{\bx}) \big) : (i \hat{\bk})^2  }{-\frac{{\bmu}^{\mbox{\tiny{(0)}}}}{\rho^{\mbox{\tiny{(0)}}}}\!: (i\hat{\bk})^2 -   \eps^2  \frac{{\bmu}^{\mbox{\tiny{(2)}}}}{\rho^{\mbox{\tiny{(0)}}}}\!: (i \hat{\bk})^4 -{ \sigma}{\hat{\Omega}}^2}   \Big\|_{L^2(W_{\!{\boldsymbol{I}}})} \nonumber \\
&=&  \Big \|\frac{   \langle \rho {\phi_p(\boldsymbol{0}; \cdot)} \overline{\phi_p( {\bk};{\cdot})} \rangle \phi_p(\eps \hat{\bk};\tilde{\bx})  }{-\frac{{\bmu}^{\mbox{\tiny{(0)}}}}{\rho^{\mbox{\tiny{(0)}}}}\!: (i\hat{\bk})^2 -   \eps^2  \frac{{\bmu}^{\mbox{\tiny{(2)}}}}{\rho^{\mbox{\tiny{(0)}}}}\!: (i \hat{\bk})^4 -{ \sigma}{\hat{\Omega}}^2 +O(\epsilon^4)}  \nonumber \\
&&-\frac{ \phi_p(\boldsymbol{0}; \tilde{\bx}) + \eps \bchi^{\mbox{\tiny{(1)}}}(\tilde{\bx}) \cdot i  \hat{\bk} +\eps^2  \big( \langle \rho \bchi^{(1)} \otimes \overline{\bchi^{(1)}} \rangle  \phi_p(\boldsymbol{0}; \tilde{\bx})+   \bchi^{\mbox{\tiny{(2)}}}(\tilde{\bx}) \big) : (i \hat{\bk})^2  }{-\frac{{\bmu}^{\mbox{\tiny{(0)}}}}{\rho^{\mbox{\tiny{(0)}}}}\!: (i\hat{\bk})^2-   \eps^2  \frac{{\bmu}^{\mbox{\tiny{(2)}}}}{\rho^{\mbox{\tiny{(0)}}}}\!: (i \hat{\bk})^4 -{ \sigma}{\hat{\Omega}}^2} \Big\|_{L^2(W_{\!{\boldsymbol{I}}})} \nonumber \\
&&(\mbox{using } \eqref{drf1} \mbox{ and Lemma } \ref{asymptotic main phi_0}) \nonumber \\
&=& \Big|  \frac{  1}{-\frac{{\bmu}^{\mbox{\tiny{(0)}}}}{\rho^{\mbox{\tiny{(0)}}}}\!: (i\hat{\bk})^2 -   \eps^2  \frac{{\bmu}^{\mbox{\tiny{(2)}}}}{\rho^{\mbox{\tiny{(0)}}}}\!: (i \hat{\bk})^4 -{ \sigma}{\hat{\Omega}}^2} \Big|  \cdot \Big \|    \langle \rho {\phi_p(\boldsymbol{0}; \cdot)} \overline{\phi_p( {\bk};{\cdot})} \rangle \phi_p(\eps \hat{\bk};\tilde{\bx})- \phi_p(\boldsymbol{0}; \tilde{\bx}) -\eps \bchi^{\mbox{\tiny{(1)}}}(\tilde{\bx}) \cdot i  \hat{\bk} \nonumber \\
&& -\eps^2  \big( \langle \rho \bchi^{(1)} \otimes \overline{\bchi^{(1)}} \rangle  \phi_p(\boldsymbol{0}; \tilde{\bx})+   \bchi^{\mbox{\tiny{(2)}}}(\tilde{\bx}) \big) : (i \hat{\bk})^2  \Big\|_{L^2(W_{\!{\boldsymbol{I}}})} +O(\epsilon^4) \nonumber \\
&=& O(\epsilon^3) \qquad ( \mbox{using Proposition }\ref{Prop Section Higher order eqn 1}). \label{Main theorem proof Jup-Ju0 norm 2 second-order}
\end{eqnarray}
}Now from equations \eqref{Main theorem proof Jup-Ju0 norm 1 second-order} and \eqref{Main theorem proof Jup-Ju0 norm 2 second-order}, we obtain that
{\footnotesize
\begin{eqnarray} \label{Main theorem proof Jup-Ju0 norm 3 second-order}
\|\mathcal{J}_{\mathbb{R}^d} u^{(p)}-\mathcal{J}_{\mathbb{R}^d} u_2\|^2_{L^2(W_{\!{\boldsymbol{I}^*}};\mathcal{H}^0_{\bk} (W_{\!\boldsymbol{I}}))}  \le  c \epsilon^{-d} (O(\epsilon^3) )^2,
\end{eqnarray}
}for some constant $c$ independent of $\epsilon$. Thus this proves that $\|\mathcal{J}_{\mathbb{R}^d} u^{(p)}-\mathcal{J}_{\mathbb{R}^d} u_2\|_{L^2(W_{\!{\boldsymbol{I}^*}};\mathcal{H}^0_{\bk} (W_{\!\boldsymbol{I}}))} = \epsilon^{-d/2}O(\eps^3)$ and hence $\|u^{(p)}-u_2\|_{L^2(\mathbb{R}^d)} = \epsilon^{-d/2}O(\eps^3)$, i.e., $\|U^{(p)}-U_2\|_{L^2(\mathbb{R}^d)} = O(\eps^3)$. This completes the proof in the second-order case.
\end{proof}

\begin{remark}
 { We remark that $\bchi^{\mbox{\tiny{(3)}}}$ is needed for the computation of ${\bmu}^{\mbox{\tiny{(2)}}}$.}
 Furthermore, the functions $W_0$ and $W_1$ indeed are solutions to the following equations \eqref{Main envelop W0} and \eqref{Main envelop W2} respectively (one can directly use the explicit expression of $W_0$ \eqref{Main W0} and $W_1$   \eqref{Main W2} to verify this; and conversely, taking the Fourier transform of equation \eqref{Main envelop W0} and \eqref{Main envelop W2}, one can derive the  explicit expressions $W_0$ \eqref{Main W0} and $W_1$   \eqref{Main W2})   
\begin{eqnarray} 
&&{\bmu}^{\mbox{\tiny{(0)}}} \!:\! \nabla^2 W_0(\bx) + \rho^{\mbox{\tiny{(0)}}} \sigma{\hat{\Omega}}^2 W_0(\bx) \:=\, -\rho^{\mbox{\tiny{(0)}}} \mathcal{F}^{-1}[F] (\bx) \quad \mbox{in} \quad \mathbb{R}^d, \label{Main envelop W0}\\
&&{\bmu}^{\mbox{\tiny{(0)}}} : \nabla_{\br}^2 W_2(\br) + \eps^2  {\bmu}^{\mbox{\tiny{(2)}}}    : \nabla_{\br}^4 W_2(\br)   + \rho^{\mbox{\tiny{(0)}}} \sigma{\hat{\Omega}}^2 W_2(\br) \nonumber\\
&&\hspace{4cm}=\: -\rho^{\mbox{\tiny{(0)}}} \mathcal{F}^{-1}[F] (\br) \quad \mbox{in} \quad \mathbb{R}^d. \label{Main envelop W2}
\end{eqnarray}
{ The above effective field equations describe the leading-order  and second-order effective  wave motions, respectively. Only by direct calculations, one can derive the effective field equations via formal two-scale asymptotic expansions as well, and it is then directly seen that the above effective field equations \eqref{Main envelop W0}--\eqref{Main envelop W2} agree with those obtained via formal two-scale asymptotic expansions}, both in low- and high-frequency cases, see for instance \cite{meng2018dynamic,guzina2019rational}.
\end{remark}

\section{Numerical examples} \label{Numerical example}

Consider the wave motion~\eqref{PDE} with $d=2$ in a periodic medium depicted in Fig.~\ref{fig1}(a), whose unit (Wigner-Seitz) cell~\eqref{wgcell} contains centric circular inclusion of radius~$a=0.3$, see Fig.~\ref{fig1}(b). The coefficients inside the unit cell are given by 
\[
(G(\bx),\rho(\bx)) \:=\: 
\left\{ \begin{array}{ll}
(G_1=1,\rho_1=1), & \|\bx\|>a, \\
(G_2=6,\rho_2=20), & \|\bx\|<a. \end{array} \right.
\] 
Fig.~\ref{fig1}(c) and Fig.~\ref{fig1}(d) plot respectively the first Brillouin zone and the {  dispersion relationship} for the problem (first 14 branches) that features three complete band gaps. 

{We first illustrate our asymptotic model at low frequency. We take the  (scaled) excitation ``frequency'' as $\omega = i\eps$, i.e. $\omega^2 = -\eps^2$,
that (i) formally resides inside a band gap~\eqref{bg1}, and (ii) admits asymptotic representation~\eqref{drf1} with~$p=0$ and~$\sigma=-1$.} In the context of Assumption~\ref{assumption f}, we consider the source term~\eqref{source1} with $p=0$ and
 \begin{eqnarray}\label{Gforce}
F(\hat{\bk}) = \frac{1}{2\sqrt{\pi}} e^{-\frac{\| \hat{\bk }\|^2}{4}}.
\end{eqnarray}
{ 
In what follows, we compare an ``exact'' (numerically computed) response of the medium due to~\eqref{source1} and \eqref{Gforce} with its leading-, first-, and second-order asymptotic approximations for different values of $\eps$.  We remark that, in the wavenumber-frequency space (described by the {  dispersion relationship}), the source ``triggers'' the acoustic branch $(\bk,\omega_0(\bk))$ near the origin $\bk=\boldsymbol{0}$.
}

\begin{figure}[h!] 
\centering{\includegraphics[width= 130 mm]{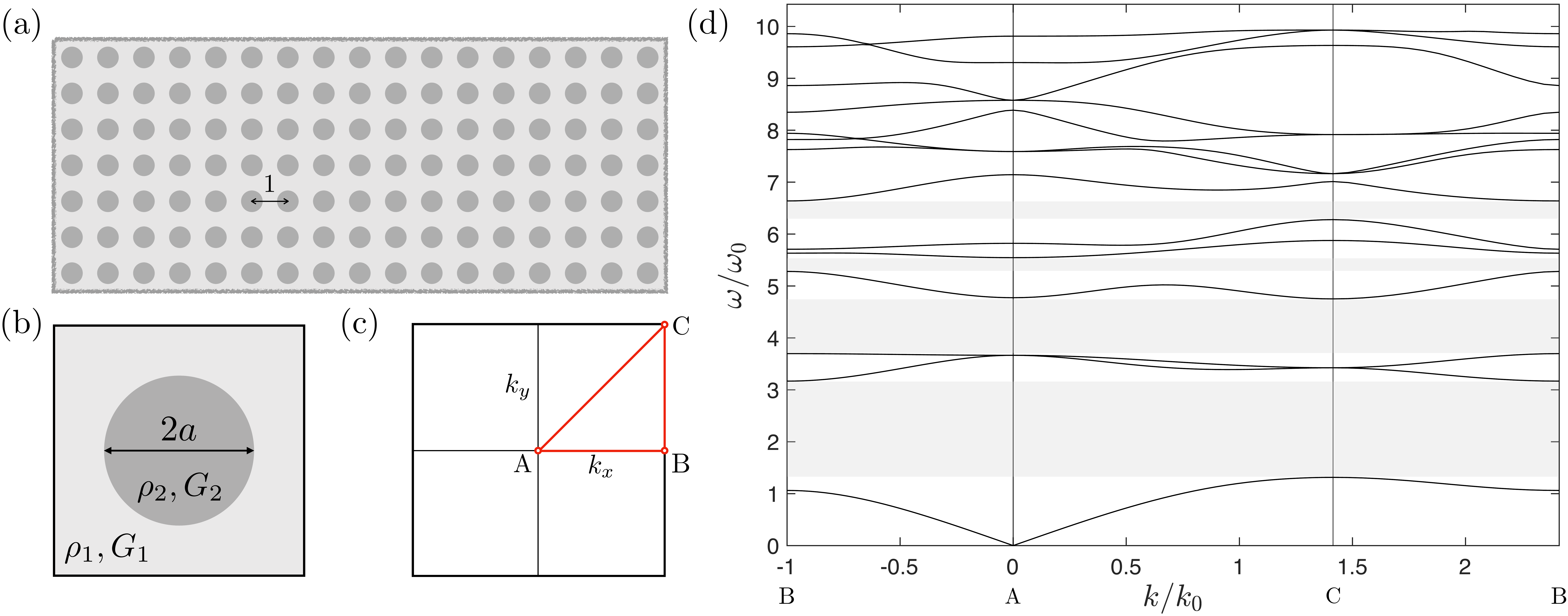}}\vspace*{0mm}
\caption{(a) Periodic medium~\eqref{Grho} with circular inclusions; (b) unit cell of periodicity~\eqref{wgcell}; (c) first Brillouin zone~\eqref{Brill}; (d) {  dispersion relationship} including the first 14 dispersion branches (band gaps are shaded). The normalization parameters are $k_0=\pi$ and $\omega_0=\sqrt{G_1/\rho_1}$. } \label{fig1} \vspace*{-0mm}
\end{figure}

To this end an ``exact'' solution is obtained, for each~$\eps$, via the finite element platform NGSolve~\cite{NGSolve}, where~$\mathbb{R}^2$ is approximated by a square domain 
\[
\mathcal{D}_{\tiny{\mbox{N}}}=[-(N+0.5),(N+0.5)]\times[-(N+0.5),(N+0.5)],
\]
with homogenous Dirichlet boundary conditions along its boundary~$\partial \mathcal{D}_{\tiny{\mbox{N}}}$. This domain is discretized using finite elements of order~3 and maximum size $h_1= 0.01125\eps$. The leading-, first- and second-order approximations of the solution are evaluated on a $64N\times 64N$ grid of points $\mathcal{G}_{\tiny{\mbox{N}}}$, by integrating numerically the expressions \eqref{Main leading-order U0}, \eqref{Main first-order U1} and \eqref{Main second-order U2} (with $p=0$), {after a change of variable $\hat{\bk}\to \bk=\epsilon\hat{\bk}$}, in MATLAB over the first Brillouin zone, {used to approximate}~{$\epsilon Y= \mathbb{R}^2 $}~{thanks to the strong exponential decay of}{~$F(\eps^{-1}\hat{\bk})$}. The eigenfunction $\phi_0$, cell functions $\bchi^{\mbox{\tiny{(1)}}},  \bchi^{\mbox{\tiny{(2)}}},  \bchi^{\mbox{\tiny{(3)}}}$ and effective coefficients $\rho^{\mbox{\tiny{(0)}}},\bmu^{\mbox{\tiny{(0)}}}$ and $\bmu^{\mbox{\tiny{(2)}}}$  featured in the approximations' expressions are evaluated using NGSolve by discretizing the unit cell with elements of order 5 and maximum length $h_2=0.0075$. The   eigenfunction $\phi_0(\boldsymbol{0},\bx)$ as well as the components of the affiliated cell functions $\bchi^{\mbox{\tiny{(1)}}}$ and $\bchi^{\mbox{\tiny{(2)}}}$ are shown in Fig. \ref{fig6}.\\ 

Here it should be noted that the perturbation parameter~$\eps$ physically signifies the amount of ``foray'' into the band gap, whereby larger values of~$\eps$ inherently yield faster (exponential) decay of the solution away from the source of disturbance, $f(\eps\bx)$. In terms of numerical simulations, this is the key alleviating factor that allows us to approximate the wave motion in~$\mathbb{R}^2$ with that in~$\mathcal{D}_{\tiny{\mbox{N}}}$ with homogeneous Dirichlet boundary conditions.  In this vein, the response of the medium is computed over $\mathcal{D}_{\tiny{\mbox{25}}}$ for $\eps=0.25$, and over $\mathcal{D}_{\tiny{\mbox{20}}}$ for $\eps=0.375$ and $\eps=0.5$. We also remark that the computation numerical responses for~$\eps\!<\!0.25$ was not feasible owing to excessive computational effort, in terms of larger values of~$N$, required to achieve sufficiently accurate numerical approximation of the wave motion in~$\mathbb{R}^2$.

Fig. \ref{fig4} plots the distribution over~$\mathcal{D}_{\tiny{\mbox{N}}}$ of~{$f_\eps(\eps\bx)$}, $u(\bx)$ and~$u_m(\cdot )=U_m(\eps \cdot)$ (which is given by \eqref{Main leading-order U0}--\eqref{Main second-order U2}) with $m=0,1,2$ for all three values of the perturbation parameter, namely~$\eps\in\{0.25,0.375,0.5\}$. An overall agreement between the solutions is evident from the display, as is the increase in fidelity of asymptotic approximation with the order of expansion. A detailed comparison is provided in Fig.~\ref{fig2} and Fig.~\ref{fig3}, which plot the variations of $u$ and~$u_m$ versus~$\bx=(x,y_0=const.)$ for $\eps=0.25$ and $\eps=0.5$, respectively. From the displays, we observe that all three approximations provide the same rate of decay as the exact solution; however, the fine solution detail in the vicinity of the peak load ($x\approx 0)$ is accurately approximated only by higher-order models. 

{  We further illustrate the asymptotic model at high frequency. 
We take $\Omega_\epsilon^2 = \epsilon^{-2} \omega_p^2(\boldsymbol{0}) -1$, i.e., the (scaled) excitation ``frequency'' as  $\omega^2 = \omega_p^2(\boldsymbol{0}) - \epsilon^2$ with $p=3$, that (i) formally resides inside a band gap between the $3$rd branch and $4$th branch, and (ii) admits asymptotic representation~\eqref{drf1} with~$p=3$ and~$\sigma=-1$. In the context of Assumption~\ref{assumption f}, we consider the source term~\eqref{source1} with $p=3$ given by $f_\eps(\epsilon \bx)=\frac{1}{2}\rho(\bx) \phi_3(\boldsymbol{0},\bx)e^{-\epsilon^2 \|\bx\|^2}$.

The eigenfunction $\phi_3$, cell functions $\bchi^{\mbox{\tiny{(1)}}},  \bchi^{\mbox{\tiny{(2)}}},  \bchi^{\mbox{\tiny{(3)}}}$ and effective coefficients $\rho^{\mbox{\tiny{(0)}}},\bmu^{\mbox{\tiny{(0)}}}$ and $\bmu^{\mbox{\tiny{(2)}}}$  featured in the approximations' expressions are evaluated using NGSolve by discretizing the unit cell with elements of order 5 and maximum length $h_3=0.021$. The corresponding eigenfunction $\phi_3(\boldsymbol{0},\bx)$ as well as the components of the affiliated cell functions $\bchi^{\mbox{\tiny{(1)}}}$ and $\bchi^{\mbox{\tiny{(2)}}}$ are shown in Fig. \ref{fig7}. Furthermore, Fig. \ref{fig8} plots $u_m(\cdot )=U_m(\eps \cdot)$ (which is given by \eqref{Main leading-order U0}--\eqref{Main second-order U2}) with $m=0,1,2$ for $\eps=0.25$ (left) and $\eps=0.5$ (right). The function $f_\eps(\eps\bx)$ (plotted over $10\times 10$ unit cells) is also included in the figure. Note that computation for the exact solution was not feasible owing to excessive computational effort (which is one of the reason to pursue high-order asymptotic models), in terms of larger values of~$N$, required to achieve sufficiently accurate numerical approximation of the wave motion in~$\mathbb{R}^2$. Nevertheless, it is possible to infer the increase in fidelity of asymptotic approximation with the order of expansion.}

Finally, for a better insight into the convergence of the asymptotic models, we next consider the relative approximation errors at low frequency (i.e. $p=0$ with $\omega$, $\epsilon$, $f$ given in the first numerical example) $e^{\tiny{\mbox{$(m)$}}}_{\tiny{\mbox{M}}}$ ($m=0,1,2$) given by
\begin{eqnarray}\label{relerr}
e^{\tiny{\mbox{$(m)$}}}_{\tiny{\mbox{M}}} (\eps)= \frac{\|u_m-u\|_{L^2(\mathcal{D}_{\tiny{\mbox{M-0.5}}})}} {\|u\|_{L^2(\mathcal{D}_{\tiny{\mbox{M-0.5}}})}},
\end{eqnarray}
\noindent where $M\leqslant N$. The domain integrals featured in~\eqref{relerr} are evaluated numerically by sampling the reference ``exact'' solution~$u$ and its approximations~$u_m$ over grid~$\mathcal{G}_{\tiny{\mbox{M}}}$. The values of $e^{\tiny{\mbox{$(m)$}}}_{\tiny{\mbox{M}}}(\eps)$ inherently depend on $M$, and were found to numerically converge at $M\!=\!15$ for all three values of the perturbation parameter, $\eps\in\{0.25,0.375,0.5\}$. Fig.~\ref{fig5} illustrates the observed scaling of $e^{\tiny{\mbox{$(m)$}}}_{\tiny{\mbox{M}}}(\eps)$, on the log-log scale, over this limited range of the perturbation parameter. Specifically, we find that the apparent linear trends (indicated by dashed lines) have slopes $e^{\tiny{\mbox{(0)}}}_{\tiny{\mbox{15}}}=O(\eps^{1.11})$, $e^{\tiny{\mbox{(1)}}}_{\tiny{\mbox{15}}}=O(\eps^{2.00})$ and $e^{\tiny{\mbox{(2)}}}_{\tiny{\mbox{15}}}=O(\eps^{3.06})$, which is in good agreement with the expected result $\lim_{\tiny{\mbox{M}}\to\infty}e^{\tiny{\mbox{(m)}}}_{\tiny{\mbox{M}}}=O(\eps^{m+1})$ due to Theorem~\ref{finalthm}.

\begin{figure}[h!] 
\centering{\includegraphics[height= 60 mm]{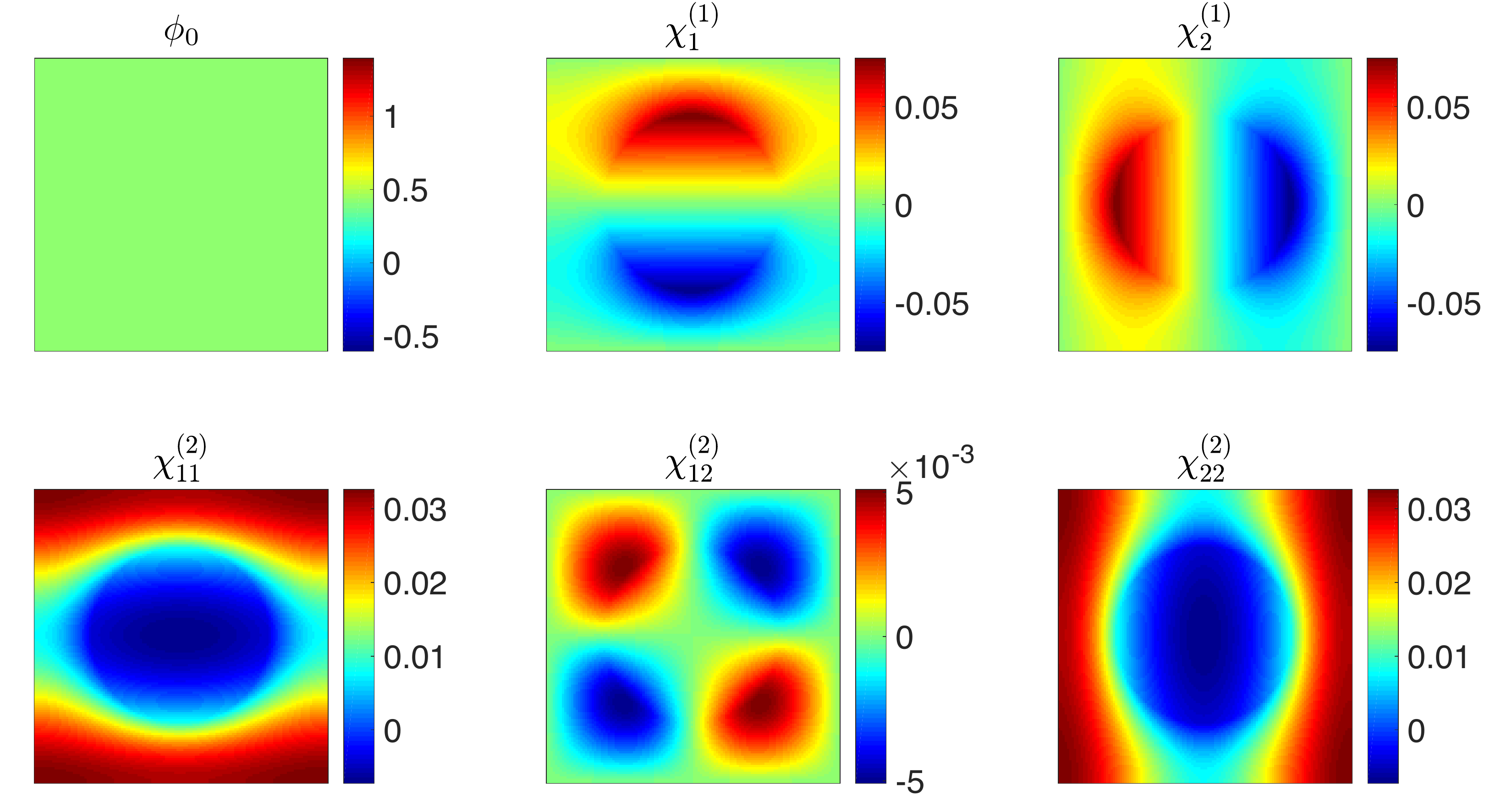}}\vspace*{-3mm}
\caption{Variation of eigenfunction $\phi_0(\boldsymbol{0},\bx)$ and components of the affiliated cell functions $\bchi^{\mbox{\tiny{(1)}}}(\boldsymbol{0},\bx)$ and $\bchi^{\mbox{\tiny{(2)}}}(\boldsymbol{0},\bx)$ (which are computed from $\phi_0(\boldsymbol{0},\bx)$) over the unit cell of periodicity~\eqref{wgcell}.} \label{fig6} \vspace*{-0mm}
\end{figure}

\begin{figure}[h!] 
\centering{\includegraphics[width= 130 mm]{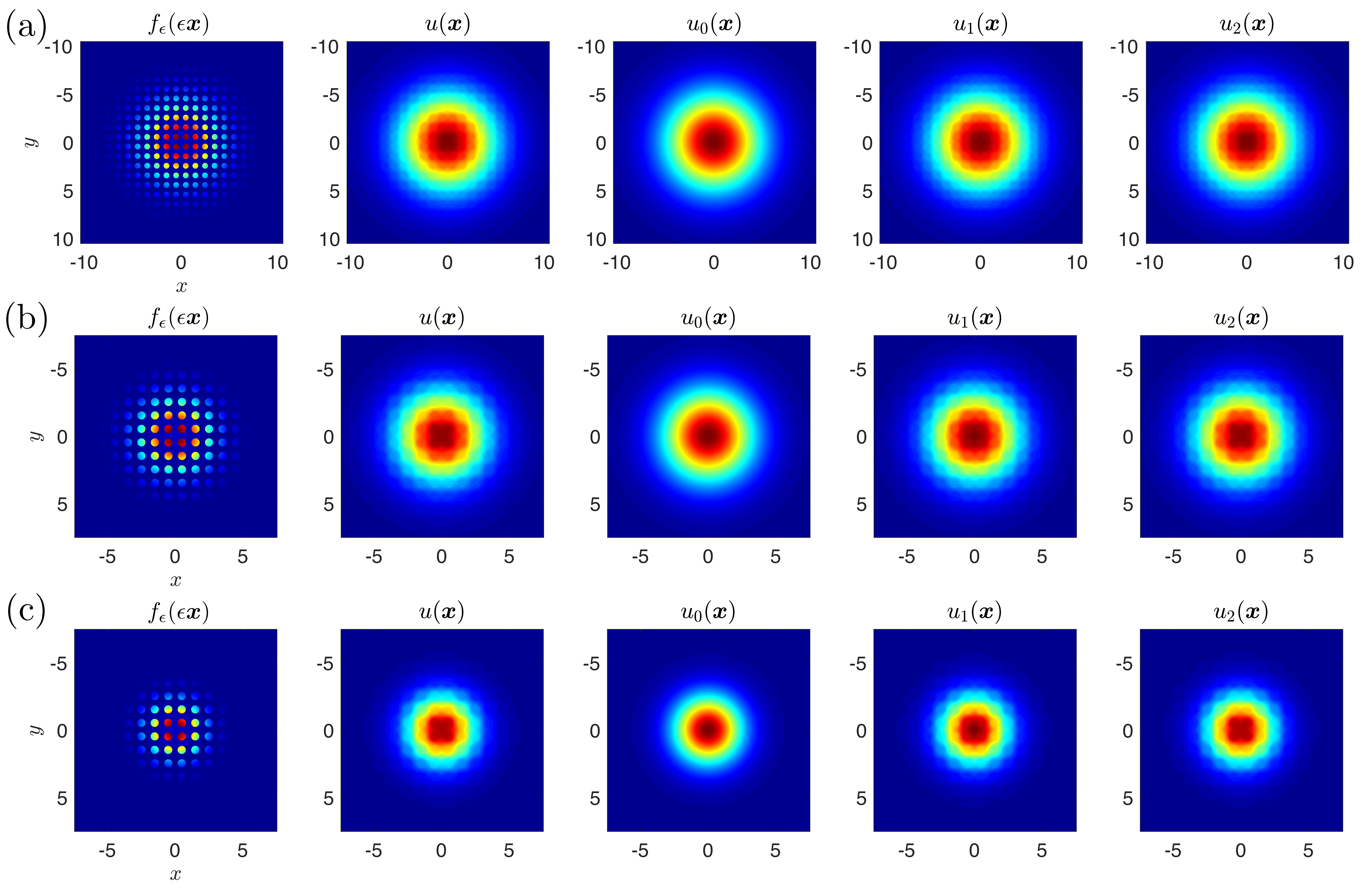}}\vspace*{-3mm}
\caption{Variation of the source term  {$f_\eps(\eps\bx)$}, ``exact'' solution $u(\bx)$, and its approximations $u_m(\bx)$ $(m\!=\!0,1,2)$ over: (a) domain $\mathcal{D}_{\tiny{\mbox{10}}}$ for $\eps=0.25$, (b) domain $\mathcal{D}_{\tiny{\mbox{7}}}$ for $\eps=0.375$, and (c) domain $\mathcal{D}_{\tiny{\mbox{7}}}$ for $\eps=0.5$. The color scale used for~{ $f_\epsilon$} is different from that used for~$u$ and~$u_m$.} \label{fig4} \vspace*{-0mm}
\end{figure}

\begin{figure}[h!] 
\centering{\includegraphics[width= 135 mm]{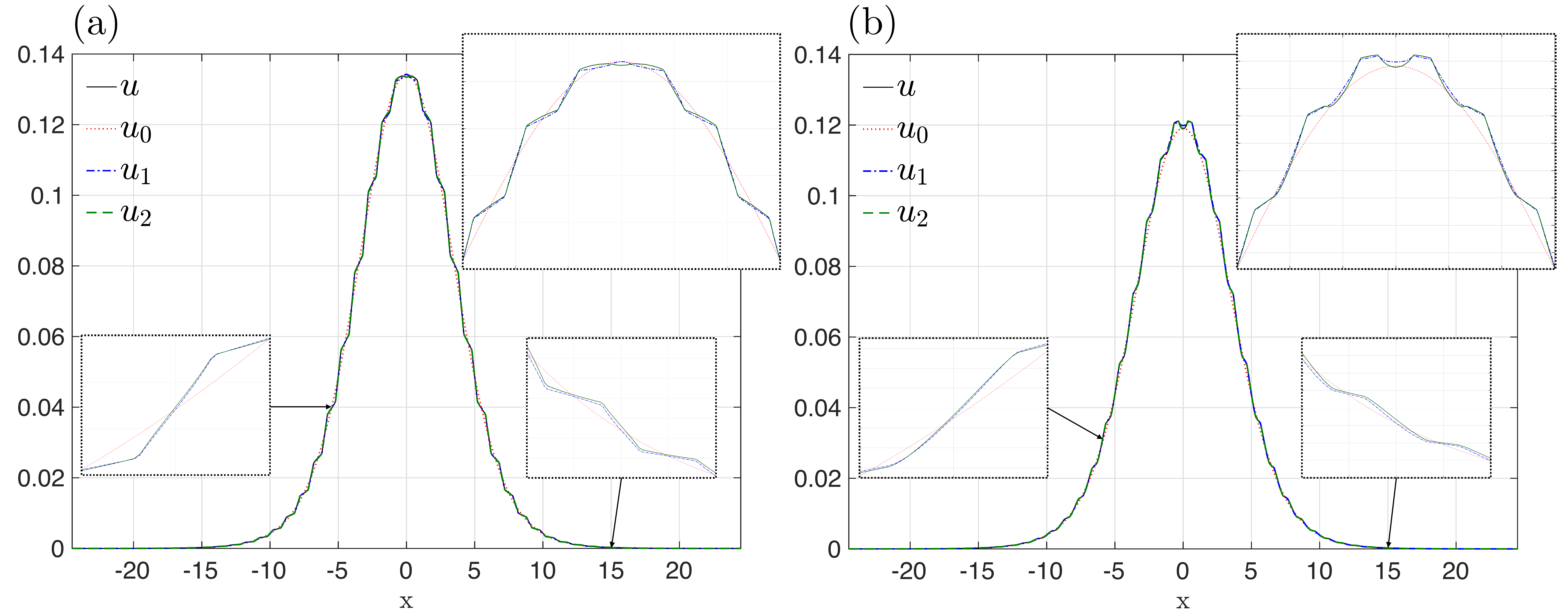}}\vspace*{-3mm}
\caption{``Exact'' solution, $u$, of the periodic medium ($\eps=0.25$) versus asymptotic approximations $u_m$ $(m\!=\!0,1,2)$ along line~$\bx=(x,y_0)$ for: (a) $y_0=0.5$ and (b) $y_0=1.75$. In each panel, the inserts provide magnified detail of the responses.} \label{fig2} \vspace*{-0mm}
\end{figure}

\begin{figure}[h!] 
\centering{\includegraphics[width= 135 mm]{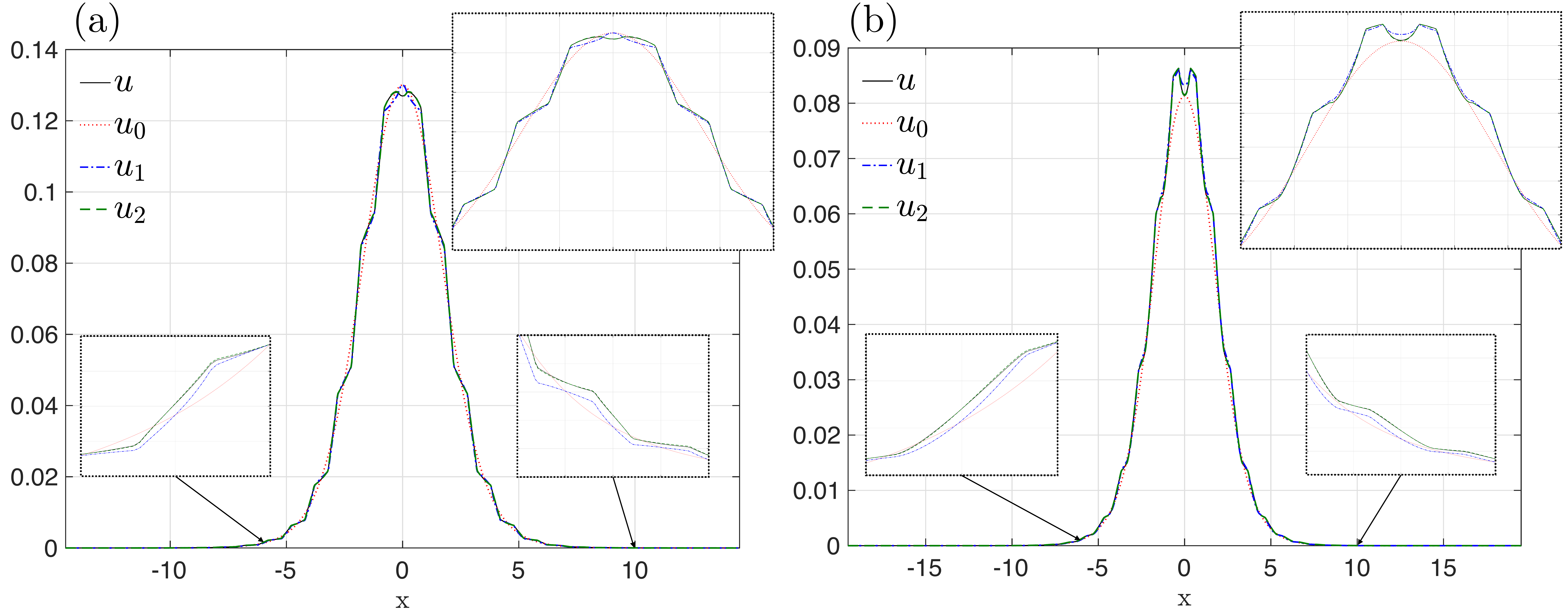}}\vspace*{-3mm}
\caption{``Exact'' solution, $u$, of the periodic medium ($\eps=0.5$) versus asymptotic approximations $u_m$ $(m\!=\!0,1,2)$ along line~$\bx=(x,y_0)$ for: (a) $y_0=0.5$ and (b) $y_0=1.75$. In each panel, the inserts provide magnified detail of the responses.} \label{fig3} \vspace*{-0mm}
\end{figure}

\begin{figure}[h!] 
\centering{\includegraphics[height= 60 mm]{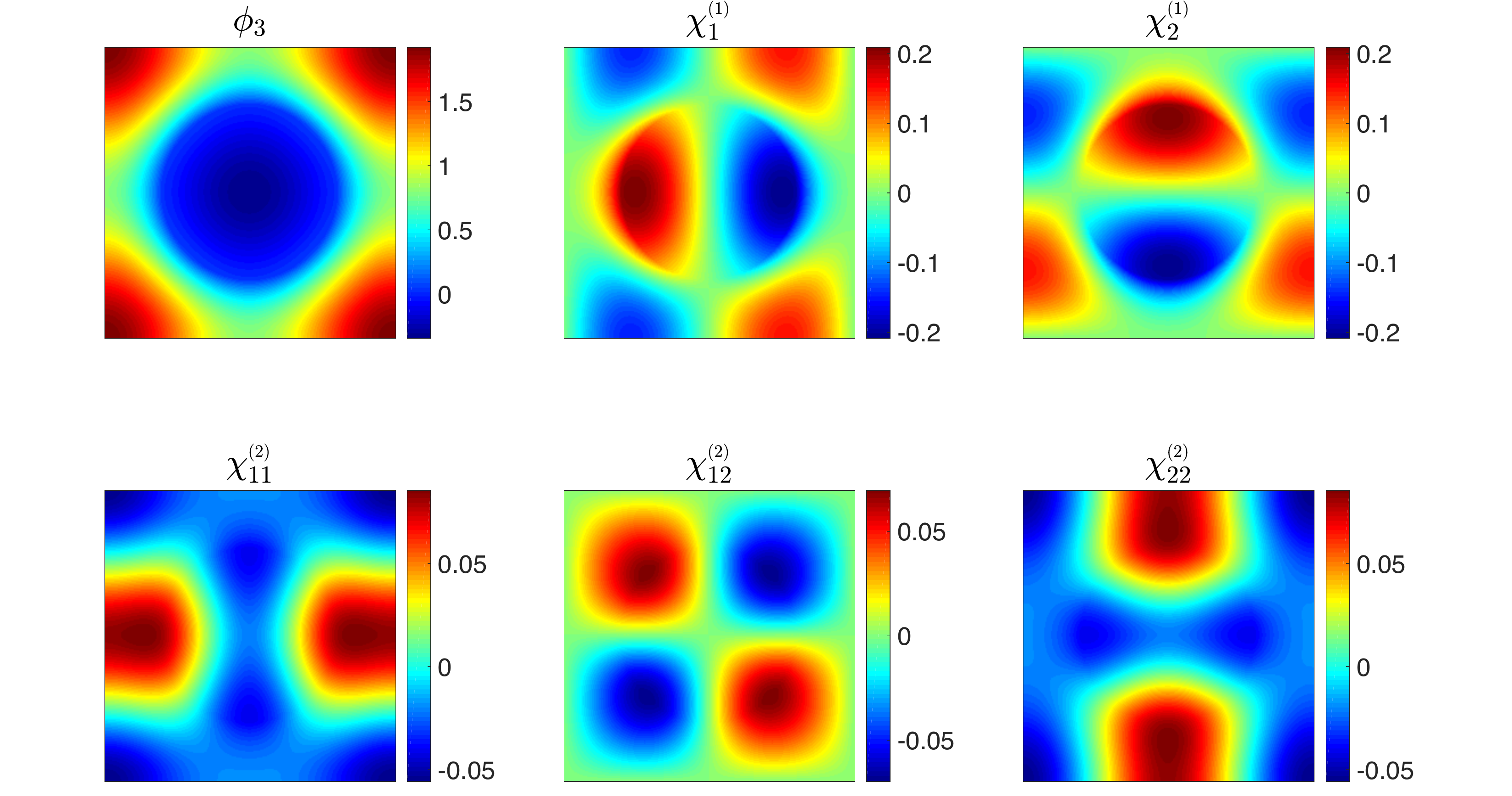}}\vspace*{-1mm}
\caption{Variation of eigenfunction $\phi_3(\boldsymbol{0},\bx)$ and components of the affiliated cell functions $\bchi^{\mbox{\tiny{(1)}}}(\boldsymbol{0},\bx)$ and $\bchi^{\mbox{\tiny{(2)}}}(\boldsymbol{0},\bx)$ (which are computed from $\phi_3(\boldsymbol{0},\bx)$) over the unit cell of periodicity~\eqref{wgcell}.} \label{fig7} \vspace*{-0mm}
\end{figure}

\begin{figure}[h!] 
\centering{\includegraphics[width= 135 mm]{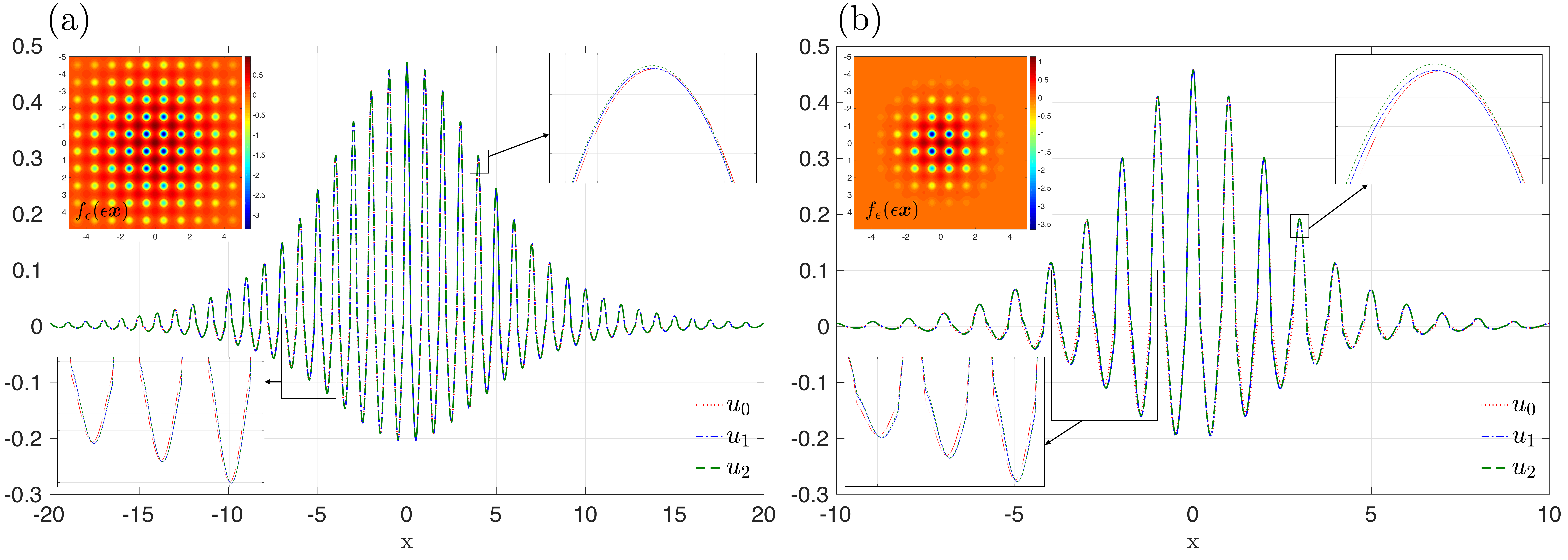}} \vspace*{-3 mm}
\caption{Asymptotic approximations $u_m$ $(m\!=\!0,1,2)$ for $p=3$, along line~$\bx=(x,0.5)$ for: (a) $\epsilon=0.25$ and (b) $\epsilon=0.5$. In each panel, the inserts provide magnified detail of the response approximations and the source  {$f_\eps(\epsilon \bx)$} that is plotted over $10\times 10$ unit cells.} \label{fig8} \vspace*{-0mm}
\end{figure}


\begin{figure}[h!] \vspace*{3mm}
\centering{\includegraphics[width= 80 mm]{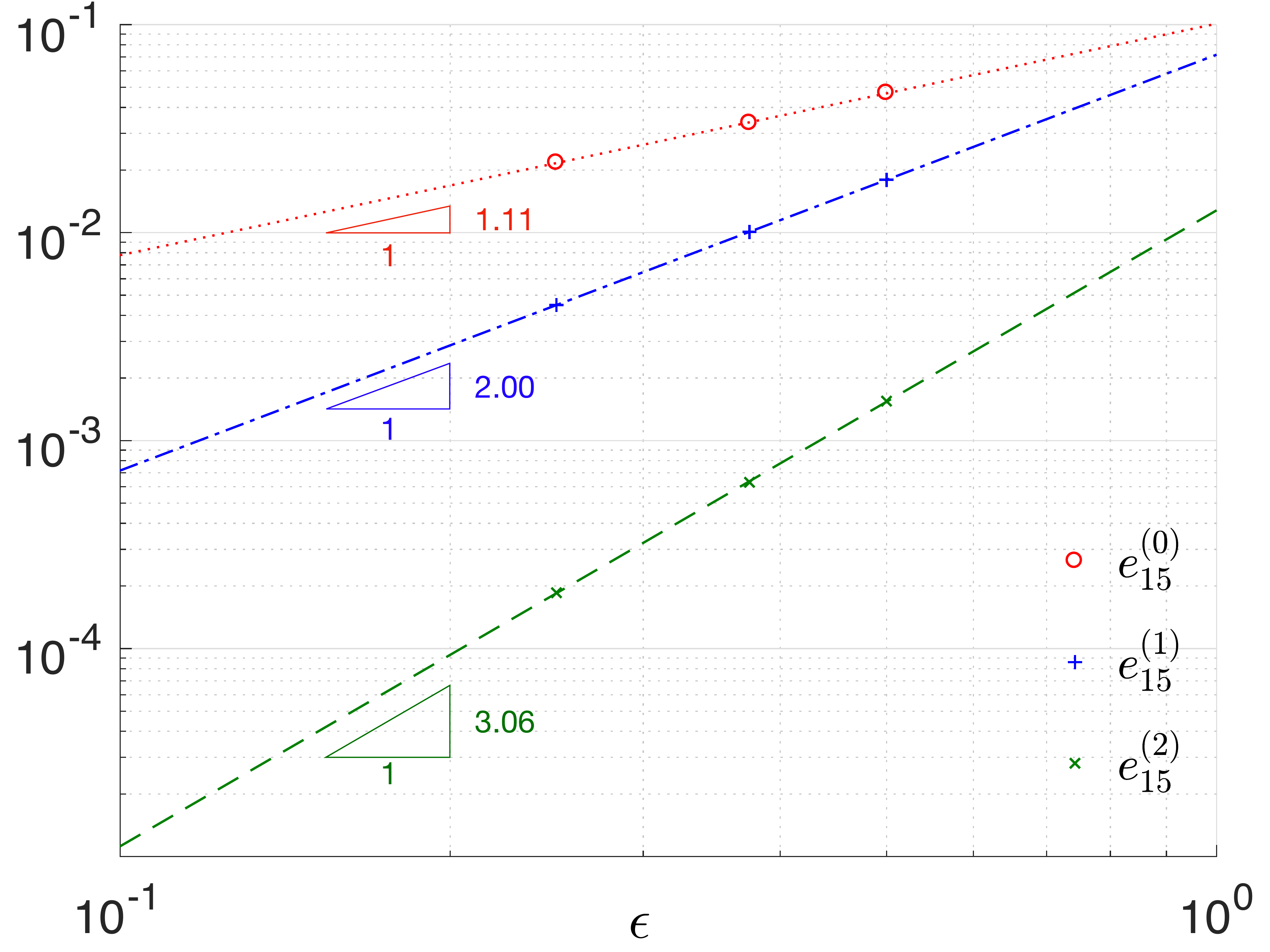}}\vspace*{-2mm}
\caption{Scaling of the relative approximation error~\eqref{relerr} ($m=0,1,2$) for $\eps\in\{0.25,0.375,0.5\}$. Dashed lines indicate the respective linear trends.} \label{fig5} \vspace*{0mm}
\end{figure}
\section{Summary}\label{Summary} 

{In this investigation we establish a convergent, second-order asymptotic model of the high-frequency, low-wavenumber wave motion in a highly-oscillating periodic medium driven by a source term. Within this framework, the driving frequency is further assumed to reside inside a band gap, while the source term is restricted to a class of functions which generate the long-wavelength motion. We first use the (Floquet-)Bloch transform to formulate an equivalent variational problem in the unit (Wigner-Seitz) cell. By investigating the source term's projection onto certain periodic functions, we establish a convergent, second-order homogenized model via asymptotic  expansion of (i) the nearest dispersion branch, (ii) germane (low-wavenumber) eigenfunction, and (iii) the source term. A set of numerical results is included to illustrate the utility of the homogenized model in representing, with high fidelity, both macroscopic and microscopic features of the exact wave motion, and to demonstrate the obtained convergence result.}

\bibliographystyle{plain}

\end{document}